\documentclass[12pt,twosides]{amsart}
\usepackage{amssymb,amsmath,amsthm, amscd, enumerate, mathrsfs}
\usepackage{graphicx, hhline}
\usepackage[all]{xy}
\usepackage{braket}
\usepackage[usenames]{color}
\usepackage{hyperref}
\usepackage{fancyhdr}
\usepackage[top=30truemm,bottom=30truemm,left=30truemm,right=30truemm]{geometry}
\usepackage{varioref}

\hypersetup{colorlinks=true}

\title{Finiteness of log abundant log canonical pairs in log minimal model program with scaling}
\author{Kenta Hashizume}
\date{2022/12/24}
\address{Department of Mathematics, Graduate School of Science, Kyoto University, Kyoto 606-8502, Japan}
\email{hkenta@math.kyoto-u.ac.jp}
\newtheorem{thm}{Theorem}[section]

\newtheorem{lem}[thm]{Lemma}
\newtheorem{cor}[thm]{Corollary}
\newtheorem{prop}[thm]{Proposition}

\theoremstyle{definition}
\newtheorem{defn}[thm]{Definition}
\newtheorem{rem}[thm]{Remark}

\newtheorem{exam}[thm]{Example}
\newtheorem*{ack}{Acknowledgments} 
\newtheorem*{divisor}{Divisors and maps} 
\newtheorem*{asym}{Asymptotic vanishing order} 
\newtheorem*{sing}{Singularities of pairs}

\newtheorem{step1}{Step}
\newtheorem{step2}{Step}
\newtheorem{step3}{Step}
\newtheorem{step4}{Step}

\newtheorem*{claim*}{Claim}
\begin{document}

\maketitle

\begin{abstract}
We study relations between the property of being log abundant for lc pairs and the termination of log MMP with scaling. 
We prove that any log MMP with scaling of an ample divisor starting with a projective  dlt pair contains only finitely many log abundant dlt pairs. 
We also discuss log MMP for slc pairs. 
\end{abstract}

\tableofcontents

\section{Introduction}\label{sec1}

Throughout this paper we will work over the complex number field $\mathbb{C}$ (see Section \ref{sec5} for generalization to varieties over an algebraically closed field of characteristic zero). 

In this paper, we study the minimal model theory for log abundant lc pairs, that is, lc pairs with the property of being log abundant (see Definition \ref{defn--abund}).  
In the minimal model theory, the properties of being abundant and being log abundant for lc pairs are closely related to the existence of good minimal models. 
These two properties coincide when we deal with klt pairs. 

It is well-known that the log MMP preserves the property of being abundant. 
Hence, for any lc pair $(X,\Delta)$, it is well-known that the divisor $K_{X}+\Delta$, which we call the log canonical divisor, is abundant if the minimal model theory holds for $(X,\Delta)$. 
In the klt case, the main result of Birkar--Cascini--Hacon--M\textsuperscript{c}Kernan \cite{bchm} and the ideas established by Lai \cite{lai} (see also \cite{kawamata-abund} and \cite[Theorem 4.3]{gongyolehmann}) show that the converse also holds. 

On the other hand, the log MMP does not necessarily preserve the property of being log abundant. 
For example, the situations of the main results of \cite{birkar-flip} and \cite{haconxu-lcc} (cf.~\cite{has-mmp}) are the cases where the log MMP preserves the property of being log abundant after finitely many steps (see Remark \ref{rem--recover}), whereas the situation of  \cite[Example 5.2]{birkarhu-arg} (see also \cite{ambrokollar}) is the case where the log MMP breaks the property of being log big, a special kind of the property of being log abundant. However, as an application of \cite[Theorem 4.11.5]{fujino-book} and the relation between the property of being abundant and the minimal model theory stated above, we know that the log canonical divisors of all lc pairs are log abundant if and only if the minimal model theory holds for all lc pairs. 

Because induction on the dimension of varieties works well under the assumption of the property of being log abundant, the abundance conjecture for log abundant lc pairs was thoroughly studied (\cite{reid-shokurov}, \cite{fukuda}, \cite{fujino-abund-logbig}, \cite{fujino-abund-saturation}, \cite{fujino-gongyo}, \cite{haconxu}) and proved by Hacon--Xu \cite{haconxu} (see also \cite{fujino-gongyo} by Fujino--Gongyo for projective case). 
In contrast, we know relations between the log MMP and the property of being log abundant only in special cases (\cite{birkar-flip}, \cite{haconxu-lcc}, \cite{has-mmp}, \cite{hashizumehu}). 

In this paper, we study how the property of being log abundant affects the termination of the log MMP. 
The following theorem is the main result of this paper.

\begin{thm}[=Theorem \ref{thm--abundantterminate}]\label{thm--intro-1}
Let $\pi\colon X\to Z$ be a projective morphism of normal quasi-projective varieties and let $(X,\Delta)$ be a $\mathbb{Q}$-factorial dlt pair such that $K_{X}+\Delta$ is pseudo-effective over $Z$. 
Let $A$ be a $\pi$-ample $\mathbb{R}$-divisor on $X$ such that $(X,\Delta+A)$ is lc and $K_{X}+\Delta+A$ is $\pi$-nef. 
Then there is an integer $l \geq 0$ satisfying the following: 
Let
$$(X_{0}:=X,\Delta_{0}:=\Delta) \dashrightarrow (X_{1},\Delta_{1}) \dashrightarrow\cdots \dashrightarrow (X_{i},\Delta_{i})\dashrightarrow \cdots$$
be a sequence of steps of a $(K_{X}+\Delta)$-MMP over $Z$ with scaling of $A$. 
If $(X_{i},\Delta_{i})$ is log abundant over $Z$ for some $i\geq l$, then the $(K_{X}+\Delta)$-MMP over $Z$ terminates with a good minimal model of $(X,\Delta)$ over $Z$. 
\end{thm}

\begin{cor}[=Corollary \ref{cor--finite-logabund}]\label{cor--intro-2}
Any log MMP with scaling of an ample divisor starting with a projective dlt pair contains only finitely many log abundant dlt pairs. 
\end{cor}

Although the case of klt pairs of Theorem \ref{thm--intro-1} is known by \cite{bchm} and \cite{lai} as mentioned above, Theorem \ref{thm--intro-1} can be applied to the situations of \cite{birkar-flip}, \cite{haconxu-lcc}, \cite{has-mmp}, and \cite{hashizumehu}, and in addition \cite[Theorem 1.3]{has-class} (see Remark \ref{rem--recover}). 
Thus, Theorem \ref{thm--intro-1} is the unification of those results as well as a relation between the log MMP and the property of being log abundant. 
Moreover, Theorem \ref{thm--intro-1} shows the existence of a good minimal model in the following new case. 

\begin{cor}[=Corollary \ref{cor--logabund-preserved-mmp}]\label{cor--intro-3}
Let $\pi\colon X\to Z$ be a projective morphism of normal quasi-projective varieties and let $(X,\Delta)$ be an lc pair.
Suppose that 
\begin{itemize}
\item
$(X,\Delta)$ is log abundant over $Z$ and $K_{X}+\Delta$ is pseudo-effective over $Z$, and
\item
the stable base locus of $K_{X}+\Delta$ over $Z$ {\rm (}\cite[Definition 3.5.1]{bchm}{\rm )} does not contain the image of any prime divisor $P$ over $X$ whose discrepancy $a(P,X,\Delta)$ is negative. 
\end{itemize}
Then $(X,\Delta)$ has a good minimal model over $Z$. 
\end{cor}

The second condition in Corollary \ref{cor--intro-3} holds, for example, if $(X,{\rm Supp}\,\Delta)$ is log smooth and there is an effective $\mathbb{R}$-divisor $E\sim_{\mathbb{R},Z}K_{X}+\Delta$ such that the support of $E$ does not contain any lc center of $(X,{\rm Supp}\,\Delta)$. 

Corollary \ref{cor--intro-3} has a generalization to slc pairs (Theorem \ref{thm--mmpslclogabund}).
For the generalization, the result by Ambro--Koll\'ar \cite{ambrokollar} plays a crucial role.  
In general, we cannot always run a log MMP for slc pairs. 
Nevertheless, \cite[Theorem 9]{ambrokollar} and Corollary \ref{cor--intro-3} enable us to run a log MMP for slc pairs terminating with a good minimal model in a similar  situation to Corollary \ref{cor--intro-3}. 
By applying Theorem \ref{thm--mmpslclogabund}, we prove the non-vanishing theorem for slc pairs in a very special case (Corollary \ref{cor--nonvanslc}). 
Furthermore, by applying Theorem \ref{thm--mmpslclogabund} and the gluing theory of Koll\'ar \cite[Section 5]{kollar-mmp}, we introduce a special case in which we can construct a projective semi-log canonical model (Theorem \ref{thm--mmpslclogbig}). 

We now explain the ideas of the proofs of Theorem \ref{thm--intro-1} and Corollary \ref{cor--intro-2}. 
We will first prove Corollary \ref{cor--intro-2} (in fact, we will prove Theorem \ref{thm--main-logabundant}, the lc and relative setting of Corollary \ref{cor--intro-2}). 
Suppose by contradiction that there exist 
a projective dlt pair $(X,\Delta)$ and 
an infinite sequence of steps of a $(K_{X}+\Delta)$-MMP with scaling of an ample divisor which contains infinitely many log abundant dlt pairs. 
In particular, the log MMP does not terminate, hence $K_{X}+\Delta$ is pseudo-effective by \cite{bchm}. 
We will prove the existence of a log minimal model of $(X,\Delta)$ to get a contradiction with \cite[Theorem 4.1]{birkar-flip}. 
By the special termination (\cite{fujino-sp-ter}), replacing $(X,\Delta)$, we may assume that $(X,\Delta)$ is log abundant and the restrictions of $K_{X}+\Delta$ to all lc centers of $(X,\Delta)$ are nef. 
Despite of the two properties, it is still difficult to prove the existence of a log minimal model of $(X,\Delta)$, as shown in the following example. 

\begin{exam}[{\cite[Example 5.2]{birkarhu-arg}}]\label{exam--intro}
Fix a smooth projective variety $S$ with $\kappa(S,K_{S})\geq0$. 
With notation as in \cite[Example 5.2]{birkarhu-arg}, we can pick $\epsilon\in \mathbb{R}_{>0}$ and an ample Cartier divisor $H$ so that $K_{S}+\epsilon H$ is ample. 
Then the projective dlt pair $(X,B)$ in \cite[Example 5.2]{birkarhu-arg} is log abundant and the restrictions of $K_{X}+B$ to all lc centers of $(X,B)$ are nef. 
Moreover, if $(X,B)$ has a log minimal model, then so does $(S,0)$. 
From this, the existence of log minimal models for projective dlt pairs with the above two properties is more difficult than that for smooth projective varieties with non-negative Kodaira dimension. 
\end{exam}

Hence, the two properties stated above are insufficient for our proof. 
One more key property is that the non-nef locus of $K_{X}+\Delta$ is disjoint from any lc center of $(X,\Delta)$. 
The reduction with the special termination (\cite{fujino-sp-ter}) allows us to assume these three properties, and therefore Corollary \ref{cor--intro-2} can be reduced to the following statement, which derives a contradiction, as discussed above. 

\begin{thm}[=Theorem \ref{thm--ind-1}]\label{thm--intro-key}
Let $(X,\Delta)$ be a projective dlt pair. 
Suppose that
\begin{itemize}
\item
$K_{X}+\Delta$ is pseudo-effective and abundant,
\item for any lc center $S$ of $(X,\Delta)$, the restriction $(K_{X}+\Delta)|_{S}$ is nef, and
\item
$\sigma_{P}(K_{X}+\Delta)=0$ for every prime divisor $P$ over $X$ such that $a(P,X,\Delta)< 0$ and $c_{X}(P)$ intersects an lc center of $(X,\Delta)$, where $\sigma_{P}(\,\cdot\,)$ is the asymptotic vanishing order of $P$ defined in Definition \ref{defn--asy-van-ord}. 
\end{itemize}
Then $(X,\Delta)$ has a log minimal model. 
\end{thm}

The third condition of Theorem \ref{thm--intro-key} has not been mentioned in the  papers related to Theorem \ref{thm--intro-1} (\cite{birkar-flip}, \cite{haconxu-lcc}, \cite{has-mmp}, \cite{hashizumehu}, \cite{has-class}). 
This condition enables us to unify the hypotheses of the results in \cite{birkar-flip}, \cite{haconxu-lcc}, \cite{has-mmp}, \cite{hashizumehu}, and \cite[Section 3]{has-class} into one situation. 
We will prove Theorem \ref{thm--intro-key} by running a special kind of log MMP used in \cite{birkar-flip}, \cite{haconxu-lcc}, \cite{has-mmp}, \cite{hashizumehu}, and \cite{has-class}. 
The third condition of Theorem \ref{thm--intro-key} plays a crucial role for the termination of the special kind of log MMP. 

We will prove Theorem \ref{thm--intro-1} by using Theorem \ref{thm--main-logabundant}. 
Fix a projective morphism $X \to Z$, a $\mathbb{Q}$-factorial dlt pair $(X,\Delta)$, and a relatively ample $\mathbb{R}$-divisor $A$ on $X$ as in Theorem \ref{thm--intro-1}.   
By using the finiteness of models (\cite[Theorem E]{bchm}), we will find an integer $l$ satisfying the following: For any $(K_{X}+\Delta)$-MMP over $Z$ with scaling of $A$, if the $(K_{X}+\Delta)$-MMP contains a log abundant dlt pair after the $l$-th step, then the $(K_{X}+\Delta)$-MMP terminates with a nef and log abundant dlt pair, or the $(K_{X}+\Delta)$-MMP contains infinitely many log abundant dlt pairs. 
Then Theorem \ref{thm--main-logabundant} shows that the latter case cannot happen. 
Furthermore, the abundance theorem for nef and log abundant lc pairs (\cite{fujino-abund-saturation}, \cite{fujino-gongyo}, \cite{haconxu}) shows that the resulting dlt pair of the $(K_{X}+\Delta)$-MMP is a good minimal model of $(X,\Delta)$ over $Z$. 
Therefore, the integer $l$ satisfies the property of Theorem \ref{thm--intro-1}. 

We would like to remark that the circle of ideas explained above is still effective when studying the MMP for generalized pairs. 
The notion of generalized pairs was introduced by Birkar--Zhang \cite{bz} to tackle the effective Iitaka fibration problem. 
The theory of the MMP for generalized pairs is rapidly being developed (\cite{hanli}, \cite{hacon-liu}) and currently a major topic in birational geometry. 
After this paper was announced, the author studied the termination of MMP for generalized lc pairs (\cite[Theorem 1.1]{has-iitakafibration}, \cite[Theorem 3.17]{has-iitakafibration}). 
The first result \cite[Theorem 1.1]{has-iitakafibration} gave a development of the boundedness of various invariants for dlt pairs polarized by nef and log big divisors (\cite{has-iitakafibration}, \cite{has-lc-trivialfib}). 
The second result \cite[Theorem 3.17]{has-iitakafibration} was applied by Lazi\'c--Tsakanikas--Jiang \cite[Appendix]{ltj} to study the existence of Mori fiber spaces for generalized lc pairs whose generalized log canonical divisor is not pseudo-effective. 
Liu--Xie \cite{liu-xie} studied a generalization of Theorem \ref{thm--main-logabundant} to generalized lc pairs. 
These developments are applications of the ideas in this paper. 

We conclude the introduction with an application of Theorem \ref{thm--intro-key} (see Corollary \ref{cor--appli-2} for an application for higher-dimensional lc pairs). 

\begin{cor}[cf.~Corollary \ref{cor--appli}]\label{thm--intro-appli}
Let $(X,\Delta)$ be a projective lc pair of dimension six such that $\llcorner \Delta \lrcorner =0$ and the invariant Iitaka dimension $\kappa_{\iota}(X,K_{X}+\Delta)$ is greater than or equal to three. 
Then $(X,\Delta)$ has a log minimal model. 
In particular, all projective lc pairs $(X,0)$ satisfying ${\rm dim}\,X=6$ and $\kappa(X,K_{X})\geq 3$ have log minimal models.   
\end{cor}

The contents of this paper are as follows: 
In Section \ref{sec2}, we collect definitions, notations and basic results. 
In Section \ref{sec3}, we prove the main results and corollaries. 
In Section \ref{sec4}, we discuss a generalization of Corollary \ref{cor--intro-3} to slc pairs. 
In Section \ref{sec5}, we give a remark on generalizations of results of this paper to varieties over an algebraically closed field of characteristic zero. 

\begin{ack}
The author was partially supported by JSPS KAKENHI Grant Number JP19J00046, JP22K13887.
Part of the work was done while the author was visiting University of Cambridge in January 2020. The author thanks Professor Caucher Birkar for his hospitality. 
He thanks Professors Caucher Birkar and Yoshinori Gongyo for discussions. 
He thanks Professor J\'anos Koll\'ar for answering questions. 
He thanks the referee(s) for long lists of suggestions which improve the paper considerably. 
\end{ack}

\section{Preliminaries}\label{sec2}
In this section, we collect definitions and some results. 

\subsection{Divisors, morphisms, and singularities of pairs}
We collect notations and definitions on divisors, morphisms and singularities of pairs. We also show some lemmas. 

\begin{divisor}
Let $\pi \colon X \to Z$ be a projective morphism from a normal variety  to a variety. 
We will use the standard definitions of $\pi$-nef $\mathbb{R}$-divisor, $\pi$-ample $\mathbb{R}$-divisor, $\pi$-semi-ample $\mathbb{R}$-divisor, and $\pi$-pseudo-effective $\mathbb{R}$-Cartier divisor. 
All $\pi$-big $\mathbb{R}$-divisors in this paper are assumed to be $\mathbb{R}$-Cartier.  

For an $\mathbb{R}$-divisor $D=\sum_{i} a_{i}D_{i}$ on a variety $X$, where $D_{i}$ are distinct prime divisors, an $\mathbb{R}$-divisor $\llcorner D \lrcorner$ is defined  by $\llcorner D \lrcorner= \sum_{i} \llcorner a_{i} \lrcorner D_{i}$, where $\llcorner a_{i} \lrcorner$ is the largest integer not exceeding $a_{i}$. 
For a morphism $f\colon X\to Y$ of varieties and an $\mathbb{R}$-Cartier divisor $D'$ on $Y$, we sometimes denote $f^{*}D'$ by $D'|_{X}$. 
For any prime divisor $P$ over $X$, the image of $P$ on $X$ is denoted by $c_{X}(P)$. 

A projective morphism $f\colon X\to Y$ of varieties is called a {\em contraction} if $f_{*}\mathcal{O}_{X}\simeq \mathcal{O}_{Y}$. 
For a variety $X$ and an $\mathbb{R}$-divisor $D$ on $X$, a {\em log resolution of} $(X,D)$ is a projective birational morphism $f\colon Y\to X$ from a smooth variety $Y$ such that the exceptional locus ${\rm Ex}(f)$ of $f$ is pure codimension one and ${\rm Ex}(f)\cup {\rm Supp}\,f_{*}^{-1}D$ is a simple normal crossing divisor. 

A birational map $\phi\colon X \dashrightarrow X'$ of varieties is called a {\em birational contraction} if $\phi^{-1}$ does not contract any divisor. 
We say that $\phi$ is {\em small} if $\phi$ and $\phi^{-1}$ are birational contractions.  
\end{divisor}

\begin{asym} 
We define asymptotic vanishing order for $\mathbb{R}$-Cartier divisors, and we collect some basic properties.  

\begin{defn}[Asymptotic vanishing order]\label{defn--asy-van-ord}
Let $X$ be a normal projective variety, and let $D$ be a pseudo-effective $\mathbb{R}$-Cartier divisor on $X$. 
Let $P$ be a prime divisor over $X$. 
We define the {\em asymptotic vanishing order} of $P$ with respect to $D$, denoted by $\sigma_{P}(D)$, as follows: 
We take a projective birational morphism $f\colon Y \to X$ so that $P$ appears as a prime divisor on $Y$. 
When $D$ is big, we define $\sigma_{P}(D)$ by
$$\sigma_{P}(D)={\rm inf}\!\set{{\rm coeff}_{P}(D')|f^{*}D\sim_{\mathbb{R}}D'\geq0}.$$
When $D$ is not necessarily big, $\sigma_{P}(D)$ is defined by
$$\sigma_{P}(D)=\underset{\epsilon\to 0+}{\rm lim}\sigma_{P}(D+\epsilon A)$$
for an ample $\mathbb{R}$-divisor $A$ on $X$. 

By definition, $\sigma_{P}(D)$ is independent of $f\colon Y\to X$, and $\sigma_{P}(D)\geq0$. 
Moreover, $\sigma_{P}(D)$ does not depend on $A$ (\cite[III, 1.5 (2) Lemma]{nakayama}) or the numerical equivalence class of $D$. 
We can easily check $\sigma_{P}(D)={\rm sup}\!\set{\sigma_{P}(D+H)|H \text{ is ample}}$. 
\end{defn}

\begin{defn}[Nakayama--Zariski decomposition] 
For $X$ and $D$ as in Definition \ref{defn--asy-van-ord}, the {\em negative part of the Nakayama--Zariski decomposition} of $D$, denoted by $N_{\sigma}(D)$, is defined by
$$N_{\sigma}(D)=\sum_{\substack {P:{\rm \,prime\,divisor}\\{\rm on\,}X}}\sigma_{P}(D)P.$$
We call the divisor $D-N_{\sigma}(D)$ the {\em positive part of the Nakayama--Zariski decomposition} of $D$. 
Note that $N_{\sigma}(D)$ and $D-N_{\sigma}(D)$ are not necessarily $\mathbb{R}$-Cartier in this paper.
\end{defn}

When $X$ is smooth, the definition of $N_{\sigma}(D)$ coincides with \cite[III, 1.12 Definition]{nakayama}. 
When $X$ is not necessarily smooth, $N_{\sigma}(D)$ in the definition can be written as $g_{*}N_{\sigma}(g^{*}D)$ where $g\colon W\to X$ is a resolution and $N_{\sigma}(g^{*}D)$ is defined in \cite[III, 1.12 Definition]{nakayama}. 
Hence $N_{\sigma}(D)$ is well-defined as an $\mathbb{R}$-divisor.  
For the Nakayama--Zariski decomposition in the case of singular varieties, see also \cite[Section 4]{bhzariski}. 

We write down some properties of the asymptotic vanishing order and the negative part of the Nakayama--Zariski decomposition. 

\begin{rem}\label{rem--asy-van-ord-1}
Let $X$ be a normal projective variety, $D$ a pseudo-effective $\mathbb{R}$-Cartier divisor on $X$, and let $P$ be a prime divisor over $X$.

\begin{enumerate}[(1)]
\item(Lower convexity, {see \cite[III, 1.6 Definition]{nakayama}}). \label{rem--asy-van-ord-1-(1)}
For any pseudo-effective $\mathbb{R}$-Cartier divisor $D'$ on $X$, we have
$$\sigma_{P}(D+D')\leq \sigma_{P}(D)+\sigma_{P}(D').$$
\item({\cite[III, 1.7 (2) Lemma]{nakayama}}). \label{rem--asy-van-ord-1-(2)}
For any pseudo-effective $\mathbb{R}$-Cartier divisor $G$ on $X$, we have
$$\underset{\epsilon \to 0+}{\rm lim}\sigma_{P}(D+\epsilon G)=\sigma_{P}(D).$$
\item({\cite[III, 5.14 Lemma]{nakayama}}). \label{rem--asy-van-ord-1-(3)}
Let $f\colon Y \to X$ be a projective birational morphism from a normal variety $Y$ and let $E$ be an effective $f$-exceptional $\mathbb{R}$-divisor on $Y$. 
If $f^{*}D+E$ is $\mathbb{R}$-Cartier (in this situation $E$ is also $\mathbb{R}$-Cartier), then 
$$\sigma_{P}(f^{*}D+E)=\sigma_{P}(f^{*}D)+{\rm ord}_{P}(E).$$
\item(cf.~{\cite[III, 5.16 Theorem]{nakayama}}). \label{rem--asy-van-ord-1-(4)}
Let $f\colon Y \to X$ be a projective birational morphism from a normal variety $Y$ and let $D''$ be an $\mathbb{R}$-Cartier divisor on $X$. 
If $D''\leq N_{\sigma}(D)$, then 
$$N_{\sigma}(f^{*}D)=f^{*}D''+F$$
for an effective $\mathbb{R}$-divisor $F$ on $Y$ such that $D''+f_{*}F=N_{\sigma}(D)$. 
\end{enumerate}
\end{rem}
We outline the proof of (\ref{rem--asy-van-ord-1-(4)}). 
Let $g \colon W \to Y$ be a resolution of $Y$. 
Since 
$$N_{\sigma}(f^{*}D)-f^{*}D''=g_{*}N_{\sigma}(g^{*}f^{*}D)-f^{*}D''=g_{*}(N_{\sigma}(g^{*}f^{*}D)-g^{*}f^{*}D''),$$
 we may replace $f$ by $f \circ g$, and therefore we may assume that $Y$ is smooth. 
We can write $N_{\sigma}(f^{*}D)-f^{*}D''=G_{+}-G_{-}$, where $G_{+}\geq 0$ and $G_{-} \geq 0$ have no common components. 
Then $G_{-}$ is $f$-exceptional since $D''\leq N_{\sigma}(D)$. 
On the other hand, we have 
$$G_{+}-G_{-}\sim_{\mathbb{R},X} N_{\sigma}(f^{*}D) \sim_{\mathbb{R},X}N_{\sigma}(f^{*}D)-f^{*}D=-P_{\sigma}(f^{*}D).$$
Hence $-(G_{+}-G_{-})$ is the limit of movable divisors over $X$. 
This implies $G_{-}=0$ (see, for example, \cite[Lemma 3.3]{birkar-flip}). 
So $F:=G_{+}$ is the desired divisor. 

In Remark \ref{rem--mmp-zariskidecom}, we will discuss further properties of the asymptotic vanishing order. 
\end{asym}

\begin{lem}\label{lem--negativesupport}
Let $X$ be a normal projective variety. 
Let $D$ and $D'$ be pseudo-effective $\mathbb{R}$-Cartier divisors on $X$. 
Then there exists a real number $t_{0}>0$ such that the support of $N_{\sigma}(D+tD')$ does not depend on $t$ for any $t\in(0,t_{0}]$. 
\end{lem}

\begin{proof}
By Remark \ref{rem--asy-van-ord-1} (\ref{rem--asy-van-ord-1-(1)}), the components of $N_{\sigma}(D+tD')$ are those of $N_{\sigma}(D)+N_{\sigma}(D')$ for any $t\in (0,1]$. 
Let $D_{1},\cdots , D_{k}$ be all components of $N_{\sigma}(D)+N_{\sigma}(D')$, and we consider the set 
$$J_{i}:=\{s\in(0,1]\,|\,\sigma_{D_{i}}(D+sD')=0\}$$
 for all $1\leq i\leq k$. 
 
Fix $i$. 
When $J_{i}$ is not empty and ${\rm inf}\,J_{i}=0$, we pick an element $s_{i}$ of $J_{i}$. 
In this case, for every $s\in(0,s_{i}]$ we can find real numbers $s', s''\in(0,s_{i}]$ such that $s\in [s',s'']$ and $\sigma_{D_{i}}(D+s'D')=\sigma_{D_{i}}(D+s''D')=0$.  
Hence we have $\sigma_{D_{i}}(D+sD')=0$ for all $s\in(0,s_{i}]$ by Remark \ref{rem--asy-van-ord-1} (\ref{rem--asy-van-ord-1-(1)}). 
We put $s_{i} =1$ when $J_{i}$ is empty, and we put $s_{i}=\frac{1}{2}{\rm inf}\,J_{i}$ when $J_{i}$ is not empty and ${\rm inf}\,J_{i}>0$, then $\sigma_{D_{i}}(D+sD')>0$ for all $s\in(0,s_{i}]$. 
In this way, we can find $s_{i}\in (0,1]$ such that $\sigma_{D_{i}}(D+sD')=0$ for all $s\in(0,s_{i}]$ or $\sigma_{D_{i}}(D+sD')>0$ for all $s\in(0,s_{i}]$. 

We set $t_{0}={\rm min}\{s_{i}\}_{1\leq i \leq k}$. 
By definition of $D_{1},\cdots , D_{k}$, we have 
$$N_{\sigma}(D+tD')=\sum_{i=1}^{k}\sigma_{D_{i}}(D+tD')D_{i}.$$
The construction of $t_{0}$ shows that ${\rm Supp}\,N_{\sigma}(D+tD')$ does not depend on $t\in (0,t_{0}]$. 
\end{proof}

\begin{sing}
A {\em pair} $(X,\Delta)$ consists of a normal variety $X$ and an effective $\mathbb{R}$-divisor $\Delta$ on $X$ such that $K_{X}+\Delta$ is $\mathbb{R}$-Cartier. 

Let $(X,\Delta)$ be a pair, and let $P$ be a prime divisor over $X$. 
We denote by $a(P,X,\Delta)$ the discrepancy of $P$ with respect to $(X,\Delta)$. 
We will freely use the standard definitions of Kawamata log terminal (klt, for short) pair, log canonical (lc, for short) pair and divisorial log terminal (dlt, for short) pair as in \cite{kollar-mori} or \cite{bchm}. 
In \cite{kollar-mori}, pairs and classes of singularities are defined in the framework of $\mathbb{Q}$-divisors, but we can similarly define those classes of singularities for pairs consisting of a normal variety and an effective $\mathbb{R}$-divisor. 
When $(X,\Delta)$ is an lc pair, an {\em lc center} of $(X,\Delta)$ is $c_{X}(P)$ for a prime divisor $P$ over $X$ such that $a(P,X,\Delta)=-1$. 
We will freely use basic properties of lc centers of dlt pairs proved in \cite[Proposition 3.9.2]{fujino-what-log-ter} and \cite[Theorem 4.16]{kollar-mmp}.  
\end{sing}

\begin{lem}\label{lem--adjunction}
Let $(X,\Delta)$ be a dlt pair and $S$ an lc center of $(X,\Delta)$. 
Fix a log resolution $f\colon Y\to X$ of $(X,\Delta)$ which is an isomorphism over an open subset of $X$ intersecting $S$. 
Let $T\subset Y$ be a subvariety such that the restriction of $f$ to $T$ induces a birational morphism $f|_{T}\colon T \to S$, and let $M$ be an $f$-exceptional divisor on $Y$ such that all components $P$ of $M$ satisfy $a(P,X,\Delta)>0$. 
Then all components of $M|_{T}$ are exceptional over $S$. 
\end{lem}

\begin{proof}
We may write $K_{Y}+\Gamma=f^{*}(K_{X}+\Delta)+E$ such that $\Gamma \geq0$ and $E\geq0$ have no common components and all components of $M$ are components of $E$. 
Since $(X,\Delta)$ is dlt, there are components $D_{1},\cdots ,D_{k}$ of $\llcorner \Delta \lrcorner$ such that $S$ is an irreducible component of $D_{1}\cap\cdots \cap D_{k}$ (see \cite[Proposition 3.9.2]{fujino-what-log-ter} or \cite[Theorem 4.16]{kollar-mmp}). 
Let $D'_{1},\cdots ,D'_{k}$ be the birational transforms of $D_{1},\cdots , D_{k}$ on $Y$, respectively. 
By definitions of $\Gamma$ and $T$, the divisors $D'_{1},\cdots,D'_{k}$ are components of $\llcorner \Gamma \lrcorner$ and $T$ is an irreducible component of $D'_{1}\cap\cdots \cap D'_{k}$. 
We put $G=\Gamma-\sum_{i=1}^{k}D'_{i}$. We apply \cite[Proof of Lemma 2.4]{hashizumehu} to check that $E|_{T}$ is exceptional over $S$. Suppose that $G|_{T}$ and $E|_{T}$ have a common component $Q$. Then $Q\subset G'\cap E'\cap D'_{1}\cap\cdots \cap D'_{k}$ for some components $G'$ of $G$ and $E'$ of $E$. 
The left hand side has codimension $k+1$ in $Y$ but the right hand side has codimension $k+2$ in $Y$, which is impossible. Therefore, $G|_{T}$ and $E|_{T}$ have no common components. We put $f_{T}=f|_{T}\colon T\to S$ and define $\Delta_{S}$ on $S$ by adjunction $K_{S}+\Delta_{S}=(K_{X}+\Delta)|_{S}$. Then $\Delta_{S}\geq0$ and $\Delta_{S}=f_{T*}G|_{T}-f_{T*}E|_{T}$. Since $f_{T}\colon T\to S$ is birational, $f_{T*}G|_{T}$ and $f_{T*}E|_{T}$ have no common components. From these facts, it follows that $f_{T*}E|_{T}=0$. Therefore $E|_{T}$ is exceptional over $S$, so is $M|_{T}$.   
\end{proof}

\begin{lem}\label{lem--discre-relation}
Let $(X,\Delta)$ and $(X',\Delta')$ be projective dlt pairs, $S$ and $S'$ lc centers of $(X,\Delta)$ and $(X',\Delta')$ respectively, and let $f\colon X\dashrightarrow X'$ be a birational map such that $f$ is an isomorphism on an open set intersecting $S$ and the restriction of $f$ to $S$ induces a birational map $f|_{S}\colon S\dashrightarrow S'$. 
Suppose that $K_{X}+\Delta$ is pseudo-effective. 
Suppose in addition that 
\begin{itemize}
\item
$a(D',X',\Delta')\leq a(D',X,\Delta)$ for all prime divisors $D'$ on $X'$, and
\item
$\sigma_{P}(K_{X}+\Delta)=0$ for all prime divisors $P$ over $X$ such that $a(P,X,\Delta)< 0$ and $c_{X}(P)$ intersects $S$, where $\sigma_{P}(\,\cdot\,)$ is the asymptotic vanishing order of $P$ defined in Definition \ref{defn--asy-van-ord}. 
\end{itemize}
Let $(S,\Delta_{S})$ and $(S',\Delta_{S'})$ be projective dlt pairs, where $\Delta_{S}$ and $\Delta_{S'}$ are constructed by adjunctions $K_{S}+\Delta_{S}=(K_{X}+\Delta)|_{S}$ and $K_{S'}+\Delta_{S'}=(K_{X'}+\Delta')|_{S'}$, respectively. 
Then the inequality 
$$a(Q,S',\Delta_{S'})\leq a(Q,S,\Delta_{S})$$
 holds for all prime divisors $Q$ on $S'$. 
\end{lem}

\begin{proof}
By taking a log resolution $g\colon Y\to X$ of $(X,\Delta)$ which resolves the indeterminacy of $f\colon X\dashrightarrow X'$, we may find a common log resolution $g$ and $g' \colon Y \to {X'}$ of the map $(X,\Delta)\dashrightarrow (X',\Delta')$ and a subvariety $T\subset Y$ birational to $S$ and $S'$ such that the induced morphisms $g|_{T}\colon T\to S$ and $g'|_{T}\colon T\to S'$ form a common resolution of $f|_{S}\colon S\dashrightarrow S'$. 
We may write 
\begin{equation*}\tag{$*$}\label{proof-lem--discre-relation-*}
g^{*}(K_{X}+\Delta)=g'^{*}(K_{X'}+\Delta')+M-N
\end{equation*}
such that $M\geq0$ and $N\geq0$ have no common components. 
By the first condition of Lemma \ref{lem--discre-relation}, $M$ is $g'$-exceptional. 
Since the morphism $g'\colon Y \to X'$ is an isomorphism over an open subset which contains general points of $S'$, we see that $T\not\subset {\rm Supp}\,M$. 
We also see that $T\not\subset {\rm Supp}\,N$. 
Indeed, we can find a prime divisor $P'$ over $X$ such that $c_{Y}(P')=T$ and $a(P',X,\Delta)=-1$. 
If $T\subset {\rm Supp}\,N$, then the fact $T\not\subset {\rm Supp}\,M$ shows that $a(P',X',\Delta')=a(P',X,\Delta)+{\rm ord}_{P'}(M)-{\rm ord}_{P'}(N)<-1$, which contradicts that $(X',\Delta')$ is dlt. 
Thus, we see that $T\not\subset {\rm Supp}\,N$. 
Restricting (\ref{proof-lem--discre-relation-*}) to $T$, we have 
\begin{equation*}\tag{$**$}\label{proof-lem--discre-relation-**}
g|_{T}^{*}(K_{S}+\Delta_{S})=g'|_{T}^{*}(K_{S'}+\Delta_{S'})+M|_{T}-N|_{T},
\end{equation*}
where $M|_{T}$ and $N|_{T}$ are effective. 

We decompose $M=M_{0}+M_{1}$ for $M_{0}\geq 0$ and $M_{1}\geq 0$ such that any component of $M_{0}$ does not intersect $T$ and all components of $M_{1}$ intersect $T$.  
By (\ref{proof-lem--discre-relation-**}), for any prime divisor $Q$ on $S'$ we obtain
\begin{equation*}\tag{$*\!*\!*$}\label{proof-lem--discre-relation-***}
a(Q,S',\Delta_{S'})-a(Q,S,\Delta_{S})={\rm ord}_{Q}(M|_{T}-N|_{T})\leq {\rm ord}_{Q}(M_{1}|_{T}). 
\end{equation*}
Pick any component $E$ of $M_{1}$. 
Then $\sigma_{E}(N)={\rm coeff}_{E}(N)=0$ since $M$ and $N$ have no common components. 
Since $M$ is $g'$-exceptional, we have $K_{X'}+\Delta'=g'_{*}(g^{*}(K_{X}+\Delta)+N)$ by (\ref{proof-lem--discre-relation-*}). 
Since $K_{X}+\Delta$ is pseudo-effective, which is the hypothesis of Lemma \ref{lem--discre-relation}, we see that $K_{X'}+\Delta'$ is pseudo-effective.  
By Remark \ref{rem--asy-van-ord-1} (\ref{rem--asy-van-ord-1-(1)}), (\ref{rem--asy-van-ord-1-(3)}), and relation (\ref{proof-lem--discre-relation-*}), we have
\begin{equation*}
\begin{split}
\sigma_{E}(K_{X}+\Delta)=&\sigma_{E}(g^{*}(K_{X}+\Delta))\\
=&\sigma_{E}(g^{*}(K_{X}+\Delta))+\sigma_{E}(N)\\
\geq& \sigma_{E}(g^{*}(K_{X}+\Delta)+N)\\
=&\sigma_{E}(g'^{*}(K_{X'}+\Delta')+M)\\
=&\sigma_{E}(g'^{*}(K_{X'}+\Delta'))+{\rm coeff}_{E}(M)>0.
\end{split}
\end{equation*}
By definition of $M_{1}$, the image $g(E)$ intersects $S$. 
If $a(E,X,\Delta)<0$, then the inequality $\sigma_{E}(K_{X}+\Delta)>0$ contradicts the second condition of Lemma \ref{lem--discre-relation}. 
Therefore, it holds that $a(E,X,\Delta)\geq0$. 
Using (\ref{proof-lem--discre-relation-*}) again, we obtain 
$$a(E,X',\Delta')>a(E,X,\Delta)\geq0.$$
Thus, we see that all components $E$ of $M_{1}$ satisfy $a(E,X',\Delta')>0$. 
By Lemma \ref{lem--adjunction}, all components of $M_{1}|_{T}$ are exceptional over $S'$. 
Recalling that $Q$ is a prime divisor on $S'$, we see that ${\rm ord}_{Q}(M_{1}|_{T})=0$. 
By (\ref{proof-lem--discre-relation-***}), we have $a(Q,S',\Delta_{S'})\leq a(Q,S,\Delta_{S})$. 
\end{proof}

\subsection{Log MMP and models} 
In this paper, we will freely use the log MMP with scaling as in {\cite[Definition 2.4]{birkar-flip}} and {\cite[4.4.11]{fujino-book}}. 

\begin{defn}[Models, {\cite[Definition 2.2]{has-trivial}}]
Let $\pi\colon X \to Z$ be a projective morphism from a normal variety to a variety, and let $(X,\Delta)$ be an lc pair.  
Let $\pi '\colon X' \to Z$ be a projective morphism from a normal variety to $Z$ and let $\phi\colon X \dashrightarrow X'$ be a birational map over $Z$. 
Let $E$ be the reduced $\phi^{-1}$-exceptional divisor on $X'$, that is, $E=\sum E_{j}$ where $E_{j}$ are $\phi^{-1}$-exceptional prime divisors on $X'$. 
When $K_{X'}+\phi_{*}\Delta+E$ is $\mathbb{R}$-Cartier, the pair $(X', \Delta'=\phi_{*}\Delta+E)$ is called a {\it log birational model} of $(X,\Delta)$ over $Z$. 
A log birational model $(X', \Delta')$ of $(X,\Delta)$ over $Z$ is a {\it weak log canonical model} ({\it weak lc model}, for short) if 
\begin{itemize}
\item
$K_{X'}+\Delta'$ is nef over $Z$, and 
\item
for any prime divisor $D$ on $X$ which is exceptional over $X'$, we have
\begin{equation*}
a(D, X, \Delta) \leq a(D, X', \Delta').
\end{equation*}
\end{itemize}
A weak lc model $(X',\Delta')$ of $(X,\Delta)$ over $Z$ is a {\it log minimal model} if 
\begin{itemize}
\item
$X'$ is $\mathbb{Q}$-factorial, and 
\item
the above inequality on discrepancies is strict. 
\end{itemize}
A log minimal model $(X',\Delta')$ of $(X, \Delta)$ over $Z$ is called a {\it good minimal model} if $K_{X'}+\Delta'$ is semi-ample over $Z$. 

In particular, we do not assume that log minimal models are dlt. 
This definition is different from the definition of log minimal models in \cite{birkar-flip}. 
For a remark on the difference between these two definitions, see \cite[Remark 2.4]{has-trivial}. 
\end{defn}

\begin{rem}[{\cite[Remark 2.5]{hashizumehu}}, see also {\cite[Remark 2.7]{birkar-flip}}]\label{rem--models}
Let $\pi\colon X \to Z$ be a projective morphism from a normal variety to a variety,  and let $(X,\Delta)$ be an lc pair.  
Let $(X',\Delta')$ be a log minimal model of $(X,\Delta)$ over $Z$, and let $(X'',\Delta'')$ be a $\mathbb{Q}$-factorial lc pair with a small birational map $X' \dashrightarrow X''$ over $Z$ such that $K_{X''}+\Delta''$ is nef over $Z$ and the birational transform of $\Delta'$ on $X''$ is equal to $\Delta''$. 
Then $(X'',\Delta'')$ is also a log minimal model of $(X,\Delta)$ over $Z$. 
Furthermore, if $(X',\Delta')$ is a good minimal model of $(X,\Delta)$ over $Z$, then $(X'',\Delta'')$ is also a good minimal model of $(X,\Delta)$ over $Z$. 
\end{rem}

\begin{rem}\label{rem--mmp-zariskidecom}
Let $(X,\Delta)$ be a projective lc pair and let $(X',\Delta')$ be a weak lc model of $(X,\Delta)$. 
By taking a common resolution of $X\dashrightarrow X'$ and using the negativity lemma, we can easily check the following facts. 
\begin{enumerate}[(1)]
\item \label{rem--mmp-zariskidecom-(3)}
By Remark \ref{rem--asy-van-ord-1} (\ref{rem--asy-van-ord-1-(3)}) (see also \cite[Remark 2.6]{birkar-flip}), we have the equality
$$\sigma_{P}(K_{X}+\Delta)=a(P,X',\Delta')-a(P,X,\Delta)$$
for all prime divisors $P$ over $X$.  
\item \label{rem--mmp-zariskidecom-(1)}
Suppose that $(X',\Delta')$ is a log minimal model and 
 $(X,\Delta)\dashrightarrow (X',\Delta')$ is a sequence of steps of a $(K_{X}+\Delta)$-MMP. 
Let $D$ be a prime divisor on $X$. 
Then $D$ is contracted by the birational map $X\dashrightarrow X'$ if and only if $D$ is a component of $N_{\sigma}(K_{X}+\Delta)$. 
\end{enumerate}
\end{rem}

\begin{thm}[Dlt blow-up, {\cite[Theorem 10.4]{fujino-fund}}, {\cite[Theorem 3.1]{kollarkovacs}}]\label{thm--dltblowup} Let $X$ be a normal quasi-projective variety, and let $(X, \Delta)$ be an lc pair. 
Then, there is a projective birational morphism $f\colon Y \to X$ from a normal quasi-projective variety $Y$ such that $(Y,\Gamma)$ is a $\mathbb{Q}$-factorial dlt pair and $K_{Y}+\Gamma=f^{*}(K_{X}+\Delta)$, where $\Gamma$ is the sum of $f_{*}^{-1}\Delta$ and the reduced $f$-exceptional divisor. 
We call the morphism $f\colon (Y,\Gamma) \to (X,\Delta)$ a dlt blow-up, and we call $(Y,\Gamma)$ a $\mathbb{Q}$-factorial dlt model of $(X,\Delta)$. 
\end{thm}

\begin{lem}[$\mathbb{Q}$-factorial dlt closure]\label{lem--Qfacdltclosure}
Let $\pi\colon X\to Z$ be a projective morphism from a normal quasi-projective variety to a quasi-projective variety. 
Let $(X,\Delta)$ be a $\mathbb{Q}$-factorial dlt pair, and let $Z\hookrightarrow Z^{c}$ be an open immersion to a projective variety. 
Then there is a projective morphism $\pi^{c}\colon X^{c} \to Z^{c}$ and a projective $\mathbb{Q}$-factorial dlt pair $(X^{c},\Delta^{c})$ such that $X$ is an open subset of $X^{c}$, $\pi^{c}|_{X}=\pi$, $(X^{c}|_{X},\Delta^{c}|_{X})=(X,\Delta)$, and all lc centers of $(X^{c},\Delta^{c})$ intersect $X$. 
\end{lem}

\begin{proof}
We take an open immersion $X\hookrightarrow \overline{X}$ to a projective variety with a projective morphism $\overline{X}\to Z^{c}$. 
Since $(X,\Delta)$ is dlt, we can take a log resolution $f\colon X'\to X$ of $(X,\Delta)$ such that $a(E,X,\Delta)>-1$ for all $f$-exceptional prime divisors $E$. 
Let $\Delta'$ be the sum of $f^{-1}_{*}\Delta$ and the reduced $f$-exceptional divisor. 
By \cite[Theorem 4.17]{fujino-fund} and taking blow-ups if necessary, we may take an open immersion $X'\hookrightarrow \overline{X}'$ to a smooth projective variety and a birational morphism $\overline{X}'\to \overline{X}$ such that $(\overline{X}',\overline{\Delta}')$ is log smooth and all lc centers of $(\overline{X}',\overline{\Delta}')$ intersect $X'$, where $\overline{\Delta}'$ is the closure of $\Delta'$ in $\overline{X}'$. 

We run a $(K_{\overline{X}'}+\overline{\Delta}')$-MMP over $\overline{X}$ with scaling of an ample divisor. 
After finitely many steps, we get a projective $\mathbb{Q}$-factorial dlt pair $(X^{c},\Delta^{c})$ whose restriction over $X$, denoted by $(\tilde{X}^{c},\tilde{\Delta}^{c})$, is a $\mathbb{Q}$-factorial dlt model of $(X,\Delta)$ (cf.~\cite[Proof of Corollary 3.6]{birkar-flip}, see also \cite[Theorem 3.5]{birkar-flip}). 
$$
\xymatrix
{
 \overline{X}'\ar[dr]\ar@{-->}[rr]&&X^{c}\ar[dl]\ar@{}[r]|*{\supset}&\tilde{X}^{c}\ar[dl]\\
&\overline{X} \ar@{}[r]|*{\supset} &X&
}
$$
Since $a(E,X,\Delta)>-1$ for all $f$-exceptional prime divisors $E$, the birational  morphism $\tilde{X}^{c}\to X$ is small. 
Since $X$ is $\mathbb{Q}$-factorial, the morphism is an isomorphism. 
We see that $\Delta^{c}|_{X}=\Delta$ by construction. 
Since all lc centers of $(\overline{X}',\overline{\Delta}')$ intersect $X'$, all lc centers of $(X^{c},\Delta^{c})$ intersect $X$. 
Thus, $(X^{c},\Delta^{c})$ is the desired $\mathbb{Q}$-factorial dlt pair. 
\end{proof}

We close this subsection with a slight generalization of \cite[Lemma 2.14]{has-mmp}.

\begin{lem}\label{lem--mmp-termi}
Let $\pi\colon X \to Z$ be a projective morphism of normal quasi-projective varieties, and let $(X,\Delta)$ be an lc pair.  
Let $H\geq 0$ be an $\mathbb{R}$-Cartier divisor on $X$ such that $(X,\Delta+H)$ is an lc pair and $K_{X}+\Delta+H$ is nef over $Z$. 
Assume that for any $\mu \in (0,1]$, the pair $(X,\Delta+\mu H)$ has a log minimal model over $Z$.

Then, we can construct a sequence of steps of the $(K_{X}+\Delta)$-MMP over $Z$ with scaling of $H$ 
$$(X_{0}:=X,\Delta_{0}:=\Delta) \dashrightarrow (X_{1},\Delta_{1}) \dashrightarrow\cdots \dashrightarrow (X_{i},\Delta_{i})\dashrightarrow \cdots$$
such that if we define $$\lambda_{i}={\rm inf}\set{\mu \in \mathbb{R}_{\geq0} \!|\! K_{X_{i}}+\Delta_{i}+\mu H_{i}\text{\rm \, is nef over }Z},$$ where $H_{i}$ is the birational transform of $H$ on $X_{i}$,
then the $(K_{X}+\Delta)$-MMP terminates after finitely many steps or we have ${\rm lim}_{i \to \infty}\lambda_{i}=0$ even if the $(K_{X}+\Delta)$-MMP does not terminate. 
\end{lem}

\begin{proof}
The argument in \cite[Proof of Lemma 2.14]{has-mmp} works with no changes because we may use \cite[Proposition 6.2]{hashizumehu} or \cite[Theorem 1.7]{hashizumehu}.  
\end{proof}

\subsection{Invariant Iitaka dimension and numerical dimension}\label{subsec2.2}

We define invariant Iitaka dimension (\cite{choi}) and numerical dimension (\cite{nakayama}) for $\mathbb{R}$-Cartier divisors on normal projective varieties, and then we collect some properties. 
Afterwards, we define the notion of log abundant divisors and prove some results. 

\begin{defn}[Invariant Iitaka dimension]\label{defn--inv-iitaka-dim}
Let $X$ be a normal projective variety, and let $D$ be an $\mathbb{R}$-Cartier  divisor on $X$. 
We define the {\em invariant  Iitaka dimension} of $D$, denoted by $\kappa_{\iota}(X,D)$, as follows (\cite[Definition 2.2.1]{choi}, see also \cite[Definition 2.5.5]{fujino-book}):  
If there is an $\mathbb{R}$-divisor $E\geq 0$ such that $D\sim_{\mathbb{R}}E$, set $\kappa_{\iota}(X,D)=\kappa(X,E)$. 
Here, the right hand side is the usual Iitaka dimension of $E$. 
Otherwise, we set $\kappa_{\iota}(X,D)=-\infty$. 
Note that $\kappa(X,E)$ is independent of the choice of $E$ (see \cite[Proposition 2.1.2]{choi}), hence $\kappa_{\iota}(X,D)$ is well-defined.

Let $X\to Z$ be a projective morphism from a normal variety to a variety, and let $D$ be an $\mathbb{R}$-Cartier divisor on $X$. 
Then the {\em relative invariant Iitaka dimension} of $D$, denoted by $\kappa_{\iota}(X/Z,D)$, is similarly defined: 
If there is an $\mathbb{R}$-divisor $E\geq 0$ such that $D\sim_{\mathbb{R},Z}E$ then we set $\kappa_{\iota}(X/Z,D)=\kappa_{\iota}(F,D|_{F})$, where $F$ is a sufficiently general fiber of the Stein factorization of $X\to Z$, and otherwise we set $\kappa_{\iota}(X/Z,D)=-\infty$. 
Note that this definition is independent of the choice of $E$ and $F$, hence $\kappa_{\iota}(X/Z,D)$ is well-defined. 
\end{defn}

\begin{defn}[Numerical dimension]\label{defn--num-dim}
Let $X$ be a normal projective variety, and let $D$ be an $\mathbb{R}$-Cartier divisor on $X$. 
We define the {\em numerical dimension} of $D$, denoted by $\kappa_{\sigma}(X,D)$, as follows (\cite[V, 2.5 Definition]{nakayama}): 
For any Cartier divisor $A$ on $X$, we set
\begin{equation*}
\sigma(D;A)={\rm max}\!\Set{\!k\in \mathbb{Z}_{\geq0} | \underset{m\to \infty}{\rm lim\,sup}\frac{{\rm dim}\,H^{0}(X,\mathcal{O}_{X}(\llcorner mD \lrcorner+A))}{m^{k}}>0\!}
\end{equation*}
if ${\rm dim}\,H^{0}(X,\mathcal{O}_{X}(\llcorner mD \lrcorner+A))>0$ for infinitely many $m\in \mathbb{Z}_{>0}$, and otherwise we set $\sigma(D;A)=-\infty$. 
Then, we define 
\begin{equation*}
\kappa_{\sigma}(X,D):={\rm max}\!\set{\sigma(D;A) | \text{$A$ is a Cartier divisor on $X$}\!}.
\end{equation*}

Let $X\to Z$ be a projective morphism from a normal variety to a variety, and let $D$ be an $\mathbb{R}$-Cartier divisor on $X$. 
Then, the {\em relative numerical dimension} of $D$ over $Z$ is defined by $\kappa_{\sigma}(F,D|_{F})$, where $F$ is a sufficiently general fiber of the Stein factorization of $X\to Z$. In this paper, we denote $\kappa_{\sigma}(F,D|_{F})$ by $\kappa_{\sigma}(X/Z,D)$. 
\end{defn}

\begin{rem}[{\cite[V, 2.7 Proposition]{nakayama}, \cite[Remark 3.2]{has-class}}]\label{rem--div}
We write down some basic properties of the invariant Iitaka dimension and the numerical dimension. 
\begin{enumerate}[(1)]
\item \label{rem--div-(1)}
Let $D_{1}$ and $D_{2}$ be $\mathbb{R}$-Cartier divisors on a normal projective variety $X$. 
\begin{itemize}
\item
Suppose that $D_{1}\sim_{\mathbb{R}}D_{2}$ holds. 
Then the equalities $\kappa_{\iota}(X,D_{1})=\kappa_{\iota}(X,D_{2})$ and $\kappa_{\sigma}(X,D_{1})=\kappa_{\sigma}(X, D_{2})$ hold. 
\item
Suppose that $D_{1}\sim_{\mathbb{R}}N_{1}$ and $D_{2}\sim_{\mathbb{R}}N_{2}$ for some $\mathbb{R}$-divisors $N_{1}\geq0$ and $N_{2}\geq0$ such that ${\rm Supp}\,N_{1}={\rm Supp}\,N_{2}$. 
Then the equalities $\kappa_{\iota}(X,D_{1})=\kappa_{\iota}(X,D_{2})$ and $\kappa_{\sigma}(X,D_{1})=\kappa_{\sigma}(X, D_{2})$ hold. 
\end{itemize}
\item \label{rem--div-(2)}
Let $f\colon Y \to X$ be a surjective morphism of normal projective varieties, and let $D$ be an $\mathbb{R}$-Cartier divisor on $X$. 
\begin{itemize}
\item
The equalities $\kappa_{\iota}(X,D)=\kappa_{\iota}(Y,f^{*}D)$ and $\kappa_{\sigma}(X,D)=\kappa_{\sigma}(Y,f^{*}D)$ hold.
\item
Suppose that $f$ is birational. 
Let $D'$ be an $\mathbb{R}$-Cartier divisor on $Y$ such that  $D'-f^{*}D$ is effective and $f$-exceptional. 
Then we have $\kappa_{\iota}(X,D)=\kappa_{\iota}(Y,D')$ and $\kappa_{\sigma}(X,D)=\kappa_{\sigma}(Y, D')$.
\end{itemize}
\end{enumerate}
\end{rem}

\begin{defn}[Relatively abundant divisor and relatively log abundant divisor]\label{defn--abund}
Let $\pi\colon X\to Z$ be a projective morphism from a normal variety to a variety, and let $D$ be an $\mathbb{R}$-Cartier divisor on $X$. 
We say that $D$ is $\pi$-{\em abundant} (or {\em abundant over} $Z$) if the equality $\kappa_{\iota}(X/Z,D)=\kappa_{\sigma}(X/Z,D)$ holds. 
When $Z$ is a point, we simply say that $D$ is {\em abundant}. 

Let $\pi\colon X\to Z$ and $D$ be as above, and let $(X,\Delta)$ be an lc pair. 
We say that $D$ is $\pi$-{\em log abundant} (or {\em log abundant over} $Z$) {\em with respect to} $(X,\Delta)$ if $D$ is abundant over $Z$ and for any lc center $S$ of $(X,\Delta)$ with the normalization $S^{\nu}\to S$, the pullback $D|_{S^{\nu}}$ is abundant over $Z$. 
When $Z$ is a point, we simply say that $D$ is {\em log abundant with respect to} $(X,\Delta)$.  

Let $X\to Z$ be a projective morphism from a normal variety to a variety and $(X,\Delta)$ an lc (resp.~dlt) pair. When $K_{X}+\Delta$ is log abundant over $Z$ with respect to $(X,\Delta)$, we call $(X,\Delta)$ a {\em log abundant lc} (resp.~{\em log abundant dlt}) {\em pair over} $Z$ or we say that $(X,\Delta)$ is {\em log abundant over} $Z$. When $Z$ is a point, we simply say that $(X,\Delta)$ is {\em log abundant}.  
\end{defn}

\begin{lem}[{\cite[Lemma 2.11]{hashizumehu}}]\label{lem--iitakafib} 
Let $(X,\Delta)$ be a projective lc pair with an $\mathbb{R}$-divisor $\Delta$. 
Suppose that $K_{X}+\Delta$ is abundant and there is an effective $\mathbb{R}$-divisor $D\sim_{\mathbb{R}}K_{X}+\Delta$. 
Let $X\dashrightarrow V$ be the Iitaka fibration associated to $D$. 
Pick a log resolution $f\colon Y\to X$ of $(X,\Delta)$ such that the induced map $Y\dashrightarrow V$ is a morphism, and let $(Y,\Gamma)$ be a projective lc pair such that we can write $K_{Y}+\Gamma=f^{*}(K_{X}+\Delta)+E$ for an effective $f$-exceptional $\mathbb{R}$-divisor $E$. 
Then, we have $\kappa_{\sigma}(Y/V,K_{Y}+\Gamma)=0$. 
\end{lem}

\begin{lem}\label{lem--comparedim}
Let $\pi\colon X \to Z$ and $\pi'\colon X' \to Z$ be projective morphisms of normal quasi-projective varieties. 
Let $(X,\Delta)$ and $(X',\Delta')$ be dlt pairs, and let $S$ and $S'$ be lc centers of $(X,\Delta)$ and $(X',\Delta')$ respectively. 
Suppose that there is a small birational map $f\colon X\dashrightarrow X'$ over $Z$ such that $f_{*}\Delta=\Delta'$, $f$ is an isomorphism on an open set intersecting $S$, and the restriction of $f$ to $S$ induces a birational map $f|_{S}\colon S\dashrightarrow S'$. 
If there is an $\mathbb{R}$-divisor $A$ on $X$ such that $K_{X}+\Delta+aA$ and $K_{X'}+\Delta'+a'f_{*}A$ are nef over $Z$ for some $a>a'>0$, then the following inequalities hold true:
\begin{equation*}
\begin{split}
\kappa_{\iota}(S/Z,(K_{X}+\Delta)|_{S})&\geq \kappa_{\iota}(S'/Z,(K_{X'}+\Delta')|_{S'}) \quad{\rm and}\\
\kappa_{\sigma}(S/Z,(K_{X}+\Delta)|_{S})&\geq \kappa_{\sigma}(S'/Z,(K_{X'}+\Delta')|_{S'}). 
\end{split}
\end{equation*} 
\end{lem}

\begin{proof}
By taking a log resolution $g\colon Y\to X$ of $(X,\Delta)$ which resolves the indeterminacy of $f\colon X\dashrightarrow X'$, we may find a common log resolution $g$ and $g' \colon Y \to {X'}$ of the map $(X,\Delta)\dashrightarrow (X',\Delta')$ and a subvariety $T\subset Y$ birational to $S$ and $S'$ such that the induced morphisms $g|_{T}\colon T\to S$ and $g'|_{T}\colon T\to S'$ form a common resolution of $f|_{S}\colon S\dashrightarrow S'$. 
We note that $A$ and $f_{*}A$ are $\mathbb{R}$-Cartier by hypothesis. 
Since $f$ is small, by the negativity lemma, we may write
\begin{equation*}
\begin{split}
&E+g^{*}(K_{X}+\Delta+aA)=g'^{*}(K_{X'}+\Delta'+af_{*}A)\quad {\rm and}\\
&g^{*}(K_{X}+\Delta+a'A)=g'^{*}(K_{X'}+\Delta'+a'f_{*}A)+E'
\end{split}
\end{equation*}
for some $g$-exceptional $\mathbb{R}$-divisor $E\geq0$ and $g'$-exceptional $\mathbb{R}$-divisor $E'\geq0$. 
Since $a>a'>0$, if we put $F=\frac{1}{a-a'}(a'E+aE')$, then $F \geq 0$ and $g'$-exceptional and we have 
$$g^{*}(K_{X}+\Delta)=g'^{*}(K_{X'}+\Delta')+F.$$ 
Since $g\colon Y \to X$ is an isomorphism over an open subset which contains general points of $S$, we have $F|_{T}\geq0$. 
By restricting to $T$, we obtain
$$g|_{T}^{*}((K_{X}+\Delta)|_{S})=g'|_{T}^{*}((K_{X'}+\Delta')|_{S'})+F|_{T}.$$
From this relation, the inequalities of Lemma \ref{lem--comparedim} follow. 
\end{proof}

\begin{lem}\label{lem--relabund-globabund}
Let $\pi \colon X\to Z$ be a morphism of normal projective varieties, and let $(X,\Delta)$ be an lc pair such that $K_{X}+\Delta$ is abundant over $Z$. 
Let $A_{Z}$ be an ample Cartier divisor on $Z$, and pick $\alpha>2\cdot{\rm dim}\,X$. 
Then $K_{X}+\Delta+\alpha \pi^{*}A_{Z}$ is (globally) abundant. 
\end{lem}

\begin{proof}
We may assume that $\pi$ is a contraction and $(X,\Delta)$ is a $\mathbb{Q}$-factorial dlt pair by taking the Stein factorization of $\pi$ and a dlt blow-up of $(X,\Delta)$, respectively. 
We may assume that $K_{X}+\Delta$ is pseudo-effective over $Z$ because otherwise  $K_{X}+\Delta+\alpha \pi^{*}A_{Z}$ is not pseudo-effective. 

First we show that $K_{X}+\Delta+\alpha \pi^{*}A_{Z}$ is pseudo-effective. 
Pick any ample $\mathbb{R}$-divisor $H$ on $X$. 
By \cite{bchm}, there is a good minimal model $(\tilde{X},\tilde{\Delta}+\tilde{H})$ of $(X,\Delta+H)$ over $Z$. 
Let $\tilde{\pi}\colon \tilde{X}\to Z$ be the induced morphism. 
By the length of extremal rays (\cite[Section 18]{fujino-fund}), we see that $K_{\tilde{X}}+\tilde{\Delta}+\tilde{H}+\alpha \tilde{\pi}^{*}A_{Z}$ is  (globally) nef. 
This shows that $K_{X}+\Delta+H+\alpha \pi^{*}A_{Z}$ is pseudo-effective for any ample $H$. 
Thus $K_{X}+\Delta+\alpha \pi^{*}A_{Z}$ is pseudo-effective. 

We fix $\alpha'$ such that $2\cdot {\rm dim}\,X<\alpha'<\alpha$. 
By the previous argument, $K_{X}+\Delta+\alpha' \pi^{*}A_{Z}$ is pseudo-effective. 
Since $K_{X}+\Delta$ is pseudo-effective and abundant over $Z$, we may pick $E\geq0$ such that $K_{X}+\Delta\sim_{\mathbb{R},Z}E$. 
We take the relative Iitaka fibration $X \dashrightarrow V$ over $Z$ induced by $E$. 
By Remark \ref{rem--div} and replacing $(X,\Delta)$, we may assume that $X \dashrightarrow V$ is a morphism. 
By Lemma \ref{lem--iitakafib}, we have 
$$\kappa_{\sigma}(X/V,K_{X}+\Delta)=0.$$
Furthermore, the equality 
$$\kappa_{\sigma}(X/Z,K_{X}+\Delta)={\rm dim}\,V-{\rm dim}\,Z$$
holds by construction.  
By \cite[Lemma 3.1]{has-mmp}, we get a projective lc pair $(X',\Delta')$ and the following diagram of normal projective varieties
$$
\xymatrix
{
X\ar[d]&W\ar[l]_{p}\ar[r]^{p'}&X'\ar[d]^{\varphi'}\\
V\ar[dr]&&V'\ar[ll]\ar[dl]^{g'}\\
&Z
}
$$
such that 
\begin{itemize}
\item
the three horizontal morphisms are birational and the other morphisms are contractions, and
\item
$K_{X'}+\Delta'\sim_{\mathbb{R},V'}0$.
\end{itemize}
Furthermore, the construction of the diagram and \cite[Remark 3.2]{has-mmp} imply that we may assume the following property. 
\begin{itemize} 
\item
 $p^{*}(K_{X}+\Delta)+F_{p}=p'^{*}(K_{X'}+\Delta')+F_{p'}$ for a $p$-exceptional $\mathbb{R}$-divisor $F_{p}\geq0$ and a $p'$-exceptional $\mathbb{R}$-divisor $F_{p'}\geq 0$. 
\end{itemize}
Pick an $\mathbb{R}$-Cartier divisor $D'$ on $V'$ such that $\varphi'^{*}D'\sim_{\mathbb{R}}K_{X'}+\Delta'$. 
Then 
$$\kappa_{\sigma}(V'/Z,D')=\kappa_{\sigma}(X'/Z,K_{X'}+\Delta')=\kappa_{\sigma}(X/Z,K_{X}+\Delta)={\rm dim}\,V'-{\rm dim}\,Z,$$
 so $D'$ is big over $Z$. 
By construction and Remark \ref{rem--div}, $D'+\alpha' g'^{*}A_{Z}$ is pseudo-effective, and $K_{X}+\Delta+\alpha \pi^{*}A_{Z}$ is abundant if and only if $D'+\alpha g'^{*}A_{Z}$ is abundant. 
Now
\begin{equation*}
\begin{split}
D'+\alpha g'^{*}A_{Z}=(1-\delta)(D'+\alpha'g'^{*}A_{Z})+\delta D'+(\alpha-\alpha')g'^{*}A_{Z}+\delta \alpha'g'^{*}A_{Z}
\end{split}
\end{equation*}
for any $\delta>0$. 
Since $D'$ is big over $Z$ and $A_{Z}$ is ample, taking a sufficiently small $\delta>0$ we see that $D'+\alpha g'^{*}A_{Z}$ is big. 
Therefore $K_{X}+\Delta+\alpha \pi^{*}A_{Z}$ is abundant. 
\end{proof}

\begin{lem}[{cf.~\cite[Lemma 3.4]{has-class}}]\label{lem--relnefabundance}
Let $\pi\colon X\to Z$ be a projective morphism of normal quasi-projective varieties, and let $(X,\Delta)$ be an lc pair. 
If $K_{X}+\Delta$ is nef and log abundant over $Z$ with respect to $(X,\Delta)$, then it is semi-ample over $Z$. 
\end{lem}

\begin{proof}
When $\Delta$ is a $\mathbb{Q}$-divisor, the lemma  follows by induction using \cite[Theorem 1.1]{fujino-abund-saturation} and \cite[Theorem 2]{haconxu}. 
When $\Delta$ is not necessarily a $\mathbb{Q}$-divisor, we may reduce to the case of $\mathbb{Q}$-divisors by the argument of \cite[Proof of Lemma 3.4]{has-class} with minor changes. 
\end{proof}

\subsection{On log MMP and existence of good minimal models}

In this subsection we collect results on the log MMP and the existence of good minimal models for lc pairs. 

Firstly, we introduce a remark and a lemma on log MMP with scaling.  

\begin{rem}[{\cite{fujino-sp-ter}, \cite[Remark 2.10]{birkar-flip}}]\label{rem--mmplift}
We recall a part of the argument of the special termination which is important for the proofs of the main results of this paper. 

Let $\pi\colon X \to Z$ be a projective morphism of normal quasi-projective varieties and $(X,\Delta)$ a dlt pair. 
Fix an lc center $S$ of $(X,\Delta)$, and define a dlt pair $(S,\Delta_{S})$ where $\Delta_{S}$ is constructed by adjunction $K_{S}+\Delta_{S}=(K_{X}+\Delta)|_{S}$. 
Let $A$ be an effective $\mathbb{R}$-divisor on $X$ such that $(X,\Delta+A)$ is lc and $K_{X}+\Delta+A$ is nef over $Z$.
For a sequence of steps of a $(K_{X}+\Delta)$-MMP over $Z$ with scaling of $A$
$$(X_{0}:=X,\Delta_{0}:=\Delta) \dashrightarrow (X_{1},\Delta_{1}) \dashrightarrow\cdots \dashrightarrow (X_{i},\Delta_{i})\dashrightarrow \cdots,$$ 
suppose that each birational map $X\dashrightarrow X_{i}$ induces a small birational map $S\dashrightarrow S_{i}$ to an lc center $S_{i}$ of $(X_{i},\Delta_{i})$ such that 
the birational transform of $\Delta_{S}$ on $S_{i}$ is equal to the $\mathbb{R}$-divisor $\Delta_{S_{i}}$ defined by adjunction $K_{S_{i}}+\Delta_{S_{i}}=(K_{X_{i}}+\Delta_{i})|_{S_{i}}$. 
Let $(T,\Psi)\to (S,\Delta_{S})$ be a dlt blow-up, and we put $A_{T}=A|_{T}$.
Let $X\to V$ be the $(K_{X}+\Delta)$-negative extremal contraction of the first step of the $(K_{X}+\Delta)$-MMP over $Z$. 
It is obvious that $(S_{1},\Delta_{S_{1}})$ is a weak lc model of $(T,\Psi)$ over $V$ with relatively ample log canonical divisor. 
Then $(T,\Psi)$ has a good  minimal model over $V$ (\cite[Remark 2.10]{has-mmp}). 
By \cite[Theorem 4.1 (iii)]{birkar-flip}, we get a sequence of steps of a $(K_{T}+\Psi)$-MMP over $V$ terminating with a good minimal model $(T_{1},\Psi_{1})$ over $V$. 
By construction, this is a $(K_{T}+\Psi)$-MMP over $Z$ with scaling of $A_{T}$ and $(T_{1},\Psi_{1})$ is a $\mathbb{Q}$-factorial dlt model of $(S_{1},\Delta_{S_{1}})$ as in Theorem \ref{thm--dltblowup}. 
Repeating this discussion, over the sequence of birational maps $S\dashrightarrow \cdots \dashrightarrow S_{i} 
\dashrightarrow \cdots$, we can construct a sequence of steps of a $(K_{T}+\Psi)$-MMP over $Z$ with scaling of $A_{T}$
 $$(T_{0}:=T,\Psi_{0}:=\Psi) \dashrightarrow (T_{1},\Psi_{1}) \dashrightarrow \cdots \dashrightarrow (T_{i},\Psi_{i})\dashrightarrow \cdots$$
such that each $(T_{i},\Psi_{i})$ is a $\mathbb{Q}$-factorial dlt model of $(S_{i},\Delta_{S_{i}})$. 
We note that the birational map $(T_{i},\Psi_{i})\dashrightarrow (T_{i+1},\Psi_{i+1})$ is not necessarily one step of the $(K_{T_{i}}+\Psi_{i})$-MMP and it can be an isomorphism for some $i$. 

We define $$\lambda_{i}={\rm inf}\set{\mu \in \mathbb{R}_{\geq0} \!|\! K_{X_{i}}+\Delta_{i}+\mu A_{i}\text{\rm \, is nef over }Z},$$ where $A_{i}$ is the birational transform of $A$ on $X_{i}$. 
By the negativity lemma, for every $i$, the birational map $X_{i}\dashrightarrow X_{i+1}$ is an isomorphism on a neighborhood of $S_{i}$ if and only if the induced birational map $(S_{i}, \Delta_{S_{i}})\dashrightarrow (S_{i+1}, \Delta_{S_{i+1}})$ is an isomorphism. 
The isomorphism holds for all $i\gg 0$ if the $(K_{T}+\Psi)$-MMP over $Z$ terminates. 
Let $A_{T_{i}}$ be the birational transform of $A_{T}$ on $T_{i}$. 
We define 
$$\lambda'_{i}={\rm inf}\set{\mu \in \mathbb{R}_{\geq0} \!|\! K_{T_{i}}+\Psi_{i}+\mu A_{T_{i}}\text{\rm \, is nef over }Z}.$$ 
Then $\lambda'_{i}\leq \lambda_{i}$. 

Thus, if ${\rm lim}_{i\to \infty}\lambda_{i}=0$, then ${\rm lim}_{i\to \infty}\lambda'_{i}=0$ and furthermore $K_{T}+\Psi$ is pseudo-effective over $Z$. 
In this situation, by \cite[Theorem 4.1 (iii)]{birkar-flip} and \cite[Lemma 2.15]{has-mmp} and  since $(T,\Psi)$ is a $\mathbb{Q}$-factorial dlt model of $(S,\Delta_{S})$, the existence of a log minimal model of $(S,\Delta_{S})$ over $Z$ implies the termination of the $(K_{T}+\Psi)$-MMP over $Z$, and the termination of the $(K_{T}+\Psi)$-MMP implies that the birational map $X_{i} \dashrightarrow X_{i+1}$ is an isomorphism on a neighborhood of $S_{i}$ for all $i \gg0$, as explained above. 
\end{rem}

\begin{lem}\label{lem--mmpscaling-effective}
Let $\pi\colon X \to Z$ be a projective morphism of normal quasi-projective varieties, $(X,\Delta)$ an lc pair, and let $S$ be a subvariety of $X$. 
Let
$$(X_{0}:=X,\Delta_{0}:=\Delta) \dashrightarrow (X_{1},\Delta_{1}) \dashrightarrow\cdots \dashrightarrow (X_{i},\Delta_{i})\dashrightarrow \cdots$$ 
be a sequence of steps of a $(K_{X}+\Delta)$-MMP over $Z$ with scaling of an $\mathbb{R}$-divisor $A\geq0$. 
Let $A_{i}$ be the birational transform of $A$ on $X_{i}$. 
We define 
$$\lambda_{i}={\rm inf}\set{\mu \in \mathbb{R}_{\geq0} \!|\! K_{X_{i}}+\Delta_{i}+\mu A_{i}\text{\rm \, is nef over }Z}.$$ 
Suppose that each step of the $(K_{X}+\Delta)$-MMP is an isomorphism on a neighborhood of $S$ and ${\rm lim}_{i\to\infty}\lambda_{i}=0$.  
Then, for any $\pi$-ample $\mathbb{R}$-divisor $H$ on $X$ and closed point $x\in S$, there is $E\geq 0$ such that $E\sim_{\mathbb{R},Z}K_{X}+\Delta+H$ and ${\rm Supp}\,E \not\ni x$. 
In particular, if $Z$ is a point, then $\sigma_{P}(K_{X}+\Delta)=0$ for all prime divisors $P$ over $X$ such that $c_{X}(P)$ intersects $S$, where $\sigma_{P}(\,\cdot\,)$ is the asymptotic vanishing order of $P$ defined in Definition \ref{defn--asy-van-ord}.  
\end{lem}

\begin{proof}
Fix a $\pi$-ample $\mathbb{R}$-divisor $H$ on $X$ and a closed point $x\in S$. 
Fix $i \gg 0$ such that $\frac{1}{2}H-\lambda_{i}A$ is $\pi$-ample and $X \dashrightarrow X_{i}$ is a sequence of steps of a $(K_{X}+\Delta+\lambda_{i}A+\epsilon H)$-MMP over $Z$ for some $\epsilon\in (0,\frac{1}{2})$. 
Since $H-\epsilon H-\lambda_{i}A$ is ample over $Z$ and 
$$K_{X}+\Delta+H=K_{X}+\Delta+\lambda_{i}A+\epsilon H+(H-\epsilon H-\lambda_{i}A),$$
it is sufficient to find $E_{0}\geq 0$ such that $E_{0}\sim_{\mathbb{R},Z}K_{X}+\Delta+\lambda_{i}A+\epsilon H$ and ${\rm Supp}\,E_{0}\not\ni x$. 
Note that $X \dashrightarrow X_{i}$ is a birational contraction and the birational transform $H_{i}$ of $H$ on $X_{i}$ is $\mathbb{R}$-Cartier (\cite[Remark 6.1]{hashizumehu}). 

Let $U$ be the largest open subset of $X$ on which $X \dashrightarrow X_{i}$ is an isomorphism, and let $U_{i}$ be its image on $X_{i}$. 
Note that $U$ depends on the fixed index $i$.  
By hypothesis, the inclusion $S \subset U$ holds.
Moreover, the codimension of $X_{i}\setminus U_{i}$ in $X_{i}$ is at least two and $H_{i}|_{U_{i}}$ is ample over $Z$. 
Let $S_{i}$ (resp.~$x_{i}$) be the image of $S$ (resp.~$x$) on $X_{i}$. 
Then $x_{i}\in S_{i}\subset U_{i}$. 
Pick an ample $\mathbb{R}$-divisor $B_{i}$ on $X_{i}$ such that the restriction $(\epsilon H_{i}-B_{i})|_{U_{i}}$ is ample over $Z$. 
Then there is $E_{U_{i}}\geq 0$ such that ${\rm Supp}\,E_{U_{i}}\not\ni x_{i}$ and
$$E_{U_{i}}\sim_{\mathbb{R},Z}(K_{X_{i}}+\Delta_{i}+\lambda_{i}A_{i}+B_{i})|_{U_{i}}+(\epsilon H_{i}-B_{i})|_{U_{i}}.$$ 
Let $E_{i}$ be the closure of $E_{U_{i}}$ in $X_{i}$. 
Then we have $E_{i}\geq0$ and ${\rm Supp}\,E_{i}\not\ni x_{i}$. 
We also have $E_{i}\sim_{\mathbb{R},Z}K_{X_{i}}+\Delta_{i}+\lambda_{i}A_{i}+\epsilon H_{i}$ since the codimension of $X_{i}\setminus U_{i}$ in $X_{i}$ is at least two.  
By combining this with the fact that $H_{i}$ is $\mathbb{R}$-Cartier, we see that $E_{i}$ is $\mathbb{R}$-Cartier. 

We take the normalization $Y$ of the graph of $X \dashrightarrow X_{i}$, and we denote the natural morphisms $Y\to X$ and $Y\to X_{i}$ by $f$ and $f_{i}$, respectively. 
Since $X \dashrightarrow X_{i}$ is a sequence of steps of a $(K_{X}+\Delta+\lambda_{i}A+\epsilon H)$-MMP, we have
$$f^{*}(K_{X}+\Delta+\lambda_{i}A+\epsilon H)=f_{i}^{*}(K_{X_{i}}+\Delta_{i}+\lambda_{i}A_{i}+\epsilon H_{i})+F$$
for some $F\geq0$, and ${\rm Supp}\,F$ does not intersect $f^{-1}(U)$. 
We set $E_{0}=f_{*}(f_{i}^{*}E_{i}+F)$. 
By definition of $x_{i}$ and $U$, it follows that $E_{0}\sim_{\mathbb{R},Z}K_{X}+\Delta+\lambda_{i}A+\epsilon H$ and ${\rm Supp}\,E_{0}\not\ni x$. 

The final assertion of Lemma \ref{lem--mmpscaling-effective} is obvious from this fact and the definition of the asymptotic vanishing order. 
\end{proof}

Secondly, we introduce the good minimal model analog of \cite[Theorem 1.1]{bhzariski}. 
This is known to the experts, and the proof of \cite[Theorem 1.1]{bhzariski} works with no changes to prove the case of the existence of a good minimal model. 
So we skip the proof.

\begin{thm}[cf.~{\cite[Theorem 1.1]{bhzariski}}]\label{thmzariski} 
Let $(X,\Delta)$ be a projective lc pair such that $K_{X}+\Delta$ is pseudo-effective. 
Then $(X,\Delta)$ has a log minimal model (resp.~a good minimal model) if and only if there exists a resolution $f\colon Y\to X$ such that $f^{*}(K_{X}+\Delta)$ has the Nakayama--Zariski decomposition with nef (resp.~semi-ample) positive part. 
\end{thm}

Thirdly, we prepare results on equivalence about the existence of log minimal models and good minimal models for given two lc pairs. 

\begin{lem}[{\cite[Lemma 2.15]{has-mmp}}]\label{lem--bir-equiv}
Let $\pi\colon X \to Z$ be a projective morphism of normal quasi-projective varieties and $(X,\Delta)$ an lc pair. 
Let $(Y,\Gamma)$ be an lc pair such that there is a projective birational morphism $f\colon Y\to X$ satisfying $K_{Y}+\Gamma=f^{*}(K_{X}+\Delta)+E$ for an effective $f$-exceptional $\mathbb{R}$-divisor $E$. 
Then $(X,\Delta)$ has a weak lc model (resp.~a log minimal model, a good minimal model) over $Z$ if and only if  $(Y,\Gamma)$ has a weak lc model (resp.~a log minimal model, a good minimal model) over $Z$. 
\end{lem}

\begin{lem}[cf.~{\cite[Lemma 5.3]{hmx-boundgentype}}]\label{lem--minmodel-zariskidecom}
Let $(X,\Delta)$ be a projective lc pair, and let $(Y,\Gamma)$ be an lc pair with a projective birational morphism $f\colon Y\to X$. 
Suppose that $K_{X}+\Delta$ is pseudo-effective and all prime divisors $D$ on $Y$ satisfy
$$0\leq a(D,Y,\Gamma)-a(D,X,\Delta)\leq \sigma_{D}(K_{X}+\Delta).$$  
Then $K_{Y}+\Gamma$ is pseudo-effective. 
Furthermore, $(X,\Delta)$ has a log minimal model (resp.~a good minimal model) if and only if $(Y,\Gamma)$ has a log minimal model (resp.~a good minimal model).  
\end{lem}

\begin{proof}
Let $\Delta_{Y}$ be an $\mathbb{R}$-divisor defined by $K_{Y}+\Delta_{Y}=f^{*}(K_{X}+\Delta)$. 
Then the hypothesis on discrepancies is equivalent to $0\leq \Delta_{Y}-\Gamma\leq N_{\sigma}(K_{Y}+\Delta_{Y})$. 
Thus we have $\Delta_{Y}\geq 0$. 
By Lemma \ref{lem--bir-equiv}, we may replace $(X,\Delta)$ by $(Y,\Delta_{Y})$. 
Therefore we may assume $X=Y$. 

Pick any resolution $g\colon W\to X$. 
By Remark \ref{rem--asy-van-ord-1} (\ref{rem--asy-van-ord-1-(4)}), we have $$N_{\sigma}(g^{*}(K_{X}+\Delta))\geq g^{*}(\Delta-\Gamma) = g^{*}((K_{X}+\Delta)-(K_{X}+\Gamma)).$$
Therefore, $g^{*}(K_{X}+\Gamma)\geq g^{*}(K_{X}+\Delta)-N_{\sigma}(g^{*}(K_{X}+\Delta))$ and the right hand side is the limit of effective divisors by definition of $N_{\sigma}(\,\cdot\,)$. 
Thus, $g^{*}(K_{X}+\Gamma)$ is pseudo-effective, hence $K_{X}+\Gamma$ is pseudo-effective.  
By \cite[III, 1.8 Lemma]{nakayama}, we have
\begin{equation*}
\begin{split}
N_{\sigma}(g^{*}(K_{X}+\Gamma))&=N_{\sigma}\bigl(g^{*}(K_{X}+\Delta)-g^{*}((K_{X}+\Delta)-(K_{X}+\Gamma))\bigr)\\
&=N_{\sigma}(g^{*}(K_{X}+\Delta))-g^{*}(K_{X}+\Delta)+g^{*}(K_{X}+\Gamma).
\end{split}
\end{equation*}
Thus, the positive parts of the Nakayama--Zariski decompositions of $g^{*}(K_{X}+\Delta)$ and $g^{*}(K_{X}+\Gamma)$ are the same. 
Since the morphism $g\colon W\to X$ is an arbitrary resolution, we may apply Theorem \ref{thmzariski}. 
By Theorem \ref{thmzariski}, we see that Lemma \ref{lem--minmodel-zariskidecom} holds.
\end{proof}

\begin{lem}\label{lem--bir-relation}
Let $\pi\colon X \to Z$ and $\pi'\colon X' \to Z$ be projective morphisms of normal quasi-projective varieties, and let $X\dashrightarrow X'$ be a birational map over $Z$. 
Let $(X,\Delta)$ and $(X',\Delta')$ be lc pairs. 
Suppose that 
\begin{itemize}
\item
$a(P,X,\Delta)\leq a(P,X',\Delta')$ for all prime divisors $P$ on $X$, and 
\item
$a(P',X',\Delta')\leq a(P',X,\Delta)$ for all prime divisors $P'$ on $X'$. 
\end{itemize}
Then $K_{X}+\Delta$ is abundant over $Z$ if and only if $K_{X'}+\Delta'$ is abundant over $Z$. 
Moreover, $(X,\Delta)$ has a log minimal model (resp.~a good minimal model) over $Z$ if and only if $(X',\Delta')$ has a log minimal model (resp.~a good minimal model) over $Z$. 
\end{lem}

\begin{proof}
Let $f\colon Y\to X$ and $f'\colon Y\to X'$ be a common log resolution of the map $(X,\Delta)\dashrightarrow (X',\Delta')$. 
We define an $\mathbb{R}$-divisor $\Gamma$ on $Y$ by
$$\Gamma=-\sum_{E}{\rm min}\{0,a(E,X,\Delta),a(E,X',\Delta')\}E,$$
where $E$ runs over all prime divisors on $Y$. 
By construction of $Y$ and since $(X,\Delta)$ and $(X',\Delta')$ are lc, $(Y,\Gamma)$ is log smooth and lc. 
We put $F=K_{Y}+\Gamma-f^{*}(K_{X}+\Delta)$ and $F'=K_{Y}+\Gamma-f'^{*}(K_{X'}+\Delta').$ 
Then, for any prime divisor $D$ on $Y$, we have
$${\rm coeff}_{D}(F)=a(D,X,\Delta)-{\rm min}\{0,a(D,X,\Delta),a(D,X',\Delta')\}\geq0,$$
and ${\rm coeff}_{D}(F)=0$ when $D$ is a prime divisor on $X$ because we have $a(D,X,\Delta)\leq 0$ and $a(D,X,\Delta)\leq a(D,X',\Delta')$ by the first condition of Lemma \ref{lem--bir-relation}. 
Thus $F$ is effective and $f$-exceptional. 
The same argument shows that $F'$ is effective and $f'$-exceptional. 
By applying Remark \ref{rem--div} (\ref{rem--div-(2)}) or Lemma \ref{lem--bir-equiv}  to the morphisms $(Y,\Gamma)\to (X,\Delta)$ and $(Y,\Gamma)\to (X',\Delta')$, we see that Lemma \ref{lem--bir-relation} holds. 
\end{proof}

Finally, we prove a key lemma for the proof of Theorem \ref{thm--intro-1}.

\begin{lem}\label{lem--abund-inv}
Let $\pi \colon X\to Z$ be a projective morphism of normal quasi-projective varieties, $(X,\Delta)$ a $\mathbb{Q}$-factorial dlt pair such that $K_{X}+\Delta$ is pseudo-effective over $Z$, and let $A$ be a $\pi$-ample $\mathbb{R}$-divisor on $X$. 
Then there exists a real number $\epsilon>0$ such that for any two real numbers $t,t'\in(0,\epsilon]$, any sequence of steps of a $(K_{X}+\Delta+tA)$-MMP 
$(X,\Delta+t A)\dashrightarrow (Y,\Delta_{Y}+tA_{Y})$ over $Z$ to a good minimal model $(Y,\Delta_{Y}+tA_{Y})$, and any sequence of steps of a $(K_{X}+\Delta+t'A)$-MMP $(X,\Delta+t' A)\dashrightarrow (Y',\Delta_{Y'}+t'A_{Y'})$ over $Z$ to a good minimal model  $(Y',\Delta_{Y'}+t'A_{Y'})$, the following properties hold. 
 \begin{itemize}
 \item
The induced birational map $Y\dashrightarrow Y'$ is small,
\item
$Y\dashrightarrow Y'$ is an isomorphism on an open subset $U\subset Y$ which intersects all lc centers of $(Y,\Delta_{Y})$ and whose image $U'\subset Y'$ intersects all lc centers of $(Y',\Delta_{Y'})$, and in particular, $Y \dashrightarrow Y'$ induces a one-to-one correspondence between lc centers of $(Y,\Delta_{Y})$ and lc centers of $(Y',\Delta_{Y'})$, and 
\item
for any lc center $S_{Y}$ of $(Y,\Delta_{Y})$ with the induced birational map $S_{Y}\dashrightarrow S_{Y'}$ to the corresponding lc center $S_{Y'}$ of $(Y',\Delta_{Y'})$, the following equalities hold:
\begin{equation*}
\begin{split}
\kappa_{\iota}(S_{Y}/Z, (K_{Y}+\Delta_{Y})|_{S_{Y}})&=\kappa_{\iota}(S_{Y'}/Z, (K_{Y'}+\Delta_{Y'})|_{S_{Y'}})\quad {\rm and}\\ 
\kappa_{\sigma}(S_{Y}/Z, (K_{Y}+\Delta_{Y})|_{S_{Y}})&=\kappa_{\sigma}(S_{Y'}/Z, (K_{Y'}+\Delta_{Y'})|_{S_{Y'}}). 
\end{split}
\end{equation*}
In particular, $(Y,\Delta_{Y})$ is log abundant over $Z$ if and only if $(Y',\Delta_{Y'})$ is log abundant over $Z$. 
\end{itemize}
\end{lem}

\begin{proof}
By rescaling $A$, we may assume that $K_{X}+\Delta+A$ is nef over $Z$. 
By adding the pullback of a sufficiently ample $\mathbb{R}$-divisor on $Z$ to $A$, we may further assume that $A$ is ample. 
By replacing $A$ with a general element of $|A|_{\mathbb{R}}$ (see \cite[pp.~112--113]{hashizumehu} for general semi-ample $\mathbb{R}$-divisors), we may assume that $A\geq0$ and that $(X,\Delta+A)$ is a dlt pair whose lc centers are lc centers of $(X,\Delta)$. 
We fix a sequence of steps of a $(K_{X}+\Delta)$-MMP over $Z$ with scaling of $A$
$$(X_{0}:=X,\Delta_{0}:=\Delta) \dashrightarrow (X_{1},\Delta_{1}) \dashrightarrow\cdots \dashrightarrow (X_{i},\Delta_{i})\dashrightarrow \cdots.$$ 
For each $i \geq 0$, let $A_{i}$ be the birational transform of $A$ on $X_{i}$. 
We set 
$$\lambda_{i}={\rm inf}\set{\mu \in \mathbb{R}_{\geq0} \!|\! K_{X_{i}}+\Delta_{i}+\mu A_{i}\text{\rm \, is nef over }Z}.$$
Then ${\rm lim}_{i\to \infty} \lambda_{i}=0$ by \cite{bchm} (see also \cite[Theorem 4.1 (ii)]{birkar-flip}). 

\begin{step3}\label{step1--abund-inv}
In this step, we find an index $m$ such that the birational map $X_{m} \dashrightarrow X_{i}$ satisfies the three properties of Lemma \ref{lem--abund-inv} for any $i\geq m$. 

If the $(K_{X}+\Delta)$-MMP over $Z$ terminates with a log minimal model $(X_{m'},\Delta_{m'})$ over $Z$, then there is nothing to prove because we may take $m'$ as $m$. 
Thus we may assume that the $(K_{X}+\Delta)$-MMP over $Z$ does not terminate.  Then there is $i'$ such that for any $i\geq i'$, the birational map $X_{i}\dashrightarrow X_{i+1}$ is small and for any lc center $S_{i}$ of $(X_{i},\Delta_{i})$, the restriction of $X_{i}\dashrightarrow X_{i+1}$ to $S_{i}$ induces a birational map $S_{i}\dashrightarrow S_{i+1}$ to an lc center $S_{i+1}$ of $(X_{i+1},\Delta_{i+1})$ (cf.~\cite[Proof of Theorem 4.2.1]{fujino-sp-ter}). 
For any $i\geq i'$ and any lc center $S_{i}$ of $(X_{i},\Delta_{i})$ with the birational map $S_{i}\dashrightarrow S_{i+1}$, we may find a common resolution $ Y\to X_{i}$ and $Y\to X_{i+1}$ of $X_{i}\dashrightarrow X_{i+1}$ and a subvariety $T \subset Y$ birational to $S_{i}$ and $S_{i+1}$ such that the induced morphisms $T \to S_{i}$ and $T \to S_{i+1}$ form a common resolution of $S_{i}\dashrightarrow S_{i+1}$. 
By \cite[Lemma 4.2.10]{fujino-sp-ter}, we have 
$(K_{X_{i}}+\Delta_{i})|_{T}=(K_{X_{i+1}}+\Delta_{i+1})|_{T}+E$  
 for some $E\geq0$ on $T$. 
Then, for the sequence 
$$S_{i'}\dashrightarrow \cdots \dashrightarrow S_{i}\dashrightarrow \cdots$$
 of birational maps of lc centers, the two sequences 
 $$\{\kappa_{\iota}(S_{i}/Z, (K_{X_{i}}+\Delta_{i})|_{S_{i}})\}_{i\geq i'} \qquad {\rm and} \qquad \{\kappa_{\sigma}(S_{i}/Z, (K_{X_{i}}+\Delta_{i})|_{S_{i}})\}_{i\geq i'}$$
  consisting of elements of $\mathbb{Z}_{\geq 0}\cup\{-\infty\}$ are non-increasing. 
From this, for the sequence $S_{i'}\dashrightarrow \cdots $ of birational maps, $\kappa_{\iota}(S_{i}/Z, (K_{X_{i}}+\Delta_{i})|_{S_{i}})$ and $\kappa_{\sigma}(S_{i}/Z, (K_{X_{i}}+\Delta_{i})|_{S_{i}})$ do not depend on $i$ for all $i\gg i'$. 

By this discussion, we can find an index $m$ such that for any $i\geq m$, the birational map $X_{m} \dashrightarrow X_{i}$ satisfies the three properties of Lemma \ref{lem--abund-inv}, in other words, we have 
\begin{itemize}
\item
the birational map $X_{m}\dashrightarrow X_{i}$ is small,
\item
$X_{m}\dashrightarrow X_{i}$ is an isomorphism on an open set $U_{m}\subset X_{m}$ such that $U_{m}$ intersects all lc centers of $(X_{m},\Delta_{m})$ and the image $U_{i}\subset X_{i}$ intersects all lc centers of $(X_{i},\Delta_{i})$, and 
\item
for any lc center $S_{m}$ of $(X_{m},\Delta_{m})$ with the induced birational map $S_{m}\dashrightarrow S_{i}$ to the corresponding lc center $S_{i}$ of $(X_{i},\Delta_{i})$, the following equalities hold: 
\begin{equation*}
\begin{split}
\kappa_{\iota}(S_{m}/Z, (K_{X_{m}}+\Delta_{m})|_{S_{m}})&=\kappa_{\iota}(S_{i}/Z, (K_{X_{i}}+\Delta_{i})|_{S_{i}})\quad {\rm and}\\ 
\kappa_{\sigma}(S_{m}/Z, (K_{X_{m}}+\Delta_{m})|_{S_{m}})&=\kappa_{\sigma}(S_{i}/Z, (K_{X_{i}}+\Delta_{i})|_{S_{i}}). 
\end{split}
\end{equation*}
\end{itemize}
\end{step3}

\begin{step3}\label{step2--abund-inv}
We may assume that $m>0$. 
We set $\epsilon=\frac{\lambda_{m-1}}{2}$. 
In this step, we prove that for any $t\in(0,\epsilon]$ and any sequence of steps of a $(K_{X}+\Delta+tA)$-MMP 
$$(X,\Delta+t A)\dashrightarrow (Y,\Delta_{Y}+tA_{Y})$$ over $Z$ to a good minimal model $(Y,\Delta_{Y}+tA_{Y})$, the birational map $Y\dashrightarrow X_{m}$ satisfies the three conditions of Lemma \ref{lem--abund-inv}. 

Fix $t\in(0,\epsilon]$ and $(X,\Delta+t A)\dashrightarrow (Y,\Delta_{Y}+tA_{Y})$ the log MMP over $Z$ as above. 
Let $j$ be the minimum index such that $K_{X_{j}}+\Delta_{j}+tA_{j}$ is nef over $Z$. 
Such $j$ exists since ${\rm lim}_{i\to \infty} \lambda_{i}=0$, and we have $j\geq m$ since $\lambda_{m-1}>\epsilon\geq t$. 
By construction, 
$$(X,\Delta+tA)\dashrightarrow (X_{j},\Delta_{j}+tA_{j})$$ is a sequence of steps of a $(K_{X}+\Delta+tA)$-MMP over $Z$ to a good minimal model. 

Both $(X,\Delta+tA)\dashrightarrow (X_{j},\Delta_{j}+tA_{j})$ and $(X,\Delta+t A)\dashrightarrow (Y,\Delta_{Y}+tA_{Y})$ are sequences of steps of the $(K_{X}+\Delta+tA)$-MMP over $Z$ to good minimal models. 
So the induced birational map $Y\dashrightarrow X_{j}$ is small and it is an isomorphism on an open set $V\subset Y$ which intersects all lc centers of $(Y,\Delta_{Y}+tA_{Y})$ and whose image $V_{j}\subset X_{j}$ intersects all lc centers of $(X_{j},\Delta_{j}+tA_{j})$ (see, for example, \cite[Lemma 2.8]{has-flop}). 
Since $A$ is general, lc centers of $(X_{j},\Delta_{j}+tA_{j})$ (resp.~lc centers of $(Y,\Delta_{Y}+tA_{Y})$) are the same as lc centers of $(X_{j},\Delta_{j})$ (resp.~lc centers of $(Y,\Delta_{Y})$). 
By applying the first and the second conditions in Step \ref{step1--abund-inv} to $X_{m}\dashrightarrow X_{j}$ and $X_{m}\dashrightarrow X_{i}$ for any $i\geq m$, we see that $Y\dashrightarrow X_{i}$ is small and we can find an open set $V^{(i)}\subset Y$ such that $Y\dashrightarrow X_{i}$ is an isomorphism on $V^{(i)}$, all lc centers of $(Y,\Delta_{Y})$ intersect $V^{(i)}$, and the image $V_{i}\subset X_{i}$ intersects all lc centers of $(X_{i},\Delta_{i})$. 
In this way, the first and the second conditions of Lemma \ref{lem--abund-inv} hold for $Y\dashrightarrow X_{i}$ for all $i\geq m$. 
In particular, those conditions hold for $Y\dashrightarrow X_{m}$. 

Pick any lc center $S_{Y}$ of $(Y,\Delta_{Y})$. 
Then there is an lc center $S_{m}$ on $(X_{m},\Delta_{m})$ such that the restriction of $Y\dashrightarrow X_{m}$ to $S_{Y}$ induces a birational map $S_{Y}\dashrightarrow S_{m}$. 
Now the divisor $K_{X_{m}}+\Delta_{m}+\lambda_{m-1}A_{m}$ is nef over $Z$. 
Since $X_{m} \dashrightarrow Y$ is small, we can apply Lemma \ref{lem--comparedim} to $X_{m}\dashrightarrow Y$ over $Z$, $S_{m}\dashrightarrow S_{Y}$, $K_{X_{m}}+\Delta_{m}+\lambda_{m-1}A_{m}$, and $K_{Y}+\Delta_{Y}+tA_{Y}$. 
Since $\lambda_{m-1}>\epsilon\geq t>0$, we have 
\begin{equation*}
\begin{split}
\kappa_{\iota}(S_{m}/Z, (K_{X_{m}}+\Delta_{m})|_{S_{m}})&\geq \kappa_{\iota}(S_{Y}/Z, (K_{Y}+\Delta_{Y})|_{S_{Y}})\quad {\rm and}\\ 
\kappa_{\sigma}(S_{m}/Z, (K_{X_{m}}+\Delta_{m})|_{S_{m}})&\geq \kappa_{\sigma}(S_{Y}/Z, (K_{Y}+\Delta_{Y})|_{S_{Y}}). 
\end{split}
\end{equation*}
Let $l$ be the minimum index such that $K_{X_{l}}+\Delta_{l}+\frac{t}{2}A_{l}$ is nef over $Z$.
Then $l\geq m$ since $\lambda_{m-1}>\frac{t}{2}$, and the birational map $(X,\Delta+\frac{t}{2}A)\dashrightarrow (X_{l},\Delta_{l}+\frac{t}{2}A_{l})$ is a sequence of steps of a $(K_{X}+\Delta+\frac{t}{2}A)$-MMP over $Z$ to a good minimal model. 
By the discussion of the previous paragraph, $Y\dashrightarrow X_{l}$ is small and there is an lc center $S_{l}$ on $(X_{l},\Delta_{l})$ such that the restriction of $Y\dashrightarrow X_{l}$ induces a birational map $S_{Y}\dashrightarrow S_{l}$. 
Applying Lemma \ref{lem--comparedim} to $Y\dashrightarrow X_{l}$ over $Z$, $S_{Y} \dashrightarrow S_{l}$, $K_{Y}+\Delta_{Y}+t A_{Y}$, and $K_{X_{l}}+\Delta_{l}+\frac{t}{2} A_{l}$, we obtain 
\begin{equation*}
\begin{split}
\kappa_{\iota}(S_{Y}/Z, (K_{Y}+\Delta_{Y})|_{S_{Y}})&\geq \kappa_{\iota}(S_{l}/Z, (K_{X_{l}}+\Delta_{l})|_{S_{l}})\quad {\rm and}\\ 
\kappa_{\sigma}(S_{Y}/Z, (K_{Y}+\Delta_{Y})|_{S_{Y}})&\geq \kappa_{\sigma}(S_{l}/Z, (K_{X_{l}}+\Delta_{l})|_{S_{l}}). 
\end{split}
\end{equation*}
By construction, the birational map $X_{m}\dashrightarrow X_{l}$ induces a birational map $S_{m}\dashrightarrow S_{l}$. 
By the above inequalities and the third condition of Step \ref{step1--abund-inv} for $S_{m}\dashrightarrow S_{l}$, we obtain
\begin{equation*}
\begin{split}
\kappa_{\iota}(S_{m}/Z, (K_{X_{m}}+\Delta_{m})|_{S_{m}})\geq&\kappa_{\iota}(S_{Y}/Z, (K_{Y}+\Delta_{Y})|_{S_{Y}})\\
\geq& \kappa_{\iota}(S_{l}/Z, (K_{X_{l}}+\Delta_{l})|_{S_{l}})\\
=&\kappa_{\iota}(S_{m}/Z, (K_{X_{m}}+\Delta_{m})|_{S_{m}})
\end{split}
\end{equation*}
for any lc center $S_{Y}$, and the same relation holds between $\kappa_{\sigma}(S_{Y}/Z, (K_{Y}+\Delta_{Y})|_{S_{Y}})$ and $\kappa_{\sigma}(S_{m}/Z, (K_{X_{m}}+\Delta_{m})|_{S_{m}})$. 
In this way, the third condition of Lemma \ref{lem--abund-inv} holds for any lc center $S_{Y}$ with the birational map $S_{Y}\dashrightarrow S_{m}$ induced by 
$Y\dashrightarrow X_{m}$. 

From this discussion, $Y\dashrightarrow X_{m}$ satisfies the three properties of Lemma \ref{lem--abund-inv}. 
\end{step3}

Let $t,t'\in(0,\epsilon]$, $(X,\Delta+t A)\dashrightarrow (Y,\Delta_{Y}+tA_{Y})$, and $(X,\Delta+t' A)\dashrightarrow (Y',\Delta_{Y'}+t'A_{Y'})$ be as in Lemma \ref{lem--abund-inv}. 
Then the induced birational map $Y \dashrightarrow Y'$ is the composition of $Y\dashrightarrow X_{m}$ and $X_{m} \dashrightarrow Y'$. 
By Step \ref{step2--abund-inv}, it is easy to check that $Y\dashrightarrow Y'$ satisfies the three conditions of Lemma \ref{lem--abund-inv}. 
Therefore, the real number $\epsilon>0$ defined in Step \ref{step2--abund-inv} satisfies the condition of Lemma \ref{lem--abund-inv}. 
\end{proof}

\section{Log MMP containing log abundant log canonical pairs}\label{sec3}

In this section, we study log MMP with scaling containing log abundant lc pairs. 

\subsection{Proofs of main results}\label{subsec--mainresult}
The goal of this subsection is to prove Theorem \ref{thm--main-logabundant} and Theorem \ref{thm--abundantterminate}. 
We start with recalling the projective case of \cite[Proposition 3.3]{has-mmp}. 

\begin{lem}[cf.~{\cite[Proposition 3.3]{has-mmp}}]\label{lem--fiberspacemmp}
Let $(X,\Delta)$ be a projective lc pair. 
Suppose that there is a contraction $X \to V$ to a normal projective variety $V$ such that
\begin{itemize}
\item
$\kappa_{\sigma}(X/V, K_{X}+\Delta)\!=0$, 
\item
$\kappa_{\sigma}(X, K_{X}+\Delta)={\rm dim}\,V$, and 
\item
all lc centers of $(X,\Delta)$ dominate $V$. 
\end{itemize}
Then, $(X,\Delta)$ has a good minimal model. 
\end{lem}
Next, we prove three results using Lemma \ref{lem--fiberspacemmp} and the ideas from \cite{hashizumehu}.  

\begin{prop}\label{prop--crepantmmp}
Let $(X,\Delta)$ be a projective lc pair such that $K_{X}+\Delta$ is pseudo-effective and abundant. 
Then there is a dlt blow-up $\widetilde{f}\colon(\widetilde{X},\widetilde{\Delta})\to (X,\Delta)$ and effective $\mathbb{R}$-divisors $\widetilde{G}$ and $\widetilde{H}$ on $\widetilde{X}$ satisfying the following properties.
\makeatletter 
\renewcommand{\p@enumii}{III-} 
\makeatother
\begin{enumerate}[(I)]
\item \label{prop--crepantmmp-(I)}
$K_{\widetilde{X}}+\widetilde{\Delta}\sim_{\mathbb{R}}\widetilde{G}+\widetilde{H}$, 
\item \label{prop--crepantmmp-(II)}
${\rm Supp}\,\widetilde{G}\subset {\rm Supp}\,\llcorner \widetilde{\Delta} \lrcorner$, and
\item \label{prop--crepantmmp-(III)}
there exists a positive real number $t_{0}$ such that for any $t\in(0,t_{0}]$, the following properties hold: 
\begin{enumerate}[({III-}a)]
\item \label{prop--crepantmmp-(III-a)}
The pair $(\widetilde{X},\widetilde{\Delta}+t\widetilde{H})$ is $\mathbb{Q}$-factorial dlt and the support of $N_{\sigma}(K_{\widetilde{X}}+\widetilde{\Delta}+t\widetilde{H})$ does not depend on $t$, and 
\item \label{prop--crepantmmp-(III-b)}
the pair $(\widetilde{X},\widetilde{\Delta}-t\widetilde{G})$ has a good minimal model. 
\end{enumerate}
\end{enumerate}
\end{prop} 
 
\begin{proof}
We apply arguments in \cite[Steps 2--3 in the proof of Theorem 1.1]{hashizumehu}. 
Since $K_{X}+\Delta$ is pseudo-effective and abundant, we can find an effective $\mathbb{R}$-divisor $D$ on $X$ such that $K_{X}+\Delta\sim_{\mathbb{R}}D$. 
We take the Iitaka fibration $X\dashrightarrow V$ associated to $D$. 
Then we have ${\rm dim}\,V=\kappa_{\sigma}(X, K_{X}+\Delta)$. 
We take a log resolution $\bar{f}\colon \bar{X}\to X$ of $(X,\Delta)$ such that the induced map $\bar{X}\dashrightarrow V$ is a morphism. 
Let $(\bar{X},\bar{\Delta})$ be a log smooth model of $(X,\Delta)$ as in \cite[Definition 2.9]{has-trivial}. 
Then $(\bar{X},\bar{\Delta})$ is a log smooth lc pair such that 
\begin{enumerate}[(i)]
\item \label{proof-prop--creppantmmp-(i)}
$K_{\bar{X}}+\bar{\Delta}=\bar{f}^{*}(K_{X}+\Delta)+\bar{E}$ for an $\bar{f}$-exceptional $\mathbb{R}$-divisor $\bar{E}\geq0$ whose support contains all $\bar{f}$-exceptional prime divisors $\bar{P}$ satisfying $a(\bar{P},X,\Delta)>-1$. 
\end{enumerate}
By Remark \ref{rem--div} (\ref{rem--div-(2)}) and Lemma \ref{lem--iitakafib}, we have 
\begin{enumerate}[(i)]\setcounter{enumi}{1}
\item \label{proof-prop--creppantmmp-(ii)}
$\kappa_{\sigma}(\bar{X},K_{\bar{X}}+\bar{\Delta})={\rm dim}\,V$ and $\kappa_{\sigma}(\bar{X}/V,K_{\bar{X}}+\bar{\Delta})=0$.
\end{enumerate}
By (\ref{proof-prop--creppantmmp-(i)}) and the construction of the Iitaka fibration, $K_{\bar{X}}+\bar{\Delta}$ is $\mathbb{R}$-linearly equivalent to the sum of an effective $\mathbb{R}$-divisor on $\bar{X}$ and the pullback of an ample $\mathbb{R}$-divisor on $V$. 
Therefore, we can find an effective $\mathbb{R}$-divisor $\bar{D}\sim_{\mathbb{R}}K_{\bar{X}}+\bar{\Delta}$ such that ${\rm Supp}\,\bar{D}$ contains all lc centers of $(\bar{X},\bar{\Delta})$ which are vertical over $V$. 
By applying \cite[Lemma 2.10]{has-trivial} to the morphism $(\bar{X},\bar{\Delta})\to V$, we may assume that
\begin{enumerate}[(i)]\setcounter{enumi}{2}
\item \label{proof-prop--creppantmmp-(iii)}
$\bar{\Delta}=\bar{\Delta}'+\bar{\Delta}''$ such that $\bar{\Delta}'$ is effective, $\bar{\Delta}''=0$ or $\bar{\Delta}''$ is a reduced divisor which is vertical over $V$, and all lc centers of $(\bar{X},\bar{\Delta}-\bar{\Delta}'')$ dominate $V$. 
\end{enumerate}
By taking a log resolution of $(\bar{X},\bar{\Delta}+\bar{D})$ and replacing $(\bar{X},\bar{\Delta})$ with a log smooth model in \cite[Definition 2.9]{has-trivial} (and replacing $\bar{D}$ accordingly), we may assume that $(\bar{X},\bar{\Delta}+\bar{D})$ is log smooth.
Note that (\ref{proof-prop--creppantmmp-(iii)}) is preserved after this replacement.  
We have ${\rm Supp}\,\bar{D}\supset {\rm Supp}\,\bar{\Delta}''$ since ${\rm Supp}\,\bar{D}$ contains all lc centers of $(\bar{X},\bar{\Delta})$ which are vertical over $V$. 
By decomposing $\bar{D}$ appropriately, we obtain effective $\mathbb{R}$-divisors $\bar{G}$ and $\bar{H}$ such that 
\begin{enumerate}[(i)]\setcounter{enumi}{3}
\item \label{proof-prop--creppantmmp-(iv)}
$K_{\bar{X}}+\bar{\Delta}\sim_{\mathbb{R}}\bar{G}+\bar{H}$, 
\item \label{proof-prop--creppantmmp-(v)}
${\rm Supp}\,\bar{\Delta}''\subset {\rm Supp}\,\bar{G}\subset {\rm Supp}\,\llcorner \bar{\Delta} \lrcorner$, and 
\item \label{proof-prop--creppantmmp-(vi)}
no component of $\bar{H}$ is a component of $\llcorner \bar{\Delta} \lrcorner$ and $(\bar{X},{\rm Supp}\,(\bar{\Delta}+\bar{H}))$ is log smooth.   
\end{enumerate}
We fix a real number $t_{0}\in(0,1)$ such that $\bar{\Delta}-t_{0}\bar{G}\geq 0$. 
For any $t\in(0,t_{0}]$, we consider the pair $(\bar{X},\bar{\Delta}-t\bar{G})$.
By (\ref{proof-prop--creppantmmp-(v)}), there is a real number $t'>0$ such that $t'\bar{\Delta}''\leq t\bar{G}$. 
Then all lc centers of $(\bar{X},\bar{\Delta}-t\bar{G})$ are lc centers of $(\bar{X},\bar{\Delta}-t'\bar{\Delta}'')$. 
Since $(\bar{X},\bar{\Delta})$ is lc, all lc centers of $(\bar{X},\bar{\Delta}-t'\bar{\Delta}'')$ are lc centers of $(\bar{X},\bar{\Delta}-\bar{\Delta}'')$. 
By these facts and (\ref{proof-prop--creppantmmp-(iii)}), we see that all lc centers of $(\bar{X},\bar{\Delta}-t\bar{G})$ dominate $V$. 
Since we have 
$$K_{\bar{X}}+\bar{\Delta}\sim_{\mathbb{R}}\bar{G}+\bar{H} \quad\text{and} \quad K_{\bar{X}}+\bar{\Delta}-t\bar{G}\sim_{\mathbb{R}}(1-t)\bar{G}+\bar{H},$$ 
by Remark \ref{rem--div} (\ref{rem--div-(1)}) we have $\kappa_{\sigma}(\bar{X},K_{\bar{X}}+\bar{\Delta}-t\bar{G})=\kappa_{\sigma}(\bar{X},K_{\bar{X}}+\bar{\Delta}).$ 
By  (\ref{proof-prop--creppantmmp-(ii)}), we obtain 
$$\kappa_{\sigma}(\bar{X},K_{\bar{X}}+\bar{\Delta}-t\bar{G})={\rm dim}\,V.$$ 
Similarly, we obtain $\kappa_{\sigma}(\bar{X}/V,K_{\bar{X}}+\bar{\Delta}-t\bar{G})=0$. 
Thus, the morphism $(\bar{X},\bar{\Delta}-t\bar{G})\to V$ satisfies the conditions of Lemma \ref{lem--fiberspacemmp}. 
By Lemma \ref{lem--fiberspacemmp}, the pair $(\bar{X},\bar{\Delta}-t\bar{G})$ has a good minimal model for any $t\in(0,t_{0}]$. 

We run a $(K_{\bar{X}}+\bar{\Delta})$-MMP over $X$. 
By construction of $\bar{\Delta}$ (see (\ref{proof-prop--creppantmmp-(i)})), we get a dlt blow-up $\widetilde{f}\colon(\widetilde{X},\widetilde{\Delta})\to (X,\Delta)$.  
Let $\widetilde{G}$ (resp.~$\widetilde{H}$) be the birational transform of $\bar{G}$ (resp.~$\bar{H}$) on $\widetilde{X}$. 
We show that the dlt blow-up $\widetilde{f}\colon(\widetilde{X},\widetilde{\Delta})\to (X,\Delta)$ and the $\mathbb{R}$-divisors $\widetilde{G}$ and $\widetilde{H}$ satisfy all the conditions of Proposition \ref{prop--crepantmmp}. 
By (\ref{proof-prop--creppantmmp-(iv)}), we have the relation $K_{\widetilde{X}}+\widetilde{\Delta}\sim_{\mathbb{R}}\widetilde{G}+\widetilde{H}$, so (\ref{prop--crepantmmp-(I)}) of Proposition \ref{prop--crepantmmp} holds. 
Condition (\ref{proof-prop--creppantmmp-(v)}) shows (\ref{prop--crepantmmp-(II)}) of Proposition \ref{prop--crepantmmp}. 
By (\ref{proof-prop--creppantmmp-(vi)}), replacing  $t_{0}$ if necessary, we may assume that $(\bar{X},\bar{\Delta}+t_{0}\bar{H})$ is dlt. 
Since $\bar{X}\dashrightarrow \widetilde{X}$ is a sequence of steps of a $(K_{\bar{X}}+\bar{\Delta})$-MMP, replacing $t_{0}$ again, we may assume the following.
\begin{itemize}
\item
$\bar{X}\dashrightarrow \widetilde{X}$ is a sequence of steps of a $(K_{\bar{X}}+\bar{\Delta}+t\bar{H})$-MMP and a sequence of steps of a $(K_{\bar{X}}+\bar{\Delta}-t\bar{G})$-MMP for any $t\in(0,t_{0}]$. 
\end{itemize}
By this statement, $(\widetilde{X},\widetilde{\Delta}-t\widetilde{G})$ has a good minimal model for all $t\in(0,t_{0}]$ because $(\bar{X},\bar{\Delta}-t\bar{G})$ has a good minimal model. In particular, (\ref{prop--crepantmmp-(III-b)}) of Proposition \ref{prop--crepantmmp} holds. 
The statement and the $\mathbb{Q}$-factorial dlt property of $(\bar{X},\bar{\Delta}+t_{0}\bar{H})$ imply the $\mathbb{Q}$-factorial dlt property of $(\widetilde{X},\widetilde{\Delta}+t_{0}\widetilde{H})$. 
Applying Lemma \ref{lem--negativesupport} to $K_{\widetilde{X}}+\widetilde{\Delta}$ and $\widetilde{H}$ and replacing $t_{0}$ again, we may assume that the support of $N_{\sigma}(K_{\widetilde{X}}+\widetilde{\Delta}+t\widetilde{H})$ does not depend on $t\in (0,t_{0}]$. 
By this fact and the fact that $(\widetilde{X},\widetilde{\Delta}+t_{0}\widetilde{H})$ is $\mathbb{Q}$-factorial dlt, $t_{0}$ satisfies (\ref{prop--crepantmmp-(III-a)}) of Proposition \ref{prop--crepantmmp}. Thus $t_{0}$ satisfies (\ref{prop--crepantmmp-(III)}) of Proposition \ref{prop--crepantmmp}.  

Thus, $\widetilde{f}\colon(\widetilde{X},\widetilde{\Delta})\to (X,\Delta)$, $\widetilde{G}$, and $\widetilde{H}$ satisfy all conditions of Proposition \ref{prop--crepantmmp}. 
\end{proof} 
 
\begin{prop}\label{prop--specialmmp}
Let $(X,\Delta)$ be a projective $\mathbb{Q}$-factorial dlt pair such that there are two effective $\mathbb{R}$-divisors $G$ and $H$ on $X$ satisfying (\ref{prop--crepantmmp-(I)}), (\ref{prop--crepantmmp-(II)}), (\ref{prop--crepantmmp-(III)}), (\ref{prop--crepantmmp-(III-a)}), and (\ref{prop--crepantmmp-(III-b)}) in Proposition \ref{prop--crepantmmp}. 
Then there is a real number $\lambda_{0}>0$ and a sequence of birational maps
\begin{equation*}
X \dashrightarrow X_{1}\dashrightarrow X_{2}\dashrightarrow \cdots \dashrightarrow X_{i}\dashrightarrow \cdots
\end{equation*}
such that putting $\Delta_{i}$ and $H_{i}$ as the birational transforms of $\Delta$ and $H$ on $X_{i}$ respectively, the following properties hold.
\begin{enumerate}[(1)]
\item \label{prop--specialmmp-(1)}
The pair $(X,\Delta+\lambda_{0}H)$ is $\mathbb{Q}$-factorial dlt and all lc centers of $(X,\Delta+\lambda_{0}H)$ are lc centers of $(X,\Delta)$, 
\item \label{prop--specialmmp-(2)}
the birational map $X \dashrightarrow X_{1}$ is a sequence of steps of a $(K_{X}+\Delta+\lambda_{0}H)$-MMP to a good minimal model $(X_{1},\Delta_{1}+\lambda_{0}H_{1})$, 
\item \label{prop--specialmmp-(3)}
the sequence of birational maps $X_{1}\dashrightarrow \cdots \dashrightarrow X_{i}\dashrightarrow \cdots$ is a sequence of steps of a $(K_{X_{1}}+\Delta_{1})$-MMP with scaling of $\lambda_{0}H_{1}$ such that if we define
$$\lambda_{i}={\rm inf}\set{\mu \in \mathbb{R}_{\geq0} \!|\! K_{X_{i}}+\Delta_{i}+\mu H_{i}\text{\rm \;is nef}}$$
for each $i\geq 1$, then we have ${\rm lim}_{i \to \infty}\lambda_{i}=0$, 
\item \label{prop--specialmmp-(4)}
for all $i\geq1$ and all positive real numbers $u\in [\lambda_{i},\lambda_{i-1}]$, the pair $(X_{i},\Delta_{i}+u H_{i})$ is a good minimal model of  $(X,\Delta+u H)$ and $(X_{1},\Delta_{1}+u H_{1})$, and 
\item \label{prop--specialmmp-(5)}
the $(K_{X_{1}}+\Delta_{1})$-MMP with scaling of $\lambda_{0}H_{1}$ discussed in (\ref{prop--specialmmp-(3)}) occurs only in ${\rm Supp}\,\llcorner \Delta_{1}\lrcorner$, in other words, for every $i\geq1$, any curve contracted by the extremal contraction in the $i$-th step of the $(K_{X_{1}}+\Delta_{1})$-MMP is contained in ${\rm Supp}\,\llcorner \Delta_{i}\lrcorner$. 
\end{enumerate}
\end{prop}
 
\begin{proof}
Throughout this proof, unless otherwise stated, (\ref{prop--crepantmmp-(I)}), (\ref{prop--crepantmmp-(II)}), (\ref{prop--crepantmmp-(III)}), (\ref{prop--crepantmmp-(III-a)}), and (\ref{prop--crepantmmp-(III-b)}) mean the properties (\ref{prop--crepantmmp-(I)}), (\ref{prop--crepantmmp-(II)}), (\ref{prop--crepantmmp-(III)}), (\ref{prop--crepantmmp-(III-a)}), and (\ref{prop--crepantmmp-(III-b)}) in Proposition \ref{prop--crepantmmp}, respectively. 
We follow \cite[Step 4 in the proof of Theorem 1.1]{hashizumehu}. 

By (\ref{prop--crepantmmp-(I)}), for any $s \in \mathbb{R}_{\geq 0}$ we have 
\begin{equation*}\tag{$\spadesuit$}\label{prop--specialmmp-(spadesuit)}
K_{X}+\Delta+s H\sim_{\mathbb{R}}(1+s)\Bigl(K_{X}+\Delta-\frac{s}{1+s}G\Bigr). 
\end{equation*}
Fix $\lambda_{0}\in(0,t_{0})$, where $t_{0}$ is as in (\ref{prop--crepantmmp-(III)}). 
By (\ref{prop--crepantmmp-(III-a)}), the pair $(X,\Delta+\lambda_{0}H)$ is $\mathbb{Q}$-factorial dlt, and we can easily check that lc centers of $(X,\Delta+\lambda_{0}H)$ are lc centers of $(X,\Delta)$. 
Hence (\ref{prop--specialmmp-(1)}) of Proposition \ref{prop--specialmmp} holds. 

By (\ref{prop--crepantmmp-(III-b)}) and \cite[Theorem 4.1]{birkar-flip}, we can run a $(K_{X}+\Delta-\frac{\lambda_{0}}{1+\lambda_{0}}G)$-MMP that terminates with a good minimal model $(X_{1},\Delta_{1}-\frac{\lambda_{0}}{1+\lambda_{0}}G_{1})$, where $\Delta_{1}$ (resp.~$G_{1}$) is the birational transform of $\Delta$ (resp.~$G$) on $X_{1}$. 
Note that $\frac{\lambda_{0}}{1+\lambda_{0}}< t_{0}$, therefore we may use (\ref{prop--crepantmmp-(III-b)}). 
By (\ref{prop--specialmmp-(spadesuit)}), the birational map
\begin{equation*}
(X,\Delta+\lambda_{0}H)\dashrightarrow (X_{1},\Delta_{1}+\lambda_{0}H_{1})
\end{equation*}
is a sequence of steps of a $(K_{X}+\Delta+\lambda_{0}H)$-MMP and $(X_{1},\Delta_{1}+\lambda_{0}H_{1})$ is a good minimal model of $(X,\Delta+\lambda_{0}H)$, where $H_{1}$ is the birational transform of $H$ on $X_{1}$.  
It is obvious that the birational map $X\dashrightarrow X_{1}$ satisfies (\ref{prop--specialmmp-(2)}) of Proposition \ref{prop--specialmmp}. 

By using (\ref{prop--crepantmmp-(III-b)}), \cite[Theorem 4.1]{birkar-flip}, (\ref{prop--specialmmp-(spadesuit)}), and the same argument as above, for every $t\in(0,\lambda_{0})$ we get a sequence of steps of a $(K_{X}+\Delta+tH)$-MMP
\begin{equation*}\tag{$\clubsuit$}\label{prop--specialmmp-(clubsuit)}
(X,\Delta+tH)\dashrightarrow (\tilde{X}_{t},\tilde{\Delta}_{t}+t\tilde{H}_{t})
\end{equation*}
that terminates with a good minimal model $(\tilde{X}_{t},\tilde{\Delta}_{t}+t\tilde{H}_{t})$. 
The condition (\ref{prop--crepantmmp-(III-a)}) and Remark \ref{rem--mmp-zariskidecom} (\ref{rem--mmp-zariskidecom-(1)}) show that $X\dashrightarrow X_{1}$ and $X\dashrightarrow \tilde{X}_{t}$ contract the same divisors for all $t\in(0,\lambda_{0})$, therefore the induced birational map $X_{1}\dashrightarrow \tilde{X}_{t}$ is small. 
Then $(\tilde{X}_{t},\tilde{\Delta}_{t}+t\tilde{H}_{t})$ is a good minimal model of $(X_{1},\Delta_{1}+tH_{1})$ for all $t\in(0,\lambda_{0})$. 
Therefore, we see that $(X_{1},\Delta_{1}+tH_{1})$ has a good minimal model for all $t\in(0,\lambda_{0})$ and $K_{X_{1}}+\Delta_{1}$ is a limit of movable divisors. 
By \cite[Lemma 2.14]{has-mmp}, we get a sequence of steps of a $(K_{X_{1}}+\Delta_{1})$-MMP with scaling of $\lambda_{0}H_{1}$
\begin{equation*}
(X_{1},\Delta_{1})\dashrightarrow (X_{2},\Delta_{2})\dashrightarrow\cdots \dashrightarrow (X_{i},\Delta_{i})\dashrightarrow \cdots
\end{equation*}
such that if we define $\lambda_{i}$ by 
$$\lambda_{i}={\rm inf}\!\set{\!\mu\in\mathbb{R}_{\geq0} | \text{$K_{X_{i}}+\Delta_{i}+\mu H_{i}$ is nef}\!},$$
where $H_{i}$ is the birational transform of $H_{1}$ on $X_{i}$, then the $(K_{X_{1}}+\Delta_{1})$-MMP terminates after finitely many steps or we have ${\rm lim}_{i\to \infty}\lambda_{i}=0$ even if it does not terminate. 
Since $K_{X_{1}}+\Delta_{1}$ is pseudo-effective, in both cases we have ${\rm lim}_{i\to \infty}\lambda_{i}=0$. 
Therefore, we see that the $(K_{X_{1}}+\Delta_{1})$-MMP satisfies (\ref{prop--specialmmp-(3)}) of Proposition \ref{prop--specialmmp}.  
 
We pick any $i\geq 1$ and any positive real number $u\in [\lambda_{i},\lambda_{i-1}]$. 
By construction of the log MMP with scaling, $(X_{i},\Delta_{i}+u H_{i})$ is an lc pair and  $K_{X_{i}}+\Delta_{i}+u H_{i}$ is nef. 
Since $K_{X_{1}}+\Delta_{1}$ is a limit of movable divisors, the $(K_{X_{1}}+\Delta_{1})$-MMP contains only flips. 
Since the birational map $X_{1}\dashrightarrow \tilde{X}_{u}$ is small, where $\tilde{X}_{u}$ is as in (\ref{prop--specialmmp-(clubsuit)}), the induced birational map $X_{i}\dashrightarrow \tilde{X}_{u}$ is small. 
By (\ref{prop--specialmmp-(clubsuit)}), the pair $(\tilde{X}_{u},\tilde{\Delta}_{u}+u \tilde{H}_{u})$ is a good minimal model of $(X,\Delta+u H)$, from which and Remark \ref{rem--models} we see that $(X_{i},\Delta_{i}+u H_{i})$ is a good minimal model of $(X,\Delta+u H)$. 
Then $K_{X_{i}}+\Delta_{i}+u H_{i}$ is semi-ample. 
Since $X_{1}\dashrightarrow X_{i}$ is small, the pair $(X_{i},\Delta_{i}+u H_{i})$ is a good minimal model of $(X_{1},\Delta_{1}+u H_{1})$. 
In this way, we see that (\ref{prop--specialmmp-(4)}) of Proposition \ref{prop--specialmmp} holds. 

By construction of the log MMP with scaling, any curve which is contracted in the $i$-th step of the $(K_{X_{1}}+\Delta_{1})$-MMP with scaling of $\lambda_{0}H_{1}$ has a negative (resp.~positive) intersection number with $K_{X_{i}}+\Delta_{i}$ (resp.~$H_{i}$). 
By (\ref{prop--crepantmmp-(I)}), the relation $G_{i}\sim_{\mathbb{R}}-H_{i}+K_{X_{i}}+\Delta_{i}$ holds, where $G_{i}$ is the birational transform of $G$ on $X_{i}$. 
Therefore, curves contracted in the $i$-th step of the $(K_{X_{1}}+\Delta_{1})$-MMP with scaling of $\lambda_{0}H_{1}$ have negative intersection numbers with $G_{i}$. 
Since $G_{i}\geq0$, the $(K_{X_{1}}+\Delta_{1})$-MMP occurs only in ${\rm Supp}\, G_{1}$. 
Now the inclusion ${\rm Supp}\,G_{1}\subset {\rm Supp}\,\llcorner \Delta_{1} \lrcorner$ holds by (\ref{prop--crepantmmp-(II)}), hence (\ref{prop--specialmmp-(5)}) of Proposition \ref{prop--specialmmp} holds. 
\end{proof}
 
\begin{thm}\label{thm--ind-1}
Let $(X,\Delta)$ be a projective dlt pair. 
Suppose that
\begin{itemize}
\item
$K_{X}+\Delta$ is pseudo-effective and abundant,
\item for any lc center $S$ of $(X,\Delta)$, the restriction $(K_{X}+\Delta)|_{S}$ is nef, and
\item
$\sigma_{P}(K_{X}+\Delta)=0$ for every prime divisor $P$ over $X$ such that $a(P,X,\Delta)< 0$ and $c_{X}(P)$ intersects an lc center of $(X,\Delta)$, where $\sigma_{P}(\,\cdot\,)$ is the asymptotic vanishing order of $P$ defined in Definition \ref{defn--asy-van-ord}. 
\end{itemize}
Then $(X,\Delta)$ has a log minimal model. 
\end{thm}
 
\begin{proof}
We prove the theorem in several steps. 

\begin{step1}\label{step1-ind-1}
First, we reduce Theorem \ref{thm--ind-1} to special termination of a log MMP with scaling. 
The idea is very similar to \cite[Steps 2--6 in the proof of Theorem 1.1]{hashizumehu}. 

By Proposition \ref{prop--crepantmmp}, we get a dlt blow-up $\widetilde{f}\colon(\widetilde{X},\widetilde{\Delta})\to (X,\Delta)$ and effective $\mathbb{R}$-divisors $\widetilde{G}$ and $\widetilde{H}$ on $\widetilde{X}$ satisfying the conditions of Proposition \ref{prop--crepantmmp}. 
Moreover, we may replace $(X,\Delta)$ by $(\widetilde{X},\widetilde{\Delta})$. 
Indeed, after this replacement, the second condition of Theorem \ref{thm--ind-1} is clearly preserved, and the first (resp.~third) condition of Theorem \ref{thm--ind-1} is preserved by Remark \ref{rem--div} (\ref{rem--div-(2)}) (resp.~the definition of $\sigma_{P}(\,\cdot\,)$). 
Therefore, replacing $(X,\Delta)$ by $(\widetilde{X},\widetilde{\Delta})$, we may assume that there are effective $\mathbb{R}$-divisors $G$ and $H$ on $X$ satisfying (\ref{prop--crepantmmp-(I)}), (\ref{prop--crepantmmp-(II)}), (\ref{prop--crepantmmp-(III)}), (\ref{prop--crepantmmp-(III-a)}), and (\ref{prop--crepantmmp-(III-b)}) in Proposition \ref{prop--crepantmmp}. 
By Proposition \ref{prop--specialmmp}, there is a positive real number $\lambda_{0}$ and a sequence of birational maps
\begin{equation*}
X \dashrightarrow X_{1}\dashrightarrow X_{2}\dashrightarrow \cdots \dashrightarrow X_{i}\dashrightarrow \cdots
\end{equation*}
such that
\begin{enumerate}[(1)]
\item \label{proof-thm--ind-1-step1-(1)}
the pair $(X,\Delta+\lambda_{0}H)$ is $\mathbb{Q}$-factorial dlt and all lc centers of $(X,\Delta+\lambda_{0}H)$ are lc centers of $(X,\Delta)$, 
\item \label{proof-thm--ind-1-step1-(2)}
the birational map $X \dashrightarrow X_{1}$ is a sequence of steps of a $(K_{X}+\Delta+\lambda_{0}H)$-MMP to a good minimal model $(X_{1},\Delta_{1}+\lambda_{0}H_{1})$, 
\item \label{proof-thm--ind-1-step1-(3)}
the sequence of birational maps $X_{1}\dashrightarrow \cdots \dashrightarrow X_{i}\dashrightarrow \cdots$ is a sequence of steps of a $(K_{X_{1}}+\Delta_{1})$-MMP with scaling of $\lambda_{0}H_{1}$ such that if we define
$$\lambda_{i}={\rm inf}\set{\mu \in \mathbb{R}_{\geq0} \!|\! K_{X_{i}}+\Delta_{i}+\mu H_{i}\text{\rm \;is nef}}$$
for each $i\geq 1$, then we have ${\rm lim}_{i \to \infty}\lambda_{i}=0$, 
\item \label{proof-thm--ind-1-step1-(4)}
for all $i\geq1$ and all positive real numbers $u\in [\lambda_{i},\lambda_{i-1}]$, the pair $(X_{i},\Delta_{i}+u H_{i})$ is a good minimal model of  $(X,\Delta+u H)$ and $(X_{1},\Delta_{1}+u H_{1})$, and 
\item \label{proof-thm--ind-1-step1-(5)}
the $(K_{X_{1}}+\Delta_{1})$-MMP with scaling of $\lambda_{0}H_{1}$ in (\ref{proof-thm--ind-1-step1-(3)})  occurs only in ${\rm Supp}\,\llcorner \Delta_{1}\lrcorner$. 
\end{enumerate}

Suppose that the $(K_{X_{1}}+\Delta_{1})$-MMP with scaling of $\lambda_{0}H_{1}$ terminates. 
Then we have $\lambda_{l-1}>\lambda_{l}=0$ for some $l>0$.
For all real numbers $u\in (0,\lambda_{l-1}]$, it follows by (\ref{proof-thm--ind-1-step1-(4)}) that $K_{X_{l}}+\Delta_{l}+uH_{l}$ is nef and 
$$a(P',X,\Delta+uH)\leq a(P',X_{l},\Delta_{l}+uH_{l})$$
 for all prime divisors $P'$ on $X$ which are exceptional over $X_{l}$. 
By taking the limit $u\to 0$, it follows that $K_{X_{l}}+\Delta_{l}$ is nef and 
$$a(P',X,\Delta)\leq a(P',X_{l},\Delta_{l}).$$ 
Therefore $(X_{l},\Delta_{l})$ is a weak lc model of $(X,\Delta)$. 
By \cite[Corollary 3.7]{birkar-flip}, the pair $(X,\Delta)$ has a log minimal model. 

In this way, to prove the existence of a log minimal model of $(X,\Delta)$, it is sufficient to prove the termination of the $(K_{X_{1}}+\Delta_{1})$-MMP. 
By (\ref{proof-thm--ind-1-step1-(5)}), we only need to prove that the non-isomorphic locus of the $(K_{X_{1}}+\Delta_{1})$-MMP is disjoint from the birational transform of $\llcorner \Delta_{1}\lrcorner$ after finitely many steps. 
In the rest of the proof, we will carry out the arguments of the special termination in \cite{fujino-sp-ter}.  

There exists $m>0$ such that for any lc center $S_{m}$ of $(X_{m},\Delta_{m})$ and  any $i \geq m$, the birational map $X_{m}\dashrightarrow X_{i}$ induces a birational map $S_{m}\dashrightarrow S_{i}$ to an lc center $S_{i}$ of $(X_{i},\Delta_{i})$. 
For any $i\geq m$ and any lc center $S_{i}$ of $(X_{i},\Delta_{i})$, we define $\mathbb{R}$-divisors $\Delta_{S_{i}}$ and $H_{S_{i}}$ on $S_{i}$ by 
$$K_{S_{i}}+\Delta_{S_{i}}=(K_{X_{i}}+\Delta_{i})|_{S_{i}} \qquad {\rm and} \qquad H_{S_{i}}=H_{i}|_{S_{i}},$$
 respectively. Similarly, on every lc center $S$ of $(X,\Delta)$, we define $\mathbb{R}$-divisors $\Delta_{S}$ and $H_{S}$ by 
 $$K_{S}+\Delta_{S}=(K_{X}+\Delta)|_{S} \qquad {\rm and} \qquad H_{S}=H|_{S},$$ respectively. Then $H_{S}$ and $H_{S_{i}}$ are effective $\mathbb{R}$-Cartier divisors, and $(S_{i},\Delta_{S_{i}})$ is dlt for all $i\geq m$ and all $S_{i}$. 

From now on, we prove that for any integer $d \geq0$, there is $m_{d}\geq m$ such that  after $m_{d}$ steps the non-isomorphic locus of the $(K_{X_{1}}+\Delta_{1})$-MMP is disjoint from all $d$-dimensional lc centers of $(X_{m_{d}},\Delta_{m_{d}})$. 
The case $d={\rm dim}\,X-1$ of this statement and (\ref{proof-thm--ind-1-step1-(5)}) imply the termination of the $(K_{X_{1}}+\Delta_{1})$-MMP. 
Therefore, to prove Theorem \ref{thm--ind-1} it is sufficient to prove the statement.  We will prove the statement by induction on $d$. 
In the case where $d=0$, we can take $m$ as $m_{0}$. 
We assume the existence of $m_{d-1}$ in the statement. 
By replacing $m$ by $m_{d-1}$, we may assume $m_{d-1}=m$. 
By the arguments as in \cite{fujino-sp-ter}, replacing $m$ again, we may assume that for any $d$-dimensional lc center $S_{m}$ of $(X_{m},\Delta_{m})$, the induced birational map $S_{m}\dashrightarrow S_{i}$ is small and the birational transform of $\Delta_{S_{m}}$ on $S_{i}$ is equal to $\Delta_{S_{i}}$. 
Then $H_{S_{i}}$ is equal to the birational transform of $H_{S_{m}}$ on $S_{i}$. 
Since ${\rm lim}_{i \to \infty}\lambda_{i}=0$, as in Remark \ref{rem--mmplift}, to prove the existence of $m_{d}$, it is sufficient to prove that $(S_{m},\Delta_{S_{m}})$ has a log minimal model. 
\end{step1}

\begin{step1}\label{step2-ind-1}
In the rest of the proof, all indices $i$ are assumed to be greater than or equal to $m$, unless otherwise stated. 
We will show that for any $d$-dimensional lc center $S_{m}$ of $(X_{m},\Delta_{m})$, the pair $(S_{m},\Delta_{S_{m}})$ has a log minimal model, as required. The idea is very similar to \cite[Steps 7--9 in the proof of Theorem 1.1]{hashizumehu}.

Fix an arbitrary $d$-dimensional lc center $S_{m}$ of $(X_{m},\Delta_{m})$. 
By Step \ref{step1-ind-1}, the birational map $S_{m}\dashrightarrow S_{i}$ is small and $\Delta_{S_{i}}$ is equal to the birational transform of $\Delta_{S_{m}}$ on $S_{i}$ for all $i$. 
By construction of the birational map $X\dashrightarrow X_{i}$ in (\ref{proof-thm--ind-1-step1-(1)}), (\ref{proof-thm--ind-1-step1-(2)}), and (\ref{proof-thm--ind-1-step1-(3)}) in Step \ref{step1-ind-1}, we can find an lc center $S$ of $(X,\Delta)$ such that $X\dashrightarrow X_{i}$ induces a birational map $S\dashrightarrow S_{i}$. 
Using the lc center $S$, in the rest of this step, we establish the following variety, divisor, and inequalities. 

\begin{enumerate}[(a)]
\item \label{proof-thm--ind-1-step3-(a)}
A projective $\mathbb{Q}$-factorial variety $T$ together with a birational morphism $\psi\colon T\to S_{m}$ such that all prime divisors $\bar{D}$ on $S$ satisfying $a(\bar{D},S_{m}, \Delta_{S_{m}})<a(\bar{D},S,\Delta_{S})$ appear as prime divisors on $T$, 
\item \label{proof-thm--ind-1-step3-(b)}
an effective $\mathbb{R}$-divisor $\Psi$ on $T$ defined by $\Psi=-\sum_{\substack {D}}a(D,S,\Delta_{S})D$, where $D$ runs over all prime divisors on $T$, 
\item \label{proof-thm--ind-1-step3-(c)}
the inequality $a(Q,S,\Delta_{S}+\lambda_{i}H_{S})\leq a(Q,S_{i},\Delta_{S_{i}}+\lambda_{i}H_{S_{i}})$ for all $i$ and all prime divisors $Q$ over $S$, and 
\item \label{proof-thm--ind-1-step3-(d)}
the inequality $a(Q',S_{m},\Delta_{S_{m}})\leq a(Q',T,\Psi)$ for all prime divisors $Q'$ over $S_{m}$, in particular, the pair $(T,\Psi)$ is lc. 
\end{enumerate}

By taking a resolution $\overline{X}\to X$ of $X$ which resolves the indeterminacy of $X\dashrightarrow X_{i}$, we can construct a common resolution $\overline{X}\to X$ and $\overline{X}\to {X_{i}}$ and a subvariety $\overline{S}\subset \overline{X}$ birational to $S$ and $S_{i}$ such that the induced birational morphisms $\overline{S}\to S$ and $\overline{S}\to S_{i}$ form a common resolution of the birational map $S\dashrightarrow S_{i}$. 
By (\ref{proof-thm--ind-1-step1-(4)}) in Step \ref{step1-ind-1}, comparing coefficients of $(K_{X}+\Delta+\lambda_{i}H)|_{\overline{S}}$ and $(K_{X_{i}}+\Delta_{i}+\lambda_{i}H_{i})|_{\overline{S}}$, the inequality
\begin{equation*}
a(Q,S,\Delta_{S}+\lambda_{i}H_{S})\leq a(Q,S_{i},\Delta_{S_{i}}+\lambda_{i}H_{S_{i}})
\end{equation*}
follows for all prime divisors $Q$ over $S$. 
We have proved (\ref{proof-thm--ind-1-step3-(c)}). 

Let $\widetilde{D}$ be a prime divisor on $S_{m}$. 
Then $\widetilde{D}$ is a prime divisor on $S_{i}$ for any $i$ because the birational map $S_{m}\dashrightarrow S_{i}$ is small. 
For all $i$, we can easily check that the birational map $(X,\Delta)\dashrightarrow(X_{i},\Delta_{i})$ and dlt pairs $(S,\Delta_{S})$ and $(S_{i},\Delta_{S_{i}})$ satisfy the two conditions of Lemma \ref{lem--discre-relation}. 
Indeed, the birational map $X\dashrightarrow X_{i}$ is a birational contraction and $\Delta_{i}$ is the birational transform of $\Delta$ on $X_{i}$, so the first condition of Lemma \ref{lem--discre-relation} holds. 
Since $S$ is an lc center of $(X,\Delta)$, the second condition of Lemma \ref{lem--discre-relation} follows from the third condition of Theorem \ref{thm--ind-1}. 
In this way, we can apply Lemma \ref{lem--discre-relation} to $(X,\Delta)\dashrightarrow(X_{i},\Delta_{i})$, $(S,\Delta_{S})$, and $(S_{i},\Delta_{S_{i}})$. 
Thus, we have 
$$a(\widetilde{D},S_{i},\Delta_{S_{i}}) \leq a(\widetilde{D},S,\Delta_{S})$$
 for all $i$. 
Furthermore, the inequality $a(\widetilde{D},S_{i},\Delta_{S_{i}}+\lambda_{i}H_{S_{i}})\leq a(\widetilde{D},S_{i},\Delta_{S_{i}})$ also holds by a basic property of discrepancies. 
Combining these inequalities with (\ref{proof-thm--ind-1-step3-(c)}), we have
\begin{equation*}
a(\widetilde{D},S,\Delta_{S}+\lambda_{i}H_{S})\leq a(\widetilde{D},S_{i},\Delta_{S_{i}}+\lambda_{i}H_{S_{i}})\leq a(\widetilde{D},S,\Delta_{S})
\end{equation*} 
for all $i$. 
Since $\Delta_{S_{i}}+\lambda_{i}H_{S_{i}}$ is equal to the birational transform of $\Delta_{S_{m}}+\lambda_{i}H_{S_{m}}$ on $S_{i}$, we see that $a(\widetilde{D},S_{i},\Delta_{S_{i}}+\lambda_{i}H_{S_{i}})=a(\widetilde{D},S_{m},\Delta_{S_{m}}+\lambda_{i}H_{S_{m}})$ for all $i$. 
Since ${\rm lim}_{i \to \infty}\lambda_{i}=0$ by (\ref{proof-thm--ind-1-step1-(3)}) in Step \ref{step1-ind-1}, by taking the limit $i\to \infty$, the equality
\begin{equation*}\tag{$\star$}\label{proof-thm--ind-1-(starstar)}
a(\widetilde{D},S_{m},\Delta_{S_{m}})= a(\widetilde{D},S,\Delta_{S})
\end{equation*} 
holds for all prime divisors $\widetilde{D}$ on $S_{m}$. 

We set 
\begin{equation*} 
\mathcal{C}=\Set{ \bar{D} | \begin{array}{l}\!\!\text{$\bar{D}$ is a prime divisor on $S$ such that $a(\bar{D},S_{m}, \Delta_{S_{m}})<a(\bar{D},S,\Delta_{S})$} \end{array}\!\!}. \end{equation*} 
By (\ref{proof-thm--ind-1-(starstar)}), all elements of $\mathcal{C}$ are exceptional over $S_{m}$.  
By a basic property of discrepancies, we have $a(\bar{D},S_{m}, \Delta_{S_{m}}+\lambda_{m}H_{S_{m}}) \leq a(\bar{D},S_{m}, \Delta_{S_{m}})$. 
Combining this with (\ref{proof-thm--ind-1-step3-(c)}), we obtain
$$a(\bar{D},S,\Delta_{S}+\lambda_{m}H_{S}) \leq a(\bar{D},S_{m}, \Delta_{S_{m}})< a(\bar{D},S,\Delta_{S})\leq 0$$
for all $\bar{D}\in \mathcal{C}$. 
Since every element of $\mathcal{C}$ is a prime divisor on $S$, we see that all elements of $\mathcal{C}$ are components of $H_{S}$. 
Thus $ \mathcal{C}$ is a finite set, and furthermore, any $\bar{D}\in \mathcal{C}$ satisfies $-1<a(\bar{D},S,\Delta_{S}+\lambda_{m}H_{S})$. 
Thus, the relation 
$$-1<a(\bar{D},S_{m}, \Delta_{S_{m}})< 0$$ 
holds for all $\bar{D}\in \mathcal{C}$. 
Since $(S_{m}, \Delta_{S_{m}})$ is dlt, by \cite[Lemma 2.16]{hashizumehu}, we get a projective birational morphism $\psi\colon T\to S_{m}$ such that $T$ is $\mathbb{Q}$-factorial and $\psi^{-1}$ extracts exactly elements of $\mathcal{C}$. 
We have constructed the desired variety as in (\ref{proof-thm--ind-1-step3-(a)}). 

Let $D$ be a prime divisor on $T$. 
When $D$ is $\psi$-exceptional, by definitions of $\mathcal{C}$ and $\psi$ we obtain $a(D,S_{m}, \Delta_{S_{m}})<a(D,S,\Delta_{S})\leq 0$. 
When $D$ is not $\psi$-exceptional, by (\ref{proof-thm--ind-1-(starstar)}) we obtain $a(D,S,\Delta_{S})= a(D,S_{m},\Delta_{S_{m}})\leq 0$. 
In any case, the relation
\begin{equation*}\tag{$\star$$\star$}\label{proof-thm--ind-1-(starstarstar)}
a(D,S_{m}, \Delta_{S_{m}})\leq a(D,S,\Delta_{S})\leq 0
\end{equation*} 
holds. 
Furthermore, only finitely many prime divisors $D$ on $T$ satisfy $a(D,S,\Delta_{S})<0$ because $a(D,S,\Delta_{S})=0$ when $D$ is a prime divisor on $S$ and not a component of $\Delta_{S}$. 
Therefore, we may define an $\mathbb{R}$-divisor $\Psi\geq0$ on $T$ by
$$\Psi=-\sum_{\substack {D}}a(D,S,\Delta_{S})D,$$
where $D$ runs over all prime divisors on $T$. 
This is the $\mathbb{R}$-divisor stated in (\ref{proof-thm--ind-1-step3-(b)}). 

Finally, we prove (\ref{proof-thm--ind-1-step3-(d)}). 
Since $T$ is $\mathbb{Q}$-factorial, $K_{T}+\Psi$ is $\mathbb{R}$-Cartier. 
By (\ref{proof-thm--ind-1-(starstarstar)}), we obtain $K_{T}+\Psi\leq \psi^{*}(K_{S_{m}}+\Delta_{S_{m}})$. 
From this, we have
\begin{equation*}
-1\leq a(Q',S_{m},\Delta_{S_{m}})\leq a(Q',T,\Psi) 
\end{equation*}
for any prime divisor $Q'$ over $S_{m}$. 
This shows (\ref{proof-thm--ind-1-step3-(d)}). 
\end{step1}

\begin{step1}\label{step3-ind-1}
In this step, we prove that $(T, \Psi)$ has a log minimal model. 

By the second condition of Theorem \ref{thm--ind-1}, $(S,\Delta_{S})$ is a log minimal model of $(S,\Delta_{S})$ itself. 
To apply Lemma \ref{lem--bir-relation} to the birational map $(S,\Delta_{S})\dashrightarrow (T, \Psi)$, we will prove the following claim. 
\hypertarget{proof-thm--ind-1-step3-claim}{\begin{claim*}
Let $\tilde{Q}$ be a prime divisor over $S$.
Then the following two statements hold:
\begin{itemize}
\item
If $\tilde{Q}$ is a prime divisor on $S$, then $a(\tilde{Q},S,\Delta_{S})\leq a(\tilde{Q},T,\Psi)$, and
\item
if $\tilde{Q}$ is a prime divisor on $T$, then $a(\tilde{Q},T,\Psi)\leq a(\tilde{Q},S,\Delta_{S})$. 
\end{itemize}
\end{claim*}}

\begin{proof}[Proof of Claim]
Let $\tilde{Q}$ be a prime divisor over $S$. 
When $\tilde{Q}$ is a prime divisor on $T$ or a prime divisor on $S$ which is not exceptional over $T$, by (\ref{proof-thm--ind-1-step3-(b)}) in Step \ref{step2-ind-1} we have $$a(\tilde{Q},T,\Psi)=-{\rm coeff}_{\tilde{Q}}(\Psi)=a(\tilde{Q},S,\Delta_{S}).$$ 
Therefore, the second assertion of the claim holds. 
Thus, to prove the first assertion of the claim, we may assume that $\tilde{Q}$ is a prime divisor on $S$ that is exceptional over $T$. 
In this case, (\ref{proof-thm--ind-1-step3-(a)}) in Step \ref{step2-ind-1} shows $a(\tilde{Q},S,\Delta_{S})\leq a(\tilde{Q},S_{m},\Delta_{m})$, then (\ref{proof-thm--ind-1-step3-(d)}) in Step \ref{step2-ind-1} shows
$$a(\tilde{Q},S,\Delta_{S})\leq a(\tilde{Q},S_{m},\Delta_{m})\leq a(\tilde{Q},T,\Psi).$$
Therefore, the first assertion of the claim also holds. 
In this way, the claim holds. 
\end{proof}
By the second condition of Theorem \ref{thm--ind-1} and Lemma \ref{lem--bir-relation}, we see that $(T, \Psi)$ has a log minimal model. 
\end{step1}

\begin{step1}\label{step4-ind-1}
In this step, we prove that $(S_{m},\Delta_{S_{m}})$ has a log minimal model. 

For each $i$, the pair $(S_{i},\Delta_{S_{i}}+\lambda_{i}H_{S_{i}})$ is a weak lc model of $(S_{m}, \Delta_{S_{m}}+\lambda_{i}H_{S_{m}})$ because the birational map $S_{m}\dashrightarrow S_{i}$ is small, the divisor $\Delta_{S_{i}}+\lambda_{i}H_{S_{i}}$ is equal to the birational transform of $\Delta_{S_{m}}+\lambda_{i}H_{S_{m}}$ on $S_{i}$, and the divisor $K_{S_{i}}+\Delta_{S_{i}}+\lambda_{i}H_{S_{i}}=(K_{X_{i}}+\Delta_{i}+\lambda_{i}H_{i})|_{S_{i}}$ is nef. 
Pick any prime divisor $D$ on $T$. 
By Remark \ref{rem--mmp-zariskidecom} (\ref{rem--mmp-zariskidecom-(3)}) and (\ref{proof-thm--ind-1-step3-(c)}) in Step \ref{step2-ind-1}, we obtain
\begin{equation*}
\begin{split}
\sigma_{D}(K_{S_{m}}+\Delta_{S_{m}}+\lambda_{i}H_{S_{m}})&=a(D,S_{i},\Delta_{S_{i}}+\lambda_{i}H_{S_{i}} )-a(D, S_{m}, \Delta_{S_{m}}+\lambda_{i}H_{S_{m}})\\
&\geq a(D,S,\Delta_{S}+\lambda_{i}H_{S} )-a(D, S_{m}, \Delta_{S_{m}}+\lambda_{i}H_{S_{m}}).
\end{split}
\end{equation*}
By (\ref{proof-thm--ind-1-step3-(b)}) in Step \ref{step2-ind-1}, we have $a(D,S,\Delta_{S})=a(D,T,\Psi)$. 
By Remark \ref{rem--asy-van-ord-1} (\ref{rem--asy-van-ord-1-(2)}), using the fact that ${\rm lim}_{i \to \infty}\lambda_{i}=0$ and taking the limit $i\to \infty$, we obtain  
\begin{equation*}
\begin{split}
\sigma_{D}(K_{S_{m}}+\Delta_{S_{m}})\geq&\underset{i\to \infty}{\rm lim}\bigl(a(D,S, \Delta_{S}+\lambda_{i}H_{S})-a(D, S_{m}, \Delta_{S_{m}}+\lambda_{i}H_{S_{m}})\bigr)\\
=&a(D,T,\Psi)-a(D, S_{m}, \Delta_{S_{m}})\\
\geq & 0.
\end{split}
\end{equation*}
Here, the last inequality holds by (\ref{proof-thm--ind-1-step3-(d)}) in Step \ref{step2-ind-1}. 
By Lemma \ref{lem--minmodel-zariskidecom} and the existence of a log minimal model of $(T, \Psi)$ by Step \ref{step3-ind-1}, we see that $(S_{m},\Delta_{S_{m}})$ has a log minimal model.  
\end{step1}
By Step \ref{step1-ind-1} (see also Remark \ref{rem--mmplift}) and Step \ref{step4-ind-1}, we complete the argument of the special termination. 
Therefore, $(X,\Delta)$ has a log minimal model. 
We complete the proof. 
\end{proof}

We are now ready to prove our main results, namely Theorem \ref{thm--main-logabundant} and Theorem \ref{thm--abundantterminate}. 

\begin{thm}\label{thm--main-logabundant}
Let $\pi\colon X\to Z$ be a projective morphism of normal quasi-projective varieties and let $(X,\Delta)$ be an lc pair. 
Let $A$ be an effective $\mathbb{R}$-divisor on $X$ such that $(X,\Delta+A)$ is an lc pair and $K_{X}+\Delta+A$ is nef over $Z$.
Then no infinite sequence of steps of a $(K_{X}+\Delta)$-MMP over $Z$ with scaling of $A$
$$(X_{0}:=X,\Delta_{0}:=\Delta) \dashrightarrow (X_{1},\Delta_{1}) \dashrightarrow\cdots \dashrightarrow (X_{i},\Delta_{i})\dashrightarrow \cdots$$
satisfies the following properties.
\begin{itemize}
\item
If we define $\lambda_{i}={\rm inf}\set{\mu \in \mathbb{R}_{\geq0} \!|\! K_{X_{i}}+\Delta_{i}+\mu A_{i}\text{\rm \, is nef over }Z}$, where $A_{i}$ is the birational transform of $A$ on $X_{i}$, then ${\rm lim}_{i\to\infty}\lambda_{i}=0$, and  
\item
there are infinitely many $i$ such that $(X_{i},\Delta_{i})$ is log abundant over $Z$. 
\end{itemize}
\end{thm}

\begin{proof}
We prove the theorem by induction on ${\rm dim}\,X$. 
Let $\pi\colon X \to Z$, $(X,\Delta)$, and $A$ be as in Theorem \ref{thm--main-logabundant}, and assume Theorem \ref{thm--main-logabundant} for all lc pairs whose dimensions are less than ${\rm dim}\,X$. 
Suppose by contradiction that there exists an infinite sequence of steps of a $(K_{X}+\Delta)$-MMP over $Z$ with scaling of $A$
$$(X_{0}:=X,\Delta_{0}:=\Delta) \dashrightarrow (X_{1},\Delta_{1}) \dashrightarrow\cdots \dashrightarrow (X_{i},\Delta_{i})\dashrightarrow \cdots$$
satisfying the two properties of Theorem \ref{thm--main-logabundant}. 
The two properties of Theorem \ref{thm--main-logabundant} are preserved after  lifting the log MMP as in \cite[Remark 2.9]{birkar-flip} (see Remark \ref{rem--div} (\ref{rem--div-(2)}) for the preservation of the property of being log abundant over $Z$). 
Note that the argument in \cite[Remark 2.9]{birkar-flip} shows that we can apply \cite[Remark 2.9]{birkar-flip} even if a log MMP does not necessarily contain only flips. 
Thus, by applying \cite[Remark 2.9]{birkar-flip} to this log MMP, we may assume that all $(X_{i},\Delta_{i})$ are $\mathbb{Q}$-factorial dlt. 
By the negativity lemma, $K_{X}+\Delta+\lambda_{i}A$ is pseudo-effective over $Z$ for all $i$. 
Therefore, $K_{X}+\Delta$ is pseudo-effective over $Z$. 

In the rest of the proof, we will prove that $(X,\Delta)$ has a log minimal model over $Z$, and thus we will get a contradiction (\cite[Theorem 4.1 (iii)]{birkar-flip}). 

\begin{step2}\label{step1--main-logabundant}
In this step, we prove that the non-isomorphic locus of the $(K_{X}+\Delta)$-MMP does not intersect the birational transform of $\llcorner \Delta \lrcorner$ after finitely many steps. 

For any $i$ and any lc center $S_{i}$ of $(X_{i},\Delta_{i})$, we define a dlt pair $(S_{i},\Delta_{S_{i}})$ by adjunction $K_{S_{i}}+\Delta_{S_{i}}=(K_{X_{i}}+\Delta_{i})|_{S_{i}}$. 
We prove by induction that for any integer $0\leq d < {\rm dim}\, X$, there is $m_{d}\geq 0$ such that the non-isomorphic locus of the $(K_{X}+\Delta)$-MMP is disjoint from all $d$-dimensional lc centers of $(X_{m_{d}},\Delta_{m_{d}})$ after $m_{d}$ steps. 
We assume the existence of $m_{d-1}$, and  we put $m=m_{d-1}$. 
By the arguments in \cite{fujino-sp-ter}, replacing $m$, we may assume that for any $i\geq m$ and any $d$-dimensional lc center $S_{m}$ of $(X_{m},\Delta_{m})$, the birational map $X_{i}\dashrightarrow X_{i+1}$ induces a small birational map $S_{i}\dashrightarrow S_{i+1}$  such that the birational transform of $\Delta_{S_{i}}$ on $S_{i+1}$ is equal to $\Delta_{S_{i+1}}$. 
Let $(T_{m},\Psi_{m})\to (S_{m},\Delta_{S_{m}})$ be a dlt blow-up, and put $A_{T_{m}}=A_{m}|_{T_{m}}$. 
By ${\rm lim}_{i\to \infty}\lambda_{i}=0$ and Remark \ref{rem--mmplift}, the sequence of birational maps 
$$S_{m}\dashrightarrow \cdots \dashrightarrow S_{i} \dashrightarrow \cdots$$ induces a $(K_{{T}_{m}}+\Psi_{m})$-MMP over $Z$ with scaling of $A_{{T}_{m}}$ 
$$({T}_{m},\Psi_{m})\dashrightarrow \cdots \dashrightarrow ({T}_{i},\Psi_{i})\dashrightarrow \cdots$$ such that $({T}_{i},\Psi_{i})$ is a $\mathbb{Q}$-factorial dlt model of $(S_{i},\Delta_{S_{i}})$ for every $i$ and the termination of the $(K_{{T}_{m}}+\Psi_{m})$-MMP over $Z$ implies the existence of $m_{d}$ stated above. 
By definition, if $(X_{i},\Delta_{i})$ is log abundant over $Z$, then $({T}_{i},\Psi_{i})$ is also log abundant over $Z$. 
Since there are infinitely many log abundant dlt pairs $(X_{i},\Delta_{i})$ over $Z$, there are also infinitely many log abundant dlt pairs $(T_{i},\Psi_{i})$ over $Z$. 
Thus, the $(K_{{T}_{m}}+\Psi_{m})$-MMP over $Z$ with scaling of $A_{{T}_{m}}$ satisfies the two conditions about log MMP in Theorem \ref{thm--main-logabundant}. 
By the induction hypothesis of Theorem \ref{thm--main-logabundant}, the $(K_{{T}_{m}}+\Psi_{m})$-MMP terminates. 
Therefore, we may find the desired index $m_{d}$, which completes the induction on $d$, and hence the proof of our claim in this step. 

In this way, by replacing $(X,\Delta)$ with $(X_{i},\Delta_{i})$ for some $i\gg 0$, we may assume that the non-isomorphic locus of the $(K_{X}+\Delta)$-MMP does not intersect ${\rm Supp}\,\llcorner \Delta \lrcorner$. 
By the second condition  about log MMP in Theorem \ref{thm--main-logabundant}, we may also assume that $(X,\Delta)$ is log abundant over $Z$. 
\end{step2}

\begin{step2}\label{step2--main-logabundant}
In this step, by using Theorem \ref{thm--ind-1} we prove that $(X,\Delta)$ has a log minimal model over $Z$ in the case where $Z$ is a point. 

The first condition of Theorem \ref{thm--ind-1} is clear. 
By the hypothesis of the $(K_{X}+\Delta)$-MMP and Lemma \ref{lem--mmpscaling-effective}, for any ample $\mathbb{R}$-divisor $A'$ on $X$, the stable base locus of $K_{X}+\Delta+A'$ is disjoint from ${\rm Supp}\,\llcorner \Delta \lrcorner$. 
Thus, the restriction of $K_{X}+\Delta$ to any component of ${\rm Supp}\,\llcorner \Delta \lrcorner$ is nef and $\sigma_{P}(K_{X}+\Delta)=0$ for all prime divisors $P$ over $X$ such that $c_{X}(P)$ intersects ${\rm Supp}\,\llcorner \Delta \lrcorner$. 
By these facts, the second and the third conditions of Theorem \ref{thm--ind-1} hold. 
By Theorem \ref{thm--ind-1}, $(X,\Delta)$ has a log minimal model. 
\end{step2}

\begin{step2}\label{step3--main-logabundant}
We prove the existence of a log minimal model of $(X,\Delta)$ over $Z$ in the general case. 
Taking the Stein factorization of $\pi$, we may assume that $\pi$ is a contraction. 

We take an open immersion $Z\hookrightarrow \overline{Z}$ to a normal projective variety. 
By Lemma \ref{lem--Qfacdltclosure}, there exists a $\mathbb{Q}$-factorial dlt closure $(\overline{X},\overline{\Delta})$ of $(X,\Delta)$, that is, a projective $\mathbb{Q}$-factorial dlt pair $(\overline{X},\overline{\Delta})$ with a projective morphism $\overline{\pi}\colon \overline{X} \to \overline{Z}$ such that $X$ is an open subset of $\overline{X}$, $(\overline{X}|_{X},\overline{\Delta}|_{X})=(X,\Delta)$, and all lc centers of $(\overline{X},\overline{\Delta})$ intersect $X$. 
By construction, we have ${\overline{\pi}}^{-1}(Z)=X$. 
Since all lc centers of $(\overline{X},\overline{\Delta})$ intersect $X$ and the property of being $\overline{\pi}$-abundant for $\mathbb{R}$-Cartier divisors is determined on sufficiently general fibers of $\overline{\pi}$ (see \cite[Lemma 2.10]{hashizumehu}), the pair $(\overline{X},\overline{\Delta})$ is log abundant over $\overline{Z}$. 
If $(\overline{X},\overline{\Delta})$ has a log minimal model over $\overline{Z}$, then its restriction over $Z$ is a log minimal model of $(X,\Delta)$ over $Z$. 
In this way, to prove the existence of a log minimal model of $(X,\Delta)$ over $Z$, it is sufficient to prove that $(\overline{X},\overline{\Delta})$ has a log minimal model over $\overline{Z}$. 

Pick an ample Cartier divisor $H_{\overline{Z}}$ on $\overline{Z}$. Let $\overline{H}$ be a general member of the $\mathbb{R}$-linear system $|(3\cdot {\rm dim}\,X) \overline{\pi}^{*}H_{\overline{Z}}|_{\mathbb{R}}$ (see \cite[pp.~112--113]{hashizumehu} for the notion of general semi-ample $\mathbb{R}$-divisors). 
We may assume that $(\overline{X},\overline{\Delta}+\overline{H})$ is a dlt pair whose lc centers are those of $(\overline{X},\overline{\Delta})$. 
We run a $(K_{\overline{X}}+\overline{\Delta}+\overline{H})$-MMP with scaling of an ample divisor
$$(\overline{X}_{0}:=\overline{X},\overline{\Delta}_{0}+\overline{H}_{0}:=\overline{\Delta}+\overline{H})\dashrightarrow\cdots \dashrightarrow (\overline{X}_{j},\overline{\Delta}_{j}+\overline{H}_{j})\dashrightarrow \cdots.$$
To apply the conclusion of Step \ref{step2--main-logabundant} to this log MMP, we check that this log MMP satisfies the two conditions about log MMP in Theorem \ref{thm--main-logabundant}.  
The $(K_{\overline{X}}+\overline{\Delta}+\overline{H})$-MMP satisfies the first condition about log MMP in Theorem \ref{thm--main-logabundant} (\cite[Theorem 4.1 (ii)]{birkar-flip}). 
Hence it is sufficient to prove that $(\overline{X}_{j},\overline{\Delta}_{j}+\overline{H}_{j})$ is log abundant for all $j$. 
By the definition of $\overline{H}$ and the length of extremal rays (\cite[Section 18]{fujino-fund}), the $(K_{\overline{X}}+\overline{\Delta}+\overline{H})$-MMP is also a $(K_{\overline{X}}+\overline{\Delta})$-MMP over $\overline{Z}$. 
Because the non-isomorphic locus of the $(K_{X}+\Delta)$-MMP over $Z$ is disjoint from ${\rm Supp}\,\llcorner \Delta \lrcorner$ and ${\rm lim}_{i\to \infty}\lambda_{i}=0$, we may apply Lemma \ref{lem--mmpscaling-effective}. 
For any ample $\mathbb{R}$-divisor $H'$ on $X$ and any closed point $x\in {\rm Supp}\,\llcorner \Delta \lrcorner$, we can find $E\geq 0$ such that $E \sim_{\mathbb{R},Z}K_{X}+\Delta+H'$ and ${\rm Supp}\,E\not\ni x$. 
Since $X=\overline{\pi}^{-1}(Z)$, for any curve that is contracted by $\overline{\pi}$ and intersects $\overline{\pi}^{-1}(Z)\cap {\rm Supp}\,\llcorner \overline{\Delta} \lrcorner$, its intersection number with $K_{\overline{X}}+\overline{\Delta}$ is non-negative. 
By this fact and induction on the index $j$, we see that the non-isomorphic locus of the $(K_{\overline{X}}+\overline{\Delta})$-MMP over $\overline{Z}$ is disjoint from $\overline{\pi}^{-1}(Z)\cap {\rm Supp}\,\llcorner \overline{\Delta} \lrcorner$. 
This implies that for any lc center $\overline{S}$ of $(\overline{X},\overline{\Delta}+\overline{H})$, the $(K_{\overline{X}}+\overline{\Delta})$-MMP over $\overline{Z}$ does not modify sufficiently general fibers of the morphism $\overline{S}\to \overline{\pi}(\overline{S})$. 
For any lc center $\overline{S}_{j}$ of $(\overline{X}_{j},\overline{\Delta}_{j})$, we define a dlt pair $(\overline{S}_{j},\Delta_{\overline{S}_{j}})$ by adjunction 
$$K_{\overline{S}_{j}}+\Delta_{\overline{S}_{j}}=(K_{\overline{X}_{j}}+\overline{\Delta}_{j})|_{\overline{S}_{j}}.$$ 
By \cite[Lemma 2.10]{hashizumehu}, we may check the property of being abundant over $\overline{Z}$ for $K_{\overline{S}_{j}}+\Delta_{\overline{S}_{j}}$ by restricting the divisor to a sufficiently general fiber of $\overline{S}\to \overline{\pi}(\overline{S})$. 
Since  $(\overline{X},\overline{\Delta})$ is log abundant over $\overline{Z}$, we see that $K_{\overline{S}_{j}}+\Delta_{\overline{S}_{j}}$ is abundant over $\overline{Z}$.  
By the definition of $\overline{H}$ and Lemma \ref{lem--relabund-globabund}, $(K_{\overline{X}_{j}}+\overline{\Delta}_{j}+\overline{H}_{j})|_{\overline{S}_{j}}$ is abundant. 
Because all lc centers of $(\overline{X}_{j},\overline{\Delta}_{j}+\overline{H}_{j})$ are lc centers of $(\overline{X}_{j},\overline{\Delta}_{j})$, the lc pair $(\overline{X}_{j},\overline{\Delta}_{j}+\overline{H}_{j})$ is log abundant for all $j$. 
Therefore, the $(K_{\overline{X}}+\overline{\Delta}+\overline{H})$-MMP with scaling satisfies the two conditions of Theorem \ref{thm--main-logabundant} about log MMP. 

By the conclusion of Step \ref{step2--main-logabundant} in this proof, $(\overline{X},\overline{\Delta}+\overline{H})$ has a log minimal model. 
By \cite[Theorem 4.1 (iii)]{birkar-flip}, the $(K_{\overline{X}}+\overline{\Delta}+\overline{H})$-MMP, which is also a $(K_{\overline{X}}+\overline{\Delta})$-MMP over $\overline{Z}$, terminates. 
This implies the existence of a log minimal model of $(\overline{X},\overline{\Delta})$ over $\overline{Z}$. 
Therefore, $(X,\Delta)$ has a log minimal model over $Z$. 
\end{step2}

By \cite[Theorem 4.1 (iii)]{birkar-flip}, we get a contradiction. 
We finish the proof. 
\end{proof}

\begin{thm}\label{thm--abundantterminate}
Let $\pi\colon X\to Z$ be a projective morphism of normal quasi-projective varieties and let $(X,\Delta)$ be a $\mathbb{Q}$-factorial dlt pair such that $K_{X}+\Delta$ is pseudo-effective over $Z$. 
Let $A$ be a $\pi$-ample $\mathbb{R}$-divisor on $X$ such that $(X,\Delta+A)$ is lc and $K_{X}+\Delta+A$ is $\pi$-nef. 
Then there is an integer $l \geq 0$ satisfying the following: 
Let
$$(X_{0}:=X,\Delta_{0}:=\Delta) \dashrightarrow (X_{1},\Delta_{1}) \dashrightarrow\cdots \dashrightarrow (X_{i},\Delta_{i})\dashrightarrow \cdots$$
be a sequence of steps of a $(K_{X}+\Delta)$-MMP over $Z$ with scaling of $A$. 
If $(X_{i},\Delta_{i})$ is log abundant over $Z$ for some $i\geq l$, then the $(K_{X}+\Delta)$-MMP over $Z$ terminates with a good minimal model of $(X,\Delta)$ over $Z$. 
\end{thm}

\begin{proof}
By Lemma \ref{lem--abund-inv}, there exists $\epsilon>0$ such that for any $s,s'\in(0,\epsilon]$, any sequence of steps of a $(K_{X}+\Delta+sA)$-MMP 
$(X,\Delta+s A)\dashrightarrow (Y,\Delta_{Y}+sA_{Y})$ over $Z$ to a good minimal model $(Y,\Delta_{Y}+sA_{Y})$, and any sequence of steps of a $(K_{X}+\Delta+s'A)$-MMP $(X,\Delta+s' A)\dashrightarrow (Y',\Delta_{Y'}+s'A_{Y'})$ over $Z$ to a good minimal model  $(Y',\Delta_{Y'}+s'A_{Y'})$, the three properties of Lemma \ref{lem--abund-inv} hold true. 
We may assume that $\epsilon<1$. 
By applying the finiteness of models \cite[Theorem E]{bchm} to a suitable convex rational polytope of ${\rm WDiv}_{\mathbb{R}}(X)$ which contains the set $$\left\{\Delta+tA \,\middle|\,\frac{\epsilon}{2}\leq t \leq1\right\},$$ 
we can find finitely many birational maps $X\dashrightarrow Y_{k}$ such that if $t \in [\frac{\epsilon}{2},1]$, then the birational contraction from $X$ to any weak lc model of $(X,\Delta+t A)$ over $Z$ is isomorphic to one of $X\dashrightarrow Y_{k}$. 

Let $l$ be the number of those varieties $Y_{k}$. 
We will prove that $l$ is the desired integer. 
Consider any sequence of steps of a $(K_{X}+\Delta)$-MMP over $Z$ with scaling of $A$
$$(X_{0}:=X,\Delta_{0}:=\Delta) \dashrightarrow (X_{1},\Delta_{1}) \dashrightarrow\cdots \dashrightarrow (X_{i},\Delta_{i})\dashrightarrow \cdots,$$
and suppose that $(X_{i_{0}},\Delta_{i_{0}})$ is log abundant over $Z$ for some $i_{0}\geq l$. 
We define 
$$\lambda_{i}={\rm inf}\set{\mu \in \mathbb{R}_{\geq0} \!|\! K_{X_{i}}+\Delta_{i}+\mu A_{i}\text{\rm \, is nef over }Z},$$
where $A_{i}$ is the birational transform of $A$ on $X_{i}$. 
Since $(X_{i},\Delta_{i}+\lambda_{i}A_{i})$ is a weak lc model of $(X,\Delta+\lambda_{i}A)$ over $Z$ for all $0\leq i<l$, by the choice of $l$, we see that $\lambda_{i_{0}}<\frac{\epsilon}{2}$. 
If $\lambda_{i_{0}}=0$, then $(X_{i_{0}},\Delta_{i_{0}})$ is a log minimal model of $(X,\Delta)$ over $Z$ and it is log abundant over $Z$, so $(X_{i_{0}},\Delta_{i_{0}})$ is a good minimal model over $Z$ by Lemma \ref{lem--relnefabundance}. 
From now on, we assume that $\lambda_{i_{0}}>0$. 
We define $\lambda_{-1}:=1$. 
We pick integers $0\leq l' \leq l''$ such that 
$$\lambda_{l'-1}>\lambda_{l'}=\cdots=\lambda_{i_{0}}=\cdots=\lambda_{l''-1}>\lambda_{l''}.$$
We can find such $l'$ and $l''$ since $\lambda_{-1}>\lambda_{i_{0}}>0$ and ${\rm lim}_{i \to \infty}\lambda_{i}=0$ by construction of the log MMP (cf.~\cite[Theorem 4.1 (ii)]{birkar-flip}).  
We pick $t' \in (\lambda_{l'}, \lambda_{l'-1})$ and $t'' \in (\lambda_{l''},\lambda_{l''-1})$ so that $t'<\epsilon$. 
Then
\begin{itemize}
\item
$t''<\lambda_{i_{0}}<t'$, 
\item
$(X,\Delta+t'A)\dashrightarrow (X_{l'},\Delta_{l'}+t'A_{l'})$ is finite steps of a $(K_{X}+\Delta+t'A)$-MMP over $Z$ to a good minimal model, and 
\item
$(X,\Delta+t''A)\dashrightarrow (X_{l''},\Delta_{l''}+t''A_{l''})$ is finite steps of a $(K_{X}+\Delta+t''A)$-MMP over $Z$ to a good minimal model. 
\end{itemize}
Then $X_{l'}\dashrightarrow X_{l''}$ satisfies the three properties of Lemma \ref{lem--abund-inv}. 
In particular, $X_{l'}\dashrightarrow X_{l''}$ is small and an isomorphism on an open subset of $X_{l'}$ which intersects all lc centers of $(X_{l'},\Delta_{l'})$. 
Since $X_{l'}\dashrightarrow X_{i_{0}} \dashrightarrow X_{l''}$ is a sequence of steps of a $(K_{X_{l'}}+\Delta_{l'})$-MMP over $Z$, we see that 
\begin{itemize}
\item
the birational map $X_{l'} \dashrightarrow X_{i_{0}}$ is small, and 
\item
 $X_{l'} \dashrightarrow X_{i_{0}}$ is an isomorphism on an open set $U' \subset X_{l'}$ such that $U'$ intersects all lc centers of $(X_{l'},\Delta_{l'})$ and the image $U\subset X_{i_{0}}$ intersects all lc centers of $(X_{i_{0}},\Delta_{i_{0}})$.
 \end{itemize}
The same properties hold for the birational map $X_{i_{0}}\dashrightarrow X_{l''}$. 
By Lemma \ref{lem--comparedim}, for any lc center $S_{i_{0}}$ of $(X_{i_{0}},\Delta_{i_{0}})$ with the birational map $S_{l'}\dashrightarrow S_{i_{0}}\dashrightarrow S_{l''}$, where $S_{l'}$ and $S_{l''}$ are corresponding lc centers, we have
\begin{equation*}
\begin{split}
\kappa_{\iota}(S_{l'}/Z, (K_{X_{l'}}+\Delta_{l'})|_{S_{l'}})&\geq \kappa_{\iota}(S_{i_{0}}/Z, (K_{X_{i_{0}}}+\Delta_{i_{0}})|_{S_{i_{0}}})\\&\geq \kappa_{\iota}(S_{l''}/Z, (K_{X_{l''}}+\Delta_{l''})|_{S_{l''}}),
\end{split}
\end{equation*} 
and 
\begin{equation*}
\begin{split}
\kappa_{\sigma}(S_{l'}/Z, (K_{X_{l'}}+\Delta_{l'})|_{S_{l'}})&\geq \kappa_{\sigma}(S_{i_{0}}/Z, (K_{X_{i_{0}}}+\Delta_{i_{0}})|_{S_{i_{0}}})\\&\geq \kappa_{\sigma}(S_{l''}/Z, (K_{X_{l''}}+\Delta_{l''})|_{S_{l''}}). 
\end{split}
\end{equation*}
Now we apply the third condition of Lemma \ref{lem--abund-inv} to $X_{l'}\dashrightarrow X_{l''}$ and $S_{l'}\dashrightarrow S_{l''}$. 
Using the above relations, we have 
\begin{equation*}
\begin{split}
\kappa_{\iota}(S_{l'}/Z, (K_{X_{l'}}+\Delta_{l'})|_{S_{l'}})&= \kappa_{\iota}(S_{i_{0}}/Z, (K_{X_{i_{0}}}+\Delta_{i_{0}})|_{S_{i_{0}}}) \quad {\rm and}\\ 
\kappa_{\sigma}(S_{l'}/Z, (K_{X_{l'}}+\Delta_{l'})|_{S_{l'}})&= \kappa_{\sigma}(S_{i_{0}}/Z, (K_{X_{i_{0}}}+\Delta_{i_{0}})|_{S_{i_{0}}}). 
\end{split}
\end{equation*}
By recalling that $(X_{i_{0}},\Delta_{i_{0}})$ is log abundant over $Z$, we see that $(X_{l'},\Delta_{l'})$ is log abundant over $Z$. 

For any $j \geq l''$ such that $\lambda_{j-1}>\lambda_{j}$, we take $u\in (\lambda_{j},\lambda_{j-1})$. 
Then
$$(X,\Delta+uA)\dashrightarrow (X_{j},\Delta_{j}+uA_{j})$$
 is a sequence of steps of a $(K_{X}+\Delta+uA)$-MMP over $Z$ to a good minimal model. 
Since $t', u \in (0, \epsilon]$ and $(X_{l'},\Delta_{l'})$ is log abundant over $Z$, by applying Lemma \ref{lem--abund-inv} to 
$$(X,\Delta+t'A)\dashrightarrow (X_{l'},\Delta_{l'}+t'A_{l'}) \quad {\rm and} \quad (X,\Delta+uA)\dashrightarrow (X_{j},\Delta_{j}+uA_{j}),$$ we see that $(X_{j},\Delta_{j})$ is log abundant over $Z$. 
Therefore, $(X_{j},\Delta_{j})$ is log abundant over $Z$ for any $j\geq l''$ such that $\lambda_{j-1}>\lambda_{j}$. 
By Theorem \ref{thm--main-logabundant}, we do not have infinitely many indices $j \geq l''$ such that $\lambda_{j-1}>\lambda_{j}$. 
Then we have $\lambda_{m}=0$ for some $m$, which implies that the $(K_{X}+\Delta)$-MMP over $Z$ with scaling of $A$ terminates. 

The resulting log minimal model $(X_{m},\Delta_{m})$ over $Z$ of the $(K_{X}+\Delta)$-MMP over $Z$ satisfies $\lambda_{m-1}>\lambda_{m}$ and $m\geq l''$ since $\lambda_{l''-1}=\lambda_{i_{0}}>0$. 
By the argument in the previous paragraph, $(X_{m},\Delta_{m})$ is log abundant over $Z$.  
By Lemma \ref{lem--relnefabundance}, the pair $(X_{m},\Delta_{m})$ is a good minimal model of $(X,\Delta)$ over $Z$. 
So we are done. 
\end{proof}

\begin{rem}\label{rem--recover}
We check that the main results of \cite{birkar-flip}, \cite{haconxu-lcc}, \cite{has-mmp}, and \cite{hashizumehu}, and in addition \cite[Theorem 1.3]{has-class} follow from Theorem \ref{thm--abundantterminate}. Let $\pi\colon X\to Z$ be a projective morphism of normal quasi-projective varieties, and let $(X,\Delta)$ be an lc pair. 
We may assume that $K_{X}+\Delta$ is pseudo-effective over $Z$. 
Except the situation of \cite[Theorem 1.5]{hashizumehu}, without loss of generality, we may further assume that $(X,\Delta)$ is $\mathbb{Q}$-factorial dlt. 

If $\pi$ and $(X,\Delta)$ are as in \cite[Theorem 1.1]{birkar-flip} (see also \cite[Theorem 1.6]{haconxu-lcc}), then the relative numerical dimension of the restriction of $K_{X}+\Delta$ to any stratum is either zero or $-\infty$, therefore $(X,\Delta)$ is log abundant over $Z$ (\cite{gongyo1}). 
Because any $(K_{X}+\Delta)$-MMP over $Z$ preserves the condition of \cite[Theorem 1.1]{birkar-flip}, any $(K_{X}+\Delta)$-MMP over $Z$ preserves the property of being log abundant over $Z$. 

If $\pi$ and $(X,\Delta)$ are as in \cite[Theorem 1.1]{haconxu-lcc} (or \cite[Theorem 1.4]{birkar-flip}), then there is $U^{0} \subset Z$ as in \cite[Theorem 1.1]{haconxu-lcc}.  
By running a $(K_{X}+\Delta)$-MMP over $Z$ with scaling of an ample divisor 
$$(X,\Delta) \dashrightarrow \cdots \dashrightarrow (X_{i},\Delta_{i})\dashrightarrow \cdots$$
 and replacing $(X,\Delta)$ by $(X_{i},\Delta_{i})$ for some $i\gg 0$, we may assume that $(K_{X}+\Delta)|_{\pi^{-1}(U^{0})}$ is semi-ample over $U^{0}$. 
Then $(X,\Delta)$ is log abundant over $Z$ (\cite[Lemma 2.10]{hashizumehu}). 
Since any $(K_{X}+\Delta)$-MMP over $Z$ does not modify $\pi^{-1}(U^{0})$, it preserves the property of being log abundant over $Z$. 

The above arguments are valid for the case of the main results of \cite{has-mmp} because they are $\mathbb{R}$-boundary divisor analogs of \cite[Theorem 1.1]{birkar-flip} and \cite[Theorem 1.1]{haconxu-lcc}. 

If $\pi$ and $(X,\Delta)$ are as in \cite[Theorem 1.1]{hashizumehu}, for any finite steps 
$$(X,\Delta)\dashrightarrow (X_{i},\Delta_{i})$$
 of a $(K_{X}+\Delta)$-MMP over $Z$ and any lc center $S_{i}$ of $(X_{i},\Delta_{i})$, there is an lc center $S$ of $(X,\Delta)$ such that $X\dashrightarrow X_{i}$ induces a birational map $S\dashrightarrow S_{i}$. 
Let $(S,\Delta_{S})$ and $(S_{i},\Delta_{S_{i}})$ be dlt pairs defined by adjunctions. 
Then $a(Q,S,\Delta_{S})\leq a(Q,S_{i},\Delta_{S_{i}})$ for all prime divisors $Q$ over $S$ (\cite[Lemma 4.2.10]{fujino-sp-ter}). 
By \cite[Lemma 2.17]{hashizumehu} and the hypothesis of \cite[Theorem 1.1]{hashizumehu} on lc centers, $K_{S_{i}}+\Delta_{S_{i}}$ is abundant over $Z$.
Therefore, $(X_{i},\Delta_{i})$ is log abundant over $Z$. 
Note that we need \cite[Theorem 1.1]{has-mmp} to prove \cite[Lemma 2.17]{hashizumehu}, however, \cite[Theorem 1.1]{has-mmp} has been already recovered in the above discussion.  
Thus, any $(K_{X}+\Delta)$-MMP over $Z$ preserves the property of being log abundant over $Z$. 

If $\pi$ and $(X,\Delta)$ are as in \cite[Theorem 1.3]{has-class}, then the relative numerical dimension of the restriction of $K_{X}+\Delta$ to any stratum is either zero or $-\infty$. 
Therefore, this situation is a special case (i.e., the case of $n=3$) of \cite[Theorem 1.1]{hashizumehu}. 

Finally, suppose that $\pi$ and $(X,\Delta)$ are as in \cite[Theorem 1.5]{hashizumehu}, that is, $\Delta=A+B$, where $A$ is a $\pi$-ample $\mathbb{R}$-divisor on $X$ and $B$ is an effective $\mathbb{R}$-divisor on $X$ such that $(X,B)$ is lc. 
We replace $A$ by a general one (cf.~\cite[pp.~112--113]{hashizumehu}), and we take a dlt blow-up $f\colon (Y,B_{Y})\to (X,B)$. 
We put $\Gamma=f^{*}A+B_{Y}$. 
Then we can directly apply the argument in \cite{has-nonvan-gpair}, and we see that $(Y,\Gamma)$ is log abundant over $Z$ and any $(K_{Y}+\Gamma)$-MMP over $Z$ preserves the property of being log abundant. 

In any case, by applying Theorem \ref{thm--abundantterminate} after replacing $(X,\Delta)$ with a $\mathbb{Q}$-factorial dlt model, any $(K_{X}+\Delta)$-MMP over $Z$ with scaling of an ample divisor terminates with a log minimal model $(X_{m},\Delta_{m})$ that is log abundant over $Z$. By Lemma \ref{lem--relnefabundance},  $(X_{m},\Delta_{m})$ is a good minimal model of $(X,\Delta)$ over $Z$. 
\end{rem}

\subsection{Corollaries}\label{subsec--cor}
In this subsection, we prove corollaries.

\begin{cor}\label{cor--finite-logabund}
Any log MMP with scaling of an ample divisor starting with a projective dlt pair contains only finitely many log abundant dlt pairs. 
\end{cor}

\begin{proof}
Fix a projective dlt pair $(X,\Delta)$ and an ample $\mathbb{R}$-divisor $A$ on $X$. 
By using the perturbation of the coefficients of $\Delta$ and \cite[Theorem 3.5.1]{fujino-what-log-ter}, we can find an $\mathbb{R}$-divisor $B$ on $X$ such that $(X,B)$ is klt. 
We may assume that $K_{X}+\Delta$ is pseudo-effective because otherwise any $(K_{X}+\Delta)$-MMP with scaling of $A$ terminates with a Mori fiber space.  
Consider any sequence of steps of a $(K_{X}+\Delta)$-MMP with scaling of $A$ 
$$(X_{0}:=X,\Delta_{0}:=\Delta) \dashrightarrow (X_{1},\Delta_{1}) \dashrightarrow\cdots \dashrightarrow (X_{i},\Delta_{i})\dashrightarrow \cdots.$$
We define 
$$\lambda_{i}={\rm inf}\set{\mu \in \mathbb{R}_{\geq0} \!|\! K_{X_{i}}+\Delta_{i}+\mu A_{i}\text{\rm \, is nef}},$$
where $A_{i}$ is the birational transform of $A$ on $X_{i}$, and we define $\lambda_{\infty}={\rm lim}_{i\to\infty}\lambda_{i}$. 
Suppose that the $(K_{X}+\Delta)$-MMP is an infinite sequence. 
For any $i\neq j$, the birational map $X_{i}\dashrightarrow X_{j}$ is not an isomorphism. 
Since each lc pair $(X_{i},\Delta_{i}+\lambda_{i}A_{i})$ is a weak lc model of $(X,\Delta+\lambda_{i}A)$, if $\lambda_{\infty}>0$ then it contradicts the finiteness of models \cite[Theorem E]{bchm}, therefore $\lambda_{\infty}=0$. 
By Theorem \ref{thm--main-logabundant}, the $(K_{X}+\Delta)$-MMP contains only finitely many log abundant dlt pairs. 
\end{proof}

\begin{cor}\label{cor--mmplclogabund}
For any projective lc pair $(X,\Delta)$ and any ample $\mathbb{R}$-divisor $A$ on $X$, there is a sequence of steps of a $(K_{X}+\Delta)$-MMP with scaling of $A$  which contains only finitely many log abundant lc pairs. 
\end{cor}

\begin{proof}
Let $(X,\Delta)$ and $A$ be as in the corollary. 
When $K_{X}+\Delta$ is not pseudo-effective, there is a finite sequence of steps of a $(K_{X}+\Delta)$-MMP with scaling of $A$ terminating with a Mori fiber space (\cite[Theorem 1.7]{hashizumehu}). 
When $K_{X}+\Delta$ is pseudo-effective, $(X,\Delta+tA)$ has a good minimal model for all $t \in (0,1]$ (\cite[Theorem 1.5]{hashizumehu}). 
By Lemma \ref{lem--mmp-termi}, we can construct a sequence of  steps of a $(K_{X}+\Delta)$-MMP with scaling of $A$
$$(X_{0}:=X,\Delta_{0}:=\Delta) \dashrightarrow (X_{1},\Delta_{1}) \dashrightarrow\cdots \dashrightarrow (X_{i},\Delta_{i})\dashrightarrow \cdots$$
such that if we define 
$$\lambda_{i}={\rm inf}\set{\mu \in \mathbb{R}_{\geq0} \!|\! K_{X_{i}}+\Delta_{i}+\mu A_{i}\text{\rm \, is nef}},$$
where $A_{i}$ is the birational transform of $A$ on $X_{i}$, then the $(K_{X}+\Delta)$-MMP terminates or ${\rm lim}_{i\to\infty}\lambda_{i}=0$. 
Therefore, the $(K_{X}+\Delta)$-MMP satisfies the first condition about log MMP in Theorem \ref{thm--main-logabundant} when it is an infinite sequence. 
By Theorem \ref{thm--main-logabundant}, only finitely many log abundant lc pairs appear in the log MMP. 
\end{proof}

\begin{cor}\label{cor--logabund-preserved-mmp}
Let $\pi\colon X\to Z$ be a projective morphism of normal quasi-projective varieties and let $(X,\Delta)$ be an lc pair.
Suppose that 
\begin{itemize}
\item
$(X,\Delta)$ is log abundant over $Z$ and $K_{X}+\Delta$ is pseudo-effective over $Z$, and
\item
the stable base locus of $K_{X}+\Delta$ over $Z$ {\rm (}\cite[Definition 3.5.1]{bchm}{\rm )} does not contain the image of any prime divisor $P$ over $X$ whose discrepancy $a(P,X,\Delta)$ is negative. 
\end{itemize}
Then $(X,\Delta)$ has a good minimal model over $Z$. 
\end{cor}

\begin{proof}
Note that the two properties of $(X,\Delta)$ in Corollary \ref{cor--logabund-preserved-mmp} are preserved after replacing $(X,\Delta)$ by a $\mathbb{Q}$-factorial dlt model (see Remark \ref{rem--div} (\ref{rem--div-(2)}) and Definition \ref{defn--asy-van-ord}). Thus, replacing $(X,\Delta)$ with a $\mathbb{Q}$-factorial dlt model of $(X,\Delta)$, we may assume that $(X,\Delta)$ is $\mathbb{Q}$-factorial dlt. 

We run a $(K_{X}+\Delta)$-MMP over $Z$ with scaling of an ample divisor $A$ 
$$(X_{0}:=X,\Delta_{0}:=\Delta) \dashrightarrow (X_{1},\Delta_{1}) \dashrightarrow\cdots \dashrightarrow (X_{i},\Delta_{i})\dashrightarrow \cdots.$$
For any $i\geq 0$ and any lc center $S_{i}$ of $(X_{i},\Delta_{i})$, there is an lc center $S$ of $(X,\Delta)$ such that $X\dashrightarrow X_{i}$ induces a birational map $S\dashrightarrow S_{i}$. 
We define dlt pairs $(S,\Delta_{S})$ and $(S_{i},\Delta_{S_{i}})$ by adjunctions $K_{S}+\Delta_{S}=(K_{X}+\Delta)|_{S}$ and $K_{S_{i}}+\Delta_{S_{i}}=(K_{X_{i}}+\Delta_{i})|_{S_{i}}$, respectively. 

We fix an index $i$. 
By taking a log resolution $g\colon Y\to X$ of $(X,\Delta)$ which resolves the indeterminacy of $X\dashrightarrow X_{i}$, we may find a common log resolution $g$ and $g_{i} \colon Y \to X_{i}$ of the map $(X,\Delta)\dashrightarrow (X_{i},\Delta_{i})$ and a subvariety $T\subset Y$ birational to $S$ and $S_{i}$ such that the induced morphisms $g|_{T}\colon T\to S$ and $g_{i}|_{T}\colon T\to S_{i}$ form a common resolution of $S\dashrightarrow S_{i}$. 
Then we may write
\begin{equation*}
\begin{split}
g^{*}(K_{X}+\Delta)=g_{i}^{*}(K_{X_{i}}+\Delta_{i})+E \quad {\rm and} \quad
g|_{T}^{*}(K_{S}+\Delta_{S})=&g_{i}|_{T}^{*}(K_{S_{i}}+\Delta_{S_{i}})+E|_{T}
\end{split}
\end{equation*}
for some $g_{i}$-exceptional $\mathbb{R}$-divisor $E$ on $Y$. 
By \cite[Lemma 4.2.10]{fujino-sp-ter}, we have $E \geq 0$, and thus $E|_{T} \geq 0$. 
Furthermore, the second condition of Corollary \ref{cor--logabund-preserved-mmp} implies that 
$$0 \leq a(P, X,\Delta)< a(P,X_{i}, \Delta_{i})$$
for every component $P$ of $E$. 
By Lemma \ref{lem--adjunction}, we see that $E|_{T}$ is exceptional over $S_{i}$. 
Thus, $E|_{T}$ is effective and exceptional over $S_{i}$. 
By Remark \ref{rem--div} (\ref{rem--div-(2)}) and the fact that $(X,\Delta)$ is log abundant over $Z$, which is the
 first condition of Corollary \ref{cor--logabund-preserved-mmp}, we see that $K_{S_{i}}+\Delta_{S_{i}}$ is abundant over $Z$. 
Therefore, $(X_{i},\Delta_{i})$ is log abundant over $Z$. 

The above argument shows that the $(K_{X}+\Delta)$-MMP preserves the property of being log abundant over $Z$. 
Furthermore, if we define 
$$\lambda_{i}={\rm inf}\set{\mu \in \mathbb{R}_{\geq0} \!|\! K_{X_{i}}+\Delta_{i}+\mu A_{i}\text{\rm \, is nef over }Z}$$
 for each $i$, where $A_{i}$ is the birational transform of $A$ on $X_{i}$, then we have ${\rm lim}_{i \to \infty}\lambda_{i}=0$ by \cite[Theorem 4.1 (ii)]{birkar-flip}. 
By Theorem \ref{thm--main-logabundant}, the $(K_{X}+\Delta)$-MMP terminates with a log minimal model $(X_{m},\Delta_{m})$ over $Z$ that is log abundant over $Z$. 
Then $(X_{m},\Delta_{m})$ is a good minimal model of $(X,\Delta)$ over $Z$ by Lemma \ref{lem--relnefabundance}. 
\end{proof}

We close this subsection with applications of Theorem \ref{thm--ind-1}. 

\begin{lem}\label{lem--mmp-can-bundle-formula}
Let $\pi\colon X\to Z$ be a projective morphism from a normal variety to a projective variety such that ${\rm dim}\,Z\leq 4$, and let $(X,\Delta)$ be a projective lc pair such that $K_{X}+\Delta\sim_{\mathbb{R},Z}0$.  
If $K_{X}+\Delta$ is pseudo-effective, then $(X,\Delta)$ has a log minimal model. 
\end{lem}

\begin{proof}
By taking the Stein factorization of $\pi$, we may assume that $Z$ is normal and $\pi$ is a contraction. 
By the argument using the convex geometry and Ambro's canonical bundle formula \cite{fg-bundle}, we can find a generalized lc pair $(Z,B_{Z}+M_{Z})$ such that 
$$K_{X}+\Delta\sim_{\mathbb{R}}\pi^{*}(K_{Z}+B_{Z}+M_{Z})$$
 and the nef part of $(Z,B_{Z}+M_{Z})$ is a finite $\mathbb{R}_{>0}$-linear combination of b-nef $\mathbb{Q}$-b-Cartier b-divisors (see \cite{bz} for generalized pairs). 
Since ${\rm dim}\,Z\leq 4$, by \cite[Theorem 4.1]{lt}, there is a log minimal model $(Z',B_{Z'}+M_{Z'})$ of $(Z,B_{Z}+M_{Z})$ as a generalized lc pair. 
By taking a common resolution of $Z \dashrightarrow Z'$ and the negativity lemma,  we can find a resolution $g\colon Z'\to Z$ such that $g^{*}(K_{Z}+B_{Z}+M_{Z})$ has the Nakayama--Zariski decomposition with nef positive part. 
Take a resolution $f\colon X'\to X$ so that the induced map $\pi'\colon X'\dashrightarrow Z'$ is a morphism. 
By \cite[III, 5.17 Corollary]{nakayama} and since 
$$f^{*}(K_{X}+\Delta)\sim_{\mathbb{R}} \pi'^{*}g^{*}(K_{Z}+B_{Z}+M_{Z}),$$
the positive part of the Nakayama--Zariski decomposition of $f^{*}(K_{X}+\Delta)$ is nef. 
Then $(X,\Delta)$ has a log minimal model by Theorem \ref{thmzariski}.  
\end{proof}

\begin{cor}\label{cor--appli}
Let $\pi\colon X \to Z$ be a projective morphism of normal quasi-projective varieties, and let $(X,\Delta)$ be an lc pair.
Suppose that 
\begin{itemize}
\item
${\rm dim}\,X=6$,
\item
$K_{X}+\Delta$ is pseudo-effective over $Z$ and $\kappa_{\iota}(X/Z,K_{X}+\Delta)\geq 3-{\rm dim}\,\pi(X)$, and
\item
$\llcorner \Delta \lrcorner =0$.
\end{itemize}
Then $(X,\Delta)$ has a log minimal model over $Z$. 
In particular, all projective lc pairs $(X,0)$ satisfying ${\rm dim}\,X=6$ and $\kappa(X,K_{X})\geq 3$ have log minimal models.    
\end{cor}

\begin{proof}
We first show that we may assume that $Z$ is a point. 
By \cite[Corollary 1.3]{has-mmp}, there is an lc closure $(\overline{X},\overline{\Delta})$ of $(X,\Delta)$, that is, a projective lc pair $(\overline{X},\overline{\Delta})$ and a projective morphism $\overline{\pi}\colon \overline{X} \to \overline{Z}$ such that $X$ (resp.~$Z$) is an open subset of $\overline{X}$ (resp.~$\overline{Z}$), $\overline{\pi}|_{X}=\pi$, $(\overline{X}|_{X},\overline{\Delta}|_{X})=(X,\Delta)$, and all lc centers of $(\overline{X},\overline{\Delta})$ intersect $X$. 
Let $A_{\overline{Z}}$ be a sufficiently ample $\mathbb{R}$-divisor on $\overline{Z}$.  By taking $A_{\overline{Z}}$ generally (see \cite[pp.~112--113]{hashizumehu} for the generality), we may assume $\llcorner (\overline{\Delta}  +\overline{\pi}^{*}A_{\overline{Z}})\lrcorner=0$. 
Moreover, the second condition of Corollary \ref{cor--appli} implies 
$$\kappa_{\iota}(\overline{X},K_{\overline{X}}+\overline{\Delta}  +\overline{\pi}^{*}A_{\overline{Z}}) \geq 3.$$ 
Since $A_{\overline{Z}}$ is sufficiently ample, the existence of a log minimal model of $(X,\Delta)$ over $Z$ follows from that of $(\overline{X},\overline{\Delta}+\overline{\pi}^{*}A_{\overline{Z}})$. 
Replacing $(X,\Delta) \to Z$ by $(\overline{X},\overline{\Delta}+\overline{\pi}^{*}A_{\overline{Z}})\to {\rm Spec}\,\mathbb{C}$, we may assume that $Z$ is a point.

Next, we prove that $K_{X}+\Delta$ is abundant. 
Note that we do not need the condition $\llcorner \Delta \lrcorner =0$ to prove the fact. 
Therefore, in this paragraph we will freely replace $(X,\Delta)$ by a log smooth model (\cite[Definition 2.9]{has-trivial}). 
By replacing $(X,\Delta)$, we may assume that the Iitaka fibration $X \dashrightarrow V$ associated to $K_{X}+\Delta$ is a morphism. 
Since $\kappa_{\iota}(X,K_{X}+\Delta)\geq 3$ and ${\rm dim}\,X=6$, the general fibers $F$ of the Iitaka fibration satisfy ${\rm dim}\,F\leq3$. 
Then $(F,\Delta|_{F})$ has a good minimal model. 
Therefore $\kappa_{\sigma}(X/V, K_{X}+\Delta)=\kappa_{\iota}(F, K_{F}+\Delta|_{F})=0$. 
By \cite[V, 2.7 Proposition (9)]{nakayama}, we have 
$$\kappa_{\sigma}(X, K_{X}+\Delta)\leq\kappa_{\sigma}(X/V, K_{X}+\Delta)+{\rm dim}\,V={\rm dim}\,V.$$
Since ${\rm dim}\,V=\kappa_{\iota}(X,K_{X}+\Delta)$, we have $\kappa_{\sigma}(X, K_{X}+\Delta)\leq\kappa_{\iota}(X,K_{X}+\Delta)$. 
This implies that $K_{X}+\Delta$ is abundant. 

By Lemma \ref{lem--mmp-termi}, we can construct a sequence of  steps of a $(K_{X}+\Delta)$-MMP with scaling of an ample divisor $H$
$$(X_{0}:=X,\Delta_{0}:=\Delta) \dashrightarrow (X_{1},\Delta_{1}) \dashrightarrow\cdots \dashrightarrow (X_{i},\Delta_{i})\dashrightarrow \cdots$$
such that if we define 
$$\lambda_{i}={\rm inf}\set{\mu \in \mathbb{R}_{\geq0} \!|\! K_{X_{i}}+\Delta_{i}+\mu H_{i}\text{\rm \, is nef}}$$
 and $\lambda_{\infty}:={\rm lim}_{i\to\infty}\lambda_{i}$, where $H_{i}$ is the birational transform of $H$ on $X_{i}$, then $\lambda_{\infty}=0$. 
Lifting this log MMP (\cite[Remark 2.9]{birkar-flip}), we get a dlt blow-up $f\colon (Y,\Gamma)\to (X,\Delta)$ and a sequence of steps of a $(K_{Y}+\Gamma)$-MMP with scaling of $f^{*}H$
$$(Y_{0}:=Y,\Gamma_{0}:=\Gamma) \dashrightarrow (Y_{k_{1}},\Gamma_{k_{1}}) \dashrightarrow\cdots \dashrightarrow (Y_{k_{i}},\Gamma_{k_{i}})\dashrightarrow \cdots$$
such that $(Y_{k_{i}},\Gamma_{k_{i}})$ is a $\mathbb{Q}$-factorial dlt model of $(X_{i},\Delta_{i})$ for each $i$. 

For each $i$ and any lc center $T_{k_{i}}$ of $(Y_{k_{i}},\Gamma_{k_{i}})$, we consider the dlt pair $(T_{k_{i}},\Gamma_{T_{k_{i}}})$ defined by adjunction $K_{T_{k_{i}}}+\Gamma_{T_{k_{i}}}=(K_{Y_{k_{i}}}+\Gamma_{k_{i}})|_{T_{k_{i}}}$. 
Since $\llcorner \Delta \lrcorner =0$ by the hypothesis, we have $\llcorner \Delta_{i} \lrcorner =0$. 
Since ${\rm dim}\,X_{i}=6$, all lc centers of $(X_{i},\Delta_{i})$ have the dimension at most four. 
Therefore, the morphism $Y_{k_{i}}\to X_{i}$ induces a morphism $T_{k_{i}}\to S_{i}$ to a variety $S_{i}$ such that ${\rm dim}\,S_{i}\leq 4$ and $K_{Y_{k_{i}}}+\Gamma_{k_{i}}\sim_{\mathbb{R},S_{i}}0$. 
By Lemma \ref{lem--mmp-can-bundle-formula}, the pair $(T_{k_{i}},\Gamma_{T_{k_{i}}})$ has a log minimal model or a Mori fiber space. 

By the argument of the previous paragraph and the standard argument of the special termination (\cite{fujino-sp-ter}, Remark \ref{rem--mmplift}), there is a positive integer $m$ such that for any $i \geq k_{m}$, the non-isomorphic locus of the birational map $Y_{k_{m}}\dashrightarrow Y_{i}$ does not intersect $\llcorner \Gamma_{k_{m}}\lrcorner$. 
Since $\lambda_{\infty}=0$, by Lemma \ref{lem--mmpscaling-effective}, the restriction of 
 $K_{Y_{k_{m}}}+\Gamma_{k_{m}}$ to any component of $\llcorner \Gamma_{k_{m}}\lrcorner$ is nef, and $\sigma_{P}(K_{Y_{k_{m}}}+\Gamma_{k_{m}})=0$ for every prime divisor $P$ over $Y_{k_{m}}$ such that $c_{Y_{k_{m}}}(P)$ intersects an lc center of $(Y_{k_{m}},\Gamma_{k_{m}})$. 
Because $K_{X}+\Delta$ is abundant, we see that $K_{Y_{k_{m}}}+\Gamma_{k_{m}}$ is abundant. 
By Theorem \ref{thm--ind-1}, we see that $(Y_{k_{m}},\Gamma_{k_{m}})$ has a log minimal model. 
Then $(Y,\Gamma)$ has a log minimal model, thus, $(X,\Delta)$ has a log minimal model. 

The second assertion directly follows from the first assertion. 
\end{proof} 

\begin{cor}\label{cor--appli-2}
Let $\pi\colon X \to Z$ be a projective morphism of normal quasi-projective varieties and $(X,\Delta)$ an lc pair. 
Suppose that
\begin{itemize}
\item
$K_{X}+\Delta$ is pseudo-effective over $Z$,
\item
$\kappa_{\iota}(X/Z,K_{X}+\Delta)\geq {\rm dim}\,X-3-{\rm dim}\,\pi(X)$, and 
\item
all lc centers $S$ of $(X,\Delta)$ satisfy ${\rm dim}\,S\leq 4$. 
\end{itemize}
Then $(X,\Delta)$ has a log minimal model over $Z$. 
\end{cor}

Corollary \ref{cor--appli-2} looks similar to the case $n=3$ of \cite[Theorem 1.2]{hashizumehu}. 
The difference between these two results is that Corollary \ref{cor--appli-2} allows the existence of lc centers of dimension four but only shows the existence of a log minimal model. 

\begin{proof}[Proof of Corollary \ref{cor--appli-2}]
The proof is very similar to that of Corollary \ref{cor--appli}. 
So we only outline the proof. 

First, by \cite[Corollary 1.3]{has-mmp} and the same argument as in the first paragraph of the proof of Corollary \ref{cor--appli}, we may assume that $Z$ is a point. 
Note that the three conditions of Corollary \ref{cor--appli-2} are preserved after this reduction. 
Next, by the same argument as in the second paragraph of the proof of Corollary \ref{cor--appli}, we see that $K_{X}+\Delta$ is abundant. 

By Lemma \ref{lem--mmp-termi}, we can construct a sequence of  steps of a $(K_{X}+\Delta)$-MMP with scaling of an ample divisor $H$
$$(X_{0}:=X,\Delta_{0}:=\Delta) \dashrightarrow (X_{1},\Delta_{1}) \dashrightarrow\cdots \dashrightarrow (X_{i},\Delta_{i})\dashrightarrow \cdots$$
such that if we define 
$$\lambda_{i}={\rm inf}\set{\mu \in \mathbb{R}_{\geq0} \!|\! K_{X_{i}}+\Delta_{i}+\mu H_{i}\text{\rm \, is nef}},$$
where $H_{i}$ is the birational transform of $H$ on $X_{i}$, then the $(K_{X}+\Delta)$-MMP terminates or ${\rm lim}_{i\to\infty}\lambda_{i}=0$. 
By lifting this log MMP (\cite[Remark 2.9]{birkar-flip}), we obtain a dlt blow-up $f\colon (Y,\Gamma)\to (X,\Delta)$ and a sequence of steps of a $(K_{Y}+\Gamma)$-MMP with scaling of $f^{*}H$ 
$$(Y_{0}:=Y,\Gamma_{0}:=\Gamma) \dashrightarrow (Y_{k_{1}},\Gamma_{k_{1}}) \dashrightarrow\cdots \dashrightarrow (Y_{k_{i}},\Gamma_{k_{i}})\dashrightarrow \cdots$$
such that $(Y_{k_{i}},\Gamma_{k_{i}})$ is a $\mathbb{Q}$-factorial dlt model of $(X_{i},\Delta_{i})$ for each $i$. 

For each $i$ and any lc center $T_{k_{i}}$ of $(Y_{k_{i}},\Gamma_{k_{i}})$, we consider the dlt pair $(T_{k_{i}},\Gamma_{T_{k_{i}}})$ defined by adjunction $K_{T_{k_{i}}}+\Gamma_{T_{k_{i}}}=(K_{Y_{k_{i}}}+\Gamma_{k_{i}})|_{T_{k_{i}}}$. 
Let $S_{i}$ be the image of $T_{k_{i}}$ by the morphism $Y_{k_{i}}\to X_{i}$. 
Then $S_{i}$ is an lc center of $(X_{i},\Delta_{i})$, and there is an lc center $S$ of $(X,\Delta)$ such that the birational map $X\dashrightarrow X_{i}$ induces a birational map $S \dashrightarrow S_{i}$. 
In particular, we have ${\rm dim}\,S_{i}={\rm dim}\,S\leq 4$, where the inequality follows from the third condition of Corollary \ref{cor--appli-2}. 
Therefore, the induced morphism $T_{k_{i}}\to S_{i}$ satisfies ${\rm dim}\,S_{i}\leq 4$ and $K_{Y_{k_{i}}}+\Gamma_{k_{i}}\sim_{\mathbb{R},S_{i}}0$. 
By Lemma \ref{lem--mmp-can-bundle-formula}, the dlt pair $(T_{k_{i}},\Gamma_{T_{k_{i}}})$ has a log minimal model or a Mori fiber space. 

Finally, by arguing as in the end of the proof of Corollary \ref{cor--appli}, we deduce that $(X,\Delta)$ has a log minimal model, as claimed.  
\end{proof}

\section{On log MMP for semi-log canonical pairs}\label{sec4}

In this section, by applying Corollary \ref{cor--logabund-preserved-mmp} and the main result of \cite{ambrokollar}, we discuss the log MMP for semi-log canonical (slc, for short) pairs. 
We assume that all schemes in this section are reduced separated schemes of finite type over ${\rm Spec}\,\mathbb{C}$. 
For definition of slc pairs, see \cite[Definition 4.13.3]{fujino-book}. 
We only deal with projective slc pairs, that is, slc pairs $(X,\Delta)$ such that $X$ is projective over ${\rm Spec}\,\mathbb{C}$. 

The goal of this section is to prove the following generalization of Corollary \ref{cor--logabund-preserved-mmp} and to discuss its applications.  

\begin{thm}\label{thm--mmpslclogabund}
Let $(X,\Delta)$ be a projective slc pair such that $\Delta$ is a $\mathbb{Q}$-divisor. 
Let $\nu \colon (\bar{X},\bar{\Delta})\to (X,\Delta)$ be the normalization, where $K_{\bar{X}}+\bar{\Delta}=\nu^{*}(K_{X}+\Delta)$. 
Suppose that every irreducible component $(\bar{X}^{(j)},\bar{\Delta}^{(j)})$ of $(\bar{X},\bar{\Delta})$ satisfies the following conditions. 
\begin{itemize}
\item
$(\bar{X}^{(j)},\bar{\Delta}^{(j)})$ is log abundant and $K_{\bar{X}^{(j)}}+\bar{\Delta}^{(j)}$ is pseudo-effective, and
\item
the stable base locus of $K_{\bar{X}^{(j)}}+\bar{\Delta}^{(j)}$ does not contain the image of any prime divisor $P$ over $\bar{X}^{(j)}$ satisfying $a(P,\bar{X}^{(j)},\bar{\Delta}^{(j)})<0$. 
\end{itemize}
Then there is a sequence of MMP steps for $(X,\Delta)$
$$(X,\Delta)\dashrightarrow \cdots \dashrightarrow (X_{i},\Delta_{i})\dashrightarrow \cdots \dashrightarrow (X_{m},\Delta_{m})$$
such that $K_{X_{m}}+\Delta_{m}$ is semi-ample. 
\end{thm}

We recall the definition of an MMP step for projective slc pairs. 

\begin{defn}[MMP step, {\cite[Definition 11]{ambrokollar}}]\label{defn--mmpslc}
Let $(X,\Delta)$ be a projective slc pair such that $\Delta$ is a $\mathbb{Q}$-divisor. 
Then an {\em MMP step} for $(X,\Delta)$ is a diagram  
$$
\xymatrix
{
(X,\Delta)\ar[dr]_{f}\ar@{-->}[rr]^{\phi}&&(X',\Delta')\ar[dl]^{f'}\\
&Z
}
$$
such that
\begin{enumerate}
\item \label{mmpslc-1}
$(X',\Delta')$ is an slc pair such that $X'$ is projective,
\item \label{mmpslc-2}
$\phi$ is birational and an isomorphism on an open dense subset containing all generic points of codimension one singular locus,  
\item \label{mmpslc-3}
$\Delta'=\phi_{*}\Delta$,  
\item \label{mmpslc-4}
$Z$ is a projective scheme, $f$ and $f'$ are generically finite morphisms, and $f'$ has no exceptional divisors, and 
\item \label{mmpslc-5}
$-(K_{X}+\Delta)$ and $K_{X'}+\Delta'$ are ample over $Z$. 
\end{enumerate}
Let $H$ be a $\mathbb{Q}$-Cartier $\mathbb{Q}$-divisor on $X$. 
Then the above diagram is an {\em MMP step with scaling of} $H$ if the diagram satisfies (\ref{mmpslc-1})--(\ref{mmpslc-5}) stated above and
\begin{enumerate}\setcounter{enumi}{5}
\item \label{mmpslc-6}
$H$ is $f$-ample and $-H':=-\phi_{*}H$ is an $f'$-ample $\mathbb{Q}$-Cartier $\mathbb{Q}$-divisor on $X'$, 
\item \label{mmpslc-7}
$K_{X}+\Delta+\lambda H$ is numerically trivial over $Z$ for some $\lambda\in \mathbb{Q}$, and
\item \label{mmpslc-8}
for every proper curve $C$ on $X$, if $C$ is not contracted by $f$ then the strict inequality $(K_{X}+\Delta+\lambda H)\cdot C>0$ holds.   
\end{enumerate}
\end{defn}

The above definition is different from the usual log MMP even if $(X,\Delta)$ is lc. 
The following example illustrates the difference between the usual log MMP and the log MMP in the sense of Definition \ref{defn--mmpslc}. 
\begin{exam}
Let $(X,\Delta)$ be a projective lc pair such that $\Delta$ is a $\mathbb{Q}$-divisor and $K_{X}+\Delta$ is pseudo-effective but not nef. 
Let $A$ be an ample $\mathbb{Q}$-divisor on $X$ such that $K_{X}+\Delta+A$ is nef but not ample. 
Then $K_{X}+\Delta+A$ is semi-ample (\cite[Theorem 13.1]{fujino-fund}) and big, so it induces a generically finite morphism $f\colon X\to Z$ to a projective variety $Z$. 
Then $-(K_{X}+\Delta)\sim_{\mathbb{Q},Z}A$ is ample over $Z$ and a curve $C$ on $X$ is contracted by $f$ if and only if $(K_{X}+\Delta+A)\cdot C=0$. 
By \cite[Theorem 1.1]{birkar-flip} and \cite[Theorem 1.7]{hashizumehu}, we may run a $(K_{X}+\Delta)$-MMP over $Z$ and we may get a good minimal model $(X_{l},\Delta_{l})$ of $(X,\Delta)$ over $Z$. 
Let $X_{l}\to X'$ be the birational morphism over $Z$ induced by $K_{X_{l}}+\Delta_{l}$. 
Let $\Delta'$ (resp.~$A'$) be the birational transform of $\Delta$ (resp.~$A$) on $X'$. 
Then $(X',\Delta')$ is lc, $K_{X'}+\Delta'$ is ample over $Z$, and $K_{X'}+\Delta'+A'\sim_{\mathbb{Q},Z}0$. 

We can directly check that the birational map $(X,\Delta)\dashrightarrow (X',\Delta')$  over $Z$ is an MMP step with scaling of $A$ for $(X,\Delta)$ as in Definition \ref{defn--mmpslc}. However, this birational map is not a step of the usual $(K_{X}+\Delta)$-MMP unless $f$ is birational and the relative Picard number is one.  
\end{exam}

In general, we cannot construct MMP steps for slc pairs inductively. 
See, for example, \cite[Example 4 and Example 5]{ambrokollar} and \cite[Example 5.4]{fujino-fund-slc}. 
However, by \cite[Theorem 9]{ambrokollar}, we can construct MMP steps for slc pairs whose normalization satisfies the two conditions of Corollary \ref{cor--logabund-preserved-mmp}. 
We will prove this fact. 

\begin{lem}\label{lem--mmpstepslc}
Let $(X,\Delta)$ be a projective slc pair such that $\Delta$ is a $\mathbb{Q}$-divisor. 
Let $\nu \colon (\bar{X},\bar{\Delta})\to (X,\Delta)$ be the normalization, where $K_{\bar{X}}+\bar{\Delta}=\nu^{*}(K_{X}+\Delta)$. 
Suppose that every irreducible component $(\bar{X}^{(j)},\bar{\Delta}^{(j)})$ of $(\bar{X},\bar{\Delta})$ satisfies the following.  
\begin{itemize}
\item
$|K_{\bar{X}^{(j)}}+\bar{\Delta}^{(j)}|_{\mathbb{R}}\neq \emptyset$ and the stable base locus of $K_{\bar{X}^{(j)}}+\bar{\Delta}^{(j)}$ does not contain the image of any prime divisor $P$ over $\bar{X}^{(j)}$ satisfying $a(P,\bar{X}^{(j)},\bar{\Delta}^{(j)})<0$. 
\end{itemize}
Let $H$ be a $\mathbb{Q}$-Cartier $\mathbb{Q}$-divisor on $X$ such that $K_{X}+\Delta+c H$ is ample for some $c>0$.
We define
$$\lambda={\rm inf}\set{\mu \in \mathbb{R}_{\geq 0}|\text{$K_{X}+\Delta+\mu H$ is nef}}.$$ 
If $\lambda>0$, then we can construct a diagram of an MMP step with scaling of $H$ for $(X,\Delta)$ in Definition \ref{defn--mmpslc}
$$
\xymatrix
{
(X,\Delta)\ar[dr]_{f}\ar@{-->}[rr]^{\phi}&&(X',\Delta')\ar[dl]^{f'}\\
&Z
}
$$
such that $f_{*}\mathcal{O}_{X} \simeq \mathcal{O}_{Z}$ and $K_{X'}+\Delta'+c'H'$ is ample for some real number $c' \in (0, \lambda)$, where $H'=\phi_{*}H$. 
Furthermore, for the normalization $\nu'\colon(\bar{X}',\bar{\Delta}')\to (X',\Delta')$, where $K_{\bar{X}'}+\bar{\Delta}'=\nu'^{*}(K_{X}+\Delta)$, every irreducible component $(\bar{X}'^{(j)},\bar{\Delta}'^{(j)})$ satisfies the following.
\begin{itemize}
\item
$|K_{\bar{X}'^{(j)}}+\bar{\Delta}'^{(j)}|_{\mathbb{R}}\neq \emptyset$ and the stable base locus of $K_{\bar{X}'^{(j)}}+\bar{\Delta}'^{(j)}$ does not contain the image of any prime divisor $P'$ over $\bar{X}'^{(j)}$ satisfying $a(P',\bar{X}'^{(j)},\bar{\Delta}'^{(j)})<0$. 
\end{itemize}
\end{lem}

\begin{proof}
We have $\lambda<c$. 
By the rationality theorem \cite[Theorem 1.18]{fujino-fund-slc}, we see that $\lambda$ is a rational number. 
We have 
$$K_{X}+\Delta+\lambda H=(1-\frac{\lambda}{c})(K_{X}+\Delta)+\frac{\lambda}{c}(K_{X}+\Delta+cH).$$ 
Since $K_{X}+\Delta+cH$ is ample, $K_{X}+\Delta+\lambda H$ is semi-ample by \cite[Theorem 1.15]{fujino-fund-slc}. 
Thus, $K_{X}+\Delta+\lambda H$ induces a projective morphism $f\colon X\to Z$ to a projective scheme $Z$ such that $f_{*}\mathcal{O}_{X} \simeq \mathcal{O}_{Z}$. 
Since $K_{X}+\Delta+cH$ is ample and $K_{X}+\Delta+\lambda H\sim_{\mathbb{Q}}f^{*}A$ for some ample $\mathbb{Q}$-divisor $A$ on $Z$, it follows that $H$ and $-(K_{X}+\Delta)$ are ample over $Z$. 
Moreover, a proper curve $C$ on $X$ is contracted by $f$ if and only if $(K_{X}+\Delta+\lambda H)\cdot C=0$. We put $\bar{H}=\nu^{*}H$. 
For any irreducible component $\bar{X}^{(j)}$ of the normalization $\bar{X}\to X$, since $K_{\bar{X}}+\bar{\Delta}+c\bar{H}$ is ample and $|K_{\bar{X}^{(j)}}+\bar{\Delta}^{(j)}|_{\mathbb{R}}\neq \emptyset$, the pullback of $K_{X}+\Delta+\lambda H$ to $\bar{X}^{(j)}$ is big. 
This implies that $f$ is generically finite. 
In this way, $f\colon X\to Z$, $K_{X}+\Delta$, and $H$ satisfy (\ref{mmpslc-4}), (\ref{mmpslc-5}), (\ref{mmpslc-6}), (\ref{mmpslc-7}), and (\ref{mmpslc-8}) of Definition \ref{defn--mmpslc}. 

For each component $\bar{X}^{(j)}$ of the normalization $\bar{X}\to X$, let $\bar{X}^{(j)}\to Z^{(j)}$ be the Stein factorization of $\bar{X}^{(j)}\to Z$. 
This is a birational morphism. 
Since $K_{\bar{X}}+\bar{\Delta}+\lambda\bar{H}\sim_{\mathbb{Q},Z}0$, applying \cite[Theorem 1.1]{birkar-flip}, the lc pair $(\bar{X}^{(j)},\bar{\Delta}^{(j)})$ has a good minimal model over $Z^{(j)}$ for all $j$. 
Therefore $(\bar{X}^{(j)},\bar{\Delta}^{(j)})$ has the log canonical model $(\bar{X}'^{(j)},\bar{\Delta}'^{(j)})$ over $Z^{(j)}$, that is, a weak lc model over $Z^{(j)}$ such that $K_{\bar{X}'^{(j)}}+\bar{\Delta}'^{(j)}$ is ample over $Z^{(j)}$, for all $j$. 
We put $(\bar{X}',\bar{\Delta}')={\coprod_{j}}(\bar{X}'^{(j)},\bar{\Delta}'^{(j)})$, and we consider the following diagram
$$
\xymatrix
{
(\bar{X},\bar{\Delta})\ar[dr]\ar@{-->}[rr]&&(\bar{X}',\bar{\Delta}')\ar[dl]\\
&Z
}
$$
where the birational map $\bar{X}\dashrightarrow \bar{X}'$ is defined by each birational map $\bar{X}^{(j)} \dashrightarrow \bar{X}'^{(j)}$. 
By construction, this diagram is an MMP step of Definition \ref{defn--mmpslc} for the disjoint union of lc pairs $(\bar{X},\bar{\Delta})$. 
Indeed, (\ref{mmpslc-1}), (\ref{mmpslc-2}), and (\ref{mmpslc-3}) are obvious, and (\ref{mmpslc-4}) and (\ref{mmpslc-5}) follow from the fact that $(\bar{X}'^{(j)},\bar{\Delta}'^{(j)})$ is the log canonical model over $Z^{(j)}$ for all $j$. 
Since $K_{\bar{X}}+\bar{\Delta}+\lambda\bar{H}\sim_{\mathbb{Q},Z}0$ by construction, we have $K_{\bar{X}'}+\bar{\Delta}'+\lambda\bar{H}'\sim_{\mathbb{Q},Z}0$, where $\bar{H}'$ is the birational transform of $\bar{H}$ on $\bar{X}'$. 
Since $K_{\bar{X}'}+\bar{\Delta}'$ is ample over $Z$, we see that $-\bar{H}'$ is ample over $Z$. 
Therefore (\ref{mmpslc-6}) holds, and we have already checked that (\ref{mmpslc-7}) and (\ref{mmpslc-8}) hold. 
From this discussion, we see that the above diagram is an MMP step with scaling of $\bar{H}$. 

For any irreducible component $(\bar{X}'^{(j)},\bar{\Delta}'^{(j)})$, the hypothesis of Lemma \ref{lem--mmpstepslc} on $(\bar{X},\bar{\Delta})$ and the construction of the usual log MMP imply that $|K_{\bar{X}'^{(j)}}+\bar{\Delta}'^{(j)}|_{\mathbb{R}}\neq \emptyset$ and its stable base locus does not contain the image of any prime divisor $P'$ over $\bar{X}'^{(j)}$ satisfying $a(P',\bar{X}'^{(j)},\bar{\Delta}'^{(j)})<0$. 
By taking a common resolution of $\bar{X}^{(j)}\dashrightarrow \bar{X}'^{(j)}$ and applying (\ref{mmpslc-5}) and the negativity lemma, we see that the non-isomorphic locus of  $\bar{X}^{(j)}\dashrightarrow \bar{X}'^{(j)}$ is contained in the stable base locus of $K_{\bar{X}^{(j)}}+\bar{\Delta}^{(j)}$. 
From these facts, the above MMP step satisfies the condition of \cite[Theorem 9]{ambrokollar}, therefore we can apply \cite[Theorem 9]{ambrokollar}. We get a projective slc pair $(X',\Delta')$ and a projective morphism $f'\colon X'\to Z$ such that $(\bar{X}',\bar{\Delta}')$ is the normalization of $(X',\Delta')$. 
By \cite[Lemma 12]{ambrokollar}, the induced diagram
$$
\xymatrix
{
(X,\Delta)\ar[dr]_{f}\ar@{-->}[rr]^{\phi}&&(X',\Delta')\ar[dl]^{f'}\\
&Z
}
$$
is an MMP step with scaling of $H$ for $(X,\Delta)$. 

We recall from the construction that $K_{X}+\Delta+\lambda H$ is $\mathbb{Q}$-linearly equivalent to the pullback of an ample $\mathbb{Q}$-divisor on $Z$. 
Hence the same holds for $K_{X'}+\Delta'+\lambda H'$, where $H'=\phi_{*}H$.  
Since $K_{X'}+\Delta'$ is ample over $Z$, we see that $K_{X'}+\Delta'+\lambda H'+t(K_{X'}+\Delta')$ is ample for some $t>0$. 
Put $c'=\frac{\lambda}{1+t}$. 
Since ${\coprod_{j}}(\bar{X}'^{(j)},\bar{\Delta}'^{(j)})=(\bar{X}',\bar{\Delta}')$ satisfies the condition of Lemma \ref{lem--mmpstepslc}, the above diagram is the desired MMP step for $(X,\Delta)$. 
\end{proof}

\begin{rem}\label{rem--mmp-isomincodim1}
Let 
$$
\xymatrix
{
(X,\Delta)\ar[dr]_{f}\ar@{-->}[rr]^{\phi}&&(X',\Delta')\ar[dl]^{f'}\\
&Z
}
$$
be an MMP step as in Definition \ref{defn--mmpslc}. 
We show that if $f_{*}\mathcal{O}_{X} \simeq \mathcal{O}_{Z}$ then $\phi^{-1}$ is an isomorphism on an open subset of $X'$ containing all the generic points of codimension one singular locus. 
We take the graph $g \colon W \to X$ and $g'\colon W \to X'$ of $\phi$. 
Since $X$ is an S$_{2}$ scheme and $\phi$ is an isomorphism at all codimension one singular points of $X$, we have $g_{*}\mathcal{O}_{W} \simeq \mathcal{O}_{X}$. 
Since we assume $f_{*}\mathcal{O}_{X} \simeq \mathcal{O}_{Z}$, we obtain $\mathcal{O}_{Z} \simeq (f\circ g)_{*}\mathcal{O}_{W} \simeq (f' \circ g')_{*}\mathcal{O}_{W}$. 
This implies that $f'\colon X' \to Z$ has connected fibers. 
Since $f'$ has no exceptional divisors by (\ref{mmpslc-4}) in Definition \ref{defn--mmpslc}, for any codimension one point $\eta$ of $X'$, the image $f'(\eta)$ is a codimension one point of $Z$. 
Then $f^{-1}(f'(\eta))$ is a point because $f$ is birational and $f$ has connected fibers. 
Therefore, $f$ and $f'$ are isomorphisms over $f'(\eta)$. 
From this fact, $\phi^{-1}$ is an isomorphism on an open subset of $X'$ containing all the generic points of codimension one singular locus. 
\end{rem}

\begin{rem}\label{rem--mmpslc}
Let $(X,\Delta)$ be an slc pair as in Lemma \ref{lem--mmpstepslc}. 
Let $\phi \colon (X,\Delta)\dashrightarrow (X',\Delta')$ be an MMP step constructed in Lemma \ref{lem--mmpstepslc}. 
By the construction in Lemma \ref{lem--mmpstepslc} and Remark \ref{rem--mmp-isomincodim1}, it follows that $\phi^{-1}$ is an isomorphism on an open subset of $X'$ containing all the generic points of codimension one singular locus. 
Using this, we check
$$
{\rm dim}\,H^{0}(X,\mathcal{O}_{X}(\llcorner m(K_{X}+\Delta)\lrcorner))={\rm dim}\,H^{0}(X',\mathcal{O}_{X'}(\llcorner m(K_{X'}+\Delta')\lrcorner)).
$$
for every positive integer $m$. 

Let $U_{1}$ be the largest open subset of $X$ on which $\phi$ is an isomorphism, and let $U_{2}$ be the largest open subset of $X$ such that $U_{2}$ is smooth and $\phi$ is a morphism on $U_{2}$. 
We put $U=U_{1} \cup U_{2}$ and $U'=\phi(U_{1})$. 
Then $\phi$ is a morphism on $U$, and $U$ and $U'$ contain all the generic points of codimension one singular locus of $X$ and $X'$ respectively because $\phi$ and $\phi^{-1}$ are isomorphisms at all codimension one singular points. 
Hence, the codimension of $X\setminus U$ in $X$ is at least two, and the codimension of $X'\setminus U'$ in $X'$ is also at least two. 

Fix an integer $m>0$. 
Since $U$ is demi-normal, for every Weil divisor on $U$, we can define the associated divisorial sheaf. 
We put $g'=\phi|_{U}\colon U \to X'$. 
Then $g'$ is a dominant morphism, and therefore there is a natural injective morphism
\begin{equation*}
\begin{split}
H^{0}(X', \mathcal{O}_{X'}(\llcorner m(K_{X'}+\Delta')\lrcorner))\hookrightarrow H^{0}(U, \mathcal{O}_{U}(\llcorner mg'^{*}(K_{X'}+\Delta')\lrcorner)). 
\end{split}
\end{equation*} 

Let $k$ be a positive integer such that $km(K_{X}+\Delta)$ and $km(K_{X'}+\Delta')$ are both Cartier.
We consider the Cartier divisor 
$$E:=km(K_{X}+\Delta)|_{U}-kmg'^{*}(K_{X'}+\Delta').$$ 
Since $\phi$ is an isomorphism on $U_{1}$, we see that $E|_{U_{1}}=0$, hence ${\rm Supp}\,E$ is contained in $U_{2}$. 
Since $U_{2}$ is smooth, we can think of $E$ as a Weil divisor on $U_{2}$. 
We take the normalizations $\nu \colon \bar{X}\to X$ and $\nu'\colon \bar{X}'\to X'$, and we take the graph $W\to \bar{X}$ and $W\to \bar{X}'$ of $\bar{X}\dashrightarrow \bar{X}'$. 
Since $U_{2}$ is smooth and $\phi$ is a morphism on $U_{2}$, we can think of $U_{2}$ as a open subset of $W$. 
By the standard argument as in \cite[Lemma 3.38]{kollar-mori} with the negativity lemma, we see that $E|_{U_{2}}$ is effective. 
Then $E$ is effective since ${\rm Supp}\,E$ is contained in $U_{2}$. 
From this, we obtain
$$\llcorner m(K_{X}+\Delta)|_{U} \lrcorner-\llcorner mg'^{*}(K_{X'}+\Delta') \lrcorner\geq0$$ 
as a relation of Weil divisors on $U$. 
By the above argument, there is a natural injective morphism 
$$ H^{0}(U, \mathcal{O}_{U}(\llcorner mg'^{*}(K_{X'}+\Delta')\lrcorner))\hookrightarrow H^{0}(U, \mathcal{O}_{U}(\llcorner m(K_{X}+\Delta)|_{U} \lrcorner)).$$
Since the codimension of $X\setminus U$ in $X$ is at least two, we obtain a natural isomorphism 
$$H^{0}(U, \mathcal{O}_{U}(\llcorner m(K_{X}+\Delta)|_{U} \lrcorner))\overset{\simeq}{\longrightarrow}H^{0}(X, \mathcal{O}_{X}(\llcorner m(K_{X}+\Delta) \lrcorner)).$$
Thus, we get an injective morphism 
\begin{equation*}\tag{$*$}\label{rem--mmpslc-*}
H^{0}(X', \mathcal{O}_{X'}(\llcorner m(K_{X'}+\Delta') \lrcorner ))\hookrightarrow H^{0}(X, \mathcal{O}_{X}(\llcorner m(K_{X}+\Delta) \lrcorner)).
\end{equation*}

We recall that $U_{1}$ is the largest open subset of $X$ on which $\phi$ is an isomorphism and the codimension of $X'\setminus U'$ in $X'$ is at least two, where $U'=\phi(U_{1})$. 
Thus, we obtain 
\begin{equation*}\tag{$**$}\label{rem--mmpslc-**}
\begin{split}
H^{0}(X, \mathcal{O}_{X}(\llcorner m(K_{X}+\Delta) \lrcorner))\hookrightarrow &H^{0}(U_{1}, \mathcal{O}_{U_{1}}(\llcorner m(K_{X}+\Delta)|_{U_{1}}\lrcorner))\\
\overset{\simeq}{\longrightarrow} &H^{0}(U', \mathcal{O}_{U'}(\llcorner m(K_{X'}+\Delta')|_{U'} \lrcorner))\\
\overset{\simeq}{\longrightarrow} &H^{0}(X', \mathcal{O}_{X'}(\llcorner m(K_{X'}+\Delta')\lrcorner)).
\end{split}
\end{equation*}
By (\ref{rem--mmpslc-*}) and (\ref{rem--mmpslc-**}), we have 
$${\rm dim}\,H^{0}(X, \mathcal{O}_{X}(\llcorner m(K_{X}+\Delta)\lrcorner)) = {\rm dim}\,H^{0}(X', \mathcal{O}_{X'}(\llcorner m(K_{X'}+\Delta')\lrcorner)).$$
\end{rem}

Now we prove Theorem \ref{thm--mmpslclogabund}.

\begin{proof}[Proof of Theorem \ref{thm--mmpslclogabund}]
We put $X=X_{0}$ and $\Delta=\Delta_{0}$.
Let $H_{0}$ be an effective ample $\mathbb{Q}$-Cartier $\mathbb{Q}$-divisor on $X_{0}$ such that $(X_{0},\Delta_{0}+H_{0})$ is slc and $K_{X_{0}}+\Delta_{0}+H_{0}$ is nef.
By using Lemma \ref{lem--mmpstepslc} repeatedly, we get a sequence of MMP steps with scaling of $H_{0}$
$$(X_{0},\Delta_{0})\dashrightarrow \cdots \dashrightarrow (X_{i},\Delta_{i})\dashrightarrow \cdots.$$
For each $i \geq 0$, we set 
$$\lambda_{i}={\rm inf}\set{\mu \in \mathbb{R}_{\geq 0}|\text{$K_{X_{i}}+\Delta_{i}+\mu H_{i}$ is nef}},$$
where $H_{i}$ is the birational transform of $H_{0}$ on $X_{i}$. 
By construction of each MMP step, we have $\lambda_{i}>\lambda_{i+1}$ for all $i\geq 0$. 
For each $i \geq 0$, let $(\bar{X}_{i},\bar{\Delta}_{i})$ be the normalization of $(X_{i},\Delta_{i})$. 
Let $\bar{H}_{0}$ be the pullback of $H_{0}$ to $\bar{X}_{0}$. 
Let $\bar{H}_{i}$ be the birational transform of $\bar{H}_{0}$ on $\bar{X}_{i}$. 
By \cite[Lemma 12]{ambrokollar}, the sequence of birational maps
$$(\bar{X}_{0},\bar{\Delta}_{0})\dashrightarrow \cdots \dashrightarrow (\bar{X}_{i},\bar{\Delta}_{i})\dashrightarrow \cdots$$
is a sequence of MMP steps for $(\bar{X}_{0},\bar{\Delta}_{0})$ with scaling of $\bar{H}_{0}$, and we see that 
$$\lambda_{i}={\rm inf}\set{\mu \in \mathbb{R}_{\geq 0}|\text{$K_{\bar{X}_{i}}+\bar{\Delta}_{i}+\mu \bar{H}_{i}$ is nef}}$$ 
because $\bar{H}_{i}$ is equal to the pullback of $H_{i}$ to $\bar{X}_{i}$. 
In particular, $K_{\bar{X}_{i}}+\bar{\Delta}_{i}+u \bar{H}_{i}$ is nef for all $u\in[\lambda_{i},\lambda_{i-1}]$. 

To complete the proof, it is sufficient to prove that $\lambda_{m}=0$ for some $m \geq 0$ and that $K_{X_{m}}+\Delta_{m}$ is semi-ample. 
By \cite[Theorem 2]{haconxu} or \cite[Theorem 1.5]{fujino-gongyo}, the semi-ampleness of $K_{X_{m}}+\Delta_{m}$ follows from that of $K_{\bar{X}_{m}}+\bar{\Delta}_{m}$. 
Therefore, we may replace $(X_{0},\Delta_{0})$ by any irreducible component of $(\bar{X}_{0},\bar{\Delta}_{0})$ and the MMP steps accordingly. 
In this way, we may assume that $(X_{i},\Delta_{i})$ is lc for every $i \geq 0$. 
Note that after the replacement, the birational maps $(X_{i},\Delta_{i})\dashrightarrow (X_{i+1},\Delta_{i+1})$ can be isomorphisms for some $i$. 
By identifying all $(X_{i},\Delta_{i})$ such that $\lambda_{i}$ are the same values, we may assume that $\lambda_{i}>\lambda_{i+1}$ still holds for all $i\geq 0$. 
The construction of the MMP steps and the negativity lemma show that $(X_{i},\Delta_{i}+u H_{i})$ is a weak lc model of $(X_{0},\Delta_{0}+uH_{0})$ for any $u\in[\lambda_{i},\lambda_{i-1}]$. 

By Corollary \ref{cor--logabund-preserved-mmp} and the hypothesis of Theorem \ref{thm--mmpslclogabund}, we see that the lc pair $(X_{0},\Delta_{0})$ has a good minimal model. 
By \cite[Theorem 1.7]{hashizumehu}, we can construct a sequence of steps of the usual $(K_{X_{0}}+\Delta_{0})$-MMP with scaling of $H_{0}$
$$(X'_{0}:=X_{0},\Delta'_{0}:=\Delta_{0})\dashrightarrow \cdots \dashrightarrow (X'_{j},\Delta'_{j})\dashrightarrow \cdots \dashrightarrow (X'_{n},\Delta'_{n})$$ 
such that $K_{X'_{n}}+\Delta'_{n}$ is semi-ample. 
We set $H'_{0}=H_{0}$ and define
$$\lambda'_{j}={\rm inf}\set{\mu \in \mathbb{R}_{\geq 0}|\text{$K_{X'_{j}}+\Delta'_{j}+\mu H'_{j}$ is nef}},$$
where $H'_{j}$ is the birational transform of $H'_{0}$ on $X'_{j}$. 

We show by induction that for any $0 \leq l \leq n$, there exists $k\geq 0$ such that $\lambda_{k}= \lambda'_{l}$. 
It is clear that $\lambda_{0}=\lambda'_{0}$. 
We prove the case of $l+1$ assuming $\lambda_{k}=\lambda'_{l}$ for some $k$.
If $\lambda'_{l+1}=\lambda'_{l}$, then $\lambda_{k}= \lambda'_{l+1}$. 
If $\lambda'_{l+1}<\lambda'_{l}$, then $(X'_{l+1},\Delta'_{l+1}+u'H'_{l+1})$ is a weak lc model of $(X_{0},\Delta_{0}+u'H_{0})$ for all $u'\in[\lambda'_{l+1},\lambda'_{l}]$. 
Since $(X_{k+1},\Delta_{k+1}+u H_{k+1})$ is a weak lc model of $(X_{0},\Delta_{0}+uH_{0})$ for any $u\in[\lambda_{k+1},\lambda_{k}]$, by taking a common resolution $g\colon W\to X_{k+1}$ and $g'\colon W\to X'_{l+1}$ of the induced birational map $X_{k+1}\dashrightarrow X'_{l+1}$ and the negativity lemma (see also \cite[Remark 2.7]{birkar-flip}), we have 
$$g^{*}(K_{X_{k+1}}+\Delta_{k+1}+u' H_{k+1})=g'^{*}(K_{X'_{l+1}}+\Delta'_{l+1}+u'H'_{l+1})$$
for any $u'\in [\lambda_{k+1},\lambda_{k}] \cap [\lambda'_{l+1},\lambda'_{l}]$. 
Since $\lambda_{k}=\lambda'_{l}$ and $\lambda_{k+1}<\lambda_{k}$, the relation shows
$$g^{*}(K_{X_{k+1}}+\Delta_{k+1}+t H_{k+1})=g'^{*}(K_{X'_{l+1}}+\Delta'_{l+1}+tH'_{l+1})$$
for any $t\in \mathbb{R}$. 
Thus, $K_{X_{k+1}}+\Delta_{k+1}+t H_{k+1}$ is nef if and only if $K_{X'_{l+1}}+\Delta'_{l+1}+tH'_{l+1}$ is nef. 
By definitions of $\lambda_{k+1}$ and $\lambda'_{l+1}$, we have $\lambda_{k+1}=\lambda'_{l+1}$. 

By the above argument, there is $m$ such that $\lambda_{m}=\lambda'_{n}=0$. 
Thus $(X_{m},\Delta_{m})$ is a weak lc model of $(X_{0},\Delta_{0})$. 
By taking a common resolution $h\colon \tilde{W} \to X_{m}$ and $h'\colon \tilde{W} \to X'_{n}$ of the birational map $X_{m}\dashrightarrow X'_{n}$ and the negativity lemma, we have
$$h^{*}(K_{X_{m}}+\Delta_{m})=h'^{*}(K_{X'_{n}}+\Delta'_{n}).$$
Then $K_{X_{m}}+\Delta_{m}$ is semi-ample (see \cite[Remark 2.10 (iii)]{has-mmp} and \cite[Remark 2.7]{birkar-flip}). 
We finish the proof of  Theorem \ref{thm--mmpslclogabund}. 
\end{proof}

\begin{cor}\label{cor--nonvanslc}
Let $(X,\Delta)$ be a projective slc pair such that $\Delta$ is a $\mathbb{Q}$-divisor. 
Let $\nu \colon (\bar{X},\bar{\Delta})\to (X,\Delta)$ be the normalization, where $K_{\bar{X}}+\bar{\Delta}=\nu^{*}(K_{X}+\Delta)$. 
Suppose that every irreducible component $(\bar{X}^{(j)},\bar{\Delta}^{(j)})$ of $(\bar{X},\bar{\Delta})$ satisfies the following conditions. 
\begin{itemize}
\item
$(\bar{X}^{(j)},\bar{\Delta}^{(j)})$ is log abundant and $K_{\bar{X}^{(j)}}+\bar{\Delta}^{(j)}$ is pseudo-effective, and
\item
the stable base locus of $K_{\bar{X}^{(j)}}+\bar{\Delta}^{(j)}$ does not contain the image of any prime divisor $P$ over $\bar{X}^{(j)}$ satisfying $a(P,\bar{X}^{(j)},\bar{\Delta}^{(j)})<0$. 
\end{itemize}
Then $H^{0}(X, \mathcal{O}_{X}(l(K_{X}+\Delta)))\neq \{0\}$ for some integer $l>0$.   
\end{cor}

\begin{proof}
By Theorem \ref{thm--mmpslclogabund} and Remark \ref{rem--mmpslc} (see also the proof of Lemma \ref{lem--mmpstepslc}), we may assume that $K_{X}+\Delta$ is semi-ample. 
Then the corollary is clear. 
\end{proof}

Finally, we prove that we can construct a semi-log canonical model (i.e., an slc pair $(Y,\Delta_{Y})$ such that $K_{Y}+\Delta_{Y}$ is ample) in a special situation. 
Let $(X,\Delta)$ be a projective lc pair. 
We say that $(X,\Delta)$ is {\em log big} if $K_{X}+\Delta$ is big and for any lc center $S$ of $(X,\Delta)$ with the normalization $S^{\nu} \to S$, the pullback $(K_{X}+\Delta)|_{S^{\nu}}$ is big.

\begin{thm}\label{thm--mmpslclogbig}
Let $(X,\Delta)$ be a projective slc pair such that $\Delta$ is a $\mathbb{Q}$-divisor. 
Let $\nu \colon (\bar{X},\bar{\Delta})\to (X,\Delta)$ be the normalization, where $K_{\bar{X}}+\bar{\Delta}=\nu^{*}(K_{X}+\Delta)$. 
Suppose that every irreducible component $(\bar{X}^{(j)},\bar{\Delta}^{(j)})$ of $(\bar{X},\bar{\Delta})$ satisfies the following conditions. 
\begin{itemize}
\item
$(\bar{X}^{(j)},\bar{\Delta}^{(j)})$ is log big, and
\item
the stable base locus of $K_{\bar{X}^{(j)}}+\bar{\Delta}^{(j)}$ does not contain the image of any prime divisor $P$ over $\bar{X}^{(j)}$ satisfying $a(P,\bar{X}^{(j)},\bar{\Delta}^{(j)})<0$. 
\end{itemize}
Then there is a sequence of MMP steps for $(X,\Delta)$
$$(X,\Delta)\dashrightarrow \cdots \dashrightarrow (X_{i},\Delta_{i})\dashrightarrow \cdots \dashrightarrow (X_{m},\Delta_{m})$$
such that $K_{X_{m}}+\Delta_{m}$ is semi-ample. 
Furthermore, there is a projective slc pair $(Y,\Delta_{Y})$ and a birational morphism $\pi\colon X_{m}\to Y$ such that $\pi$ is an isomorphism over the generic points of codimension one singular locus, $K_{X_{m}}+\Delta_{m}=\pi^{*}(K_{Y}+\Delta_{Y})$, and $K_{Y}+\Delta_{Y}$ is ample.   
\end{thm}

\begin{proof}
The statement on MMP steps for $(X,\Delta)$ follows from Theorem \ref{thm--mmpslclogabund} since the property of being log big for $(\bar{X}^{(j)},\bar{\Delta}^{(j)})$ implies the property of being log abundant and the pseudo-effectivity of $K_{\bar{X}^{(j)}}+\bar{\Delta}^{(j)}$. 
Therefore, it is sufficient to prove the existence of $(X_{m},\Delta_{m})\to (Y,\Delta_{Y})$ as in Theorem \ref{thm--mmpslclogbig}. 

\begin{step4}\label{step1--slc}
In this step, we prove that we may assume that $(X_{m},\Delta_{m})=(X,\Delta)$ without loss of generality. 

Let $(\bar{X}_{m},\bar{\Delta}_{m})\to (X_{m},\Delta_{m})$ be the normalization. 
We prove that $(\bar{X}^{(j)}_{m},\bar{\Delta}^{(j)}_{m})$ is log big for all irreducible components $(\bar{X}^{(j)}_{m},\bar{\Delta}^{(j)}_{m})$ of $(\bar{X}_{m},\bar{\Delta}_{m})$. 
Fix an irreducible component $(\bar{X}^{(j)}_{m},\bar{\Delta}^{(j)}_{m})$. 
Since $(X,\Delta)\dashrightarrow (X_{m},\Delta_{m})$ is a sequence of MMP steps, by \cite[Lemma 12]{ambrokollar}, we can find an irreducible component $(\bar{X}^{(j)},\bar{\Delta}^{(j)})$ of the normalization $(\bar{X},\bar{\Delta})\to (X,\Delta)$ and the induced birational map 
$$(\bar{X}^{(j)},\bar{\Delta}^{(j)}) \dashrightarrow (\bar{X}^{(j)}_{m},\bar{\Delta}^{(j)}_{m})$$
that is a sequence of MMP steps of the projective lc pair $(\bar{X}^{(j)},\bar{\Delta}^{(j)})$. 
We pick any lc center $S_{m}$ of $(\bar{X}^{(j)}_{m},\bar{\Delta}^{(j)}_{m})$ with the normalization $S^{\nu}_{m}$. 
Then there is an lc center $S$ of $(\bar{X}^{(j)},\bar{\Delta}^{(j)})$ with the normalization $S^{\nu}$ such that $\bar{X}^{(j)} \dashrightarrow \bar{X}^{(j)}_{m}$ induces a birational map $S^{\nu}\dashrightarrow S^{\nu}_{m}$. 
Let $f_{m}\colon (W_{m},\Delta_{W_{m}})\to (\bar{X}^{(j)}_{m},\bar{\Delta}^{(j)}_{m})$ be a dlt blow-up such that there is a component $T_{m}$ of $\llcorner \Delta_{W_{m}}\lrcorner$ with a surjective morphism $f_{T_{m}}\colon T_{m} \to S^{\nu}_{m}$ induced by $f_{m}$. 
Then, we can construct a dlt blow-up $f\colon (W,\Delta_{W})\to (\bar{X}^{(j)},\bar{\Delta}^{(j)})$ with a component $T$ of $\llcorner \Delta_{W} \lrcorner$ such that $f$ induces a surjective morphism $f_{T}\colon T \to S^{\nu}$, the induced birational map $W\dashrightarrow W_{m}$ is a birational contraction which is an isomorphism near the generic point of $T$, and the birational transform of $T$ on $W_{m}$ is $T_{m}$. 
$$
\xymatrix
{
(W,\Delta_{W})\ar[d]_{f}\ar@{-->}[rr]&&(W_{m},\Delta_{W_{m}})\ar[d]^{f_{m}}\\
(\bar{X}^{(j)},\bar{\Delta}^{(j)})\ar@{-->}[rr]&&(\bar{X}^{(j)}_{m},\bar{\Delta}^{(j)}_{m})
}
\qquad \qquad 
\xymatrix
{
T\ar[d]_{f_{T}}\ar@{-->}[rr]&&T_{m}\ar[d]^{f_{T_{m}}}\\
S^{\nu}\ar@{-->}[rr]&&S^{\nu}_{m}
}
$$

We define the dlt pair $(T,\Delta_{T})$ by adjunction $K_{T}+\Delta_{T}=(K_{W}+\Delta_{W})|_{T}$. 
Similarly, we define the dlt pair $(T_{m},\Delta_{T_{m}})$ by adjunction $K_{T_{m}}+\Delta_{T_{m}}=(K_{W_{m}}+\Delta_{W_{m}})|_{T_{m}}$. 
Then
$$K_{T}+\Delta_{T}\sim_{\mathbb{Q}}f^{*}_{T}(K_{\bar{X}^{(j)}}+\bar{\Delta}^{(j)})|_{S^{\nu}}\qquad {\rm and} \qquad K_{T_{m}}+\Delta_{T_{m}}\sim_{\mathbb{Q}}f^{*}_{T_{m}}(K_{\bar{X}^{(j)}_{m}}+\bar{\Delta}^{(j)}_{m})|_{S^{\nu}_{m}}.$$
By the negativity lemma (see also \cite[Lemma 4.2.10]{fujino-sp-ter}), we have 
$$a(P,T,\Delta_{T})\leq a(P,T_{m},\Delta_{T_{m}})$$
 for any prime divisor $P$ over $T$. 
On the other hand, by construction of the birational map $(W,\Delta_{W}) \dashrightarrow (W_{m},\Delta_{W_{m}})$ and the second condition of Theorem \ref{thm--mmpslclogbig}, we can apply Lemma \ref{lem--discre-relation} to $(W,\Delta_{W}) \dashrightarrow (W_{m},\Delta_{W_{m}})$, $T$, and $T_{m}$. 
By Lemma \ref{lem--discre-relation}, we have 
$$a(Q,T_{m},\Delta_{T_{m}})\leq a(Q,T,\Delta_{T})$$
for any prime divisor $Q$ on $T_{m}$. 
Take a common resolution $g\colon \widetilde{T}\to T$ and $g_{m}\colon \widetilde{T}\to T_{m}$ of the birational map $T\dashrightarrow T_{m}$. 
By comparing $g^{*}(K_{T}+\Delta_{T})$ and $g^{*}_{m}(K_{T_{m}}+\Delta_{T_{m}})$ with the aid of the above relations on discrepancies, 
we may find an effective $g_{m}$-exceptional $\mathbb{R}$-divisor $E$ such that 
$$g^{*}(K_{T}+\Delta_{T})=g^{*}_{m}(K_{T_{m}}+\Delta_{T_{m}})+E.$$ 
By Remark \ref{rem--div} (\ref{rem--div-(2)}), we have
\begin{equation*}
\begin{split}
\kappa_{\iota}(S^{\nu}_{m}, (K_{\bar{X}^{(j)}_{m}}+\bar{\Delta}^{(j)}_{m})|_{S^{\nu}_{m}})&=
\kappa_{\iota}(T_{m}, K_{T_{m}}+\Delta_{T_{m}})=\kappa_{\iota}(T, K_{T}+\Delta_{T})\\
&=\kappa_{\iota}(S^{\nu}, (K_{\bar{X}^{(j)}}+\bar{\Delta}^{(j)})|_{S^{\nu}}).
\end{split}
\end{equation*}
Since $(\bar{X}^{(j)},\bar{\Delta}^{(j)})$ is log big, which is the first condition of Theorem \ref{thm--mmpslclogbig}, we have 
$$\kappa_{\iota}(S^{\nu}, (K_{\bar{X}^{(j)}}+\bar{\Delta}^{(j)})|_{S^{\nu}})={\rm dim}\,S^{\nu}.$$
Since ${\rm dim}\,S^{\nu}={\rm dim}\,S^{\nu}_{m}$, we have
$$\kappa_{\iota}(S^{\nu}_{m}, (K_{\bar{X}^{(j)}_{m}}+\bar{\Delta}^{(j)}_{m})|_{S^{\nu}_{m}})={\rm dim}\,S^{\nu}_{m}.$$ 
Therefore $\bigl(K_{\bar{X}^{(j)}_{m}}+\bar{\Delta}^{(j)}_{m}\bigr)|_{S^{\nu}_{m}}$ is big. 
Since $S_{m}$ is an arbitrary lc center of $(\bar{X}^{(j)}_{m},\bar{\Delta}^{(j)}_{m})$, it follows that $(\bar{X}^{(j)}_{m},\bar{\Delta}^{(j)}_{m})$ is log big. 

For any irreducible component $(\bar{X}^{(j)}_{m},\bar{\Delta}^{(j)}_{m})$ of $(\bar{X}_{m},\bar{\Delta}_{m})$, the divisor $K_{\bar{X}^{(j)}_{m}}+\bar{\Delta}^{(j)}_{m}$ is semi-ample.
Thus, $(X_{m},\Delta_{m})$ satisfies the condition of Theorem \ref{thm--mmpslclogbig}. 
Therefore, we may assume $(X_{m},\Delta_{m})=(X,\Delta)$ without loss of generality. 
In particular, we may assume that $K_{X}+\Delta$ is semi-ample. 
\end{step4}

\begin{step4}\label{step2--slc}
In this step, we construct a projective slc pair $(Y,\Delta_{Y})$ by the gluing theory of Koll\'ar. 

Let $\nu \colon (\bar{X},\bar{\Delta})\to (X,\Delta)$ be the normalization, where $K_{\bar{X}}+\bar{\Delta}=\nu^{*}(K_{X}+\Delta)$. 
We may write $\bar{\Delta}=\bar{D}+\bar{G}$, where $\bar{D}$ is the conductor and $\bar{G}:=\bar{\Delta}-\bar{D}$. 
Let $\bar{D}^{\nu} \to \bar{D}$ be the normalization, and we define $G_{\bar{D}^{\nu}}$ by the divisorial adjunction (\cite[Definition 4.2]{kollar-mmp}). 
Note that we have $K_{\bar{D}^{\nu}}+G_{\bar{D}^{\nu}}=(K_{\bar{X}}+\bar{D}+\bar{G})|_{\bar{D}^{\nu}}$. 
Since $(X,\Delta)$ is slc, as in \cite[5.2]{kollar-mmp}, there is an involution $\tau$ of $\bar{D}^{\nu}$, that is, an automorphism such that $\tau^{2}$ is the identity, such that $\tau_{*}G_{\bar{D}^{\nu}}=G_{\bar{D}^{\nu}}$. 

Let $\amalg_{j} (\bar{X}^{(j)},\bar{\Delta}^{(j)})$ be the decomposition of $(\bar{X},\bar{\Delta})$ into the irreducible components. 
We put $\bar{D}_{j}=\bar{D}|_{\bar{X}^{(j)}}$ and $\bar{G}_{j}=\bar{G}|_{\bar{X}^{(j)}}$. 
Since $(\bar{X}^{(j)},\bar{\Delta}^{(j)})$ is log big, for each index $j$, $K_{\bar{X}^{(j)}}+\bar{D}_{j}+\bar{G}_{j}$ defines a birational morphism $\pi_{j}\colon \bar{X}^{(j)}\to Y_{j}$ to a normal projective variety $Y_{j}$ such that $K_{Y_{j}}+D_{Y_{j}}+G_{Y_{j}}$ is an ample $\mathbb{Q}$-divisor, where $D_{Y_{j}}=\pi_{j*}\bar{D}_{j}$ and $G_{Y_{j}}=\pi_{j*}\bar{G}_{j}$. 
Since $(\bar{X}^{(j)},\bar{\Delta}^{(j)})$ is log big, the pullback of $K_{\bar{X}^{(j)}}+\bar{\Delta}^{(j)}$ to each component of $\bar{D}^{\nu}$ is big. 
This implies that $\pi_{j}$ does not contract any component of $\bar{D}_{j}$. 
Denoting by $D^{\nu}_{Y_{j}}$ as the normalization of $D_{Y_{j}}$, the morphism $\amalg_{j}\pi_{j}\colon \bar{X}=\amalg_{j} \bar{X}^{(j)} \to \amalg_{j}Y_{j}$ induces a birational morphism $\bar{D}^{\nu} \to \amalg_{j}D^{\nu}_{Y_{j}}$. 
$$
\xymatrix
{
\bar{D}^{\nu}\ar[d]\ar[rr]&&\bar{X}\ar[d]^{\amalg_{j}\pi_{j}}\\
\amalg_{j}D^{\nu}_{Y_{j}}\ar[rr]&&\amalg_{j}Y_{j}
}
$$
We recall that $G_{\bar{D}^{\nu}}$ is defined by the divisorial adjunction and $\bar{D}^{\nu}$ has an involution $\tau$ such that $\tau_{*}G_{\bar{D}^{\nu}}=G_{\bar{D}^{\nu}}$. 
Let $G_{\amalg_{j}D^{\nu}_{Y_{j}}}$ be the birational transform of $G_{\bar{D}^{\nu}}$ on  $\amalg_{j}D^{\nu}_{Y_{j}}$. 
Note that $G_{\amalg_{j}D^{\nu}_{Y_{j}}}$ is a $\mathbb{Q}$-divisor on $\amalg_{j}D^{\nu}_{Y_{j}}$. 
By definition of the divisorial adjunction, it follows that $G_{\amalg_{j}D^{\nu}_{Y_{j}}}$ is equal to the $\mathbb{Q}$-divisor defined by applying the divisorial adjunction to $\amalg_{j}(Y_{j},D_{Y_{j}}+G_{Y_{j}})$ and $\amalg_{j}D^{\nu}_{Y_{j}}$. 
Thus, the divisor 
$$K_{\amalg_{j}D^{\nu}_{Y_{j}}}+G_{\amalg_{j}D^{\nu}_{Y_{j}}}=\Bigl(\sum_{j}(K_{Y_{j}}+D_{Y_{j}}+G_{Y_{j}})\Bigr)\Big{|}_{\amalg_{j}D^{\nu}_{Y_{j}}}$$
is ample. 
Moreover, since the morphism $\bar{D}^{\nu} \to \amalg_{j}D^{\nu}_{Y_{j}}$ is birational,  the involution $\tau$ of $\bar{D}^{\nu}$ induces a small birational map $\tau'\colon \amalg_{j}D^{\nu}_{Y_{j}} \dashrightarrow \amalg_{j}D^{\nu}_{Y_{j}}$ such that $\tau'_{*}G_{\amalg_{j}D^{\nu}_{Y_{j}}}=G_{\amalg_{j}D^{\nu}_{Y_{j}}}$. 
Since $K_{\amalg_{j}D^{\nu}_{Y_{j}}}+G_{\amalg_{j}D^{\nu}_{Y_{j}}}$ is ample, it follows that $\tau'$ is an involution of $\amalg_{j}D^{\nu}_{Y_{j}}$. 

We have constructed the following objects.
\begin{itemize}
\item
A disjoint union of projective lc pairs $\amalg_{j}(Y_{j},D_{Y_{j}}+G_{Y_{j}})$ and birational morphisms $\pi_{j}\colon(\bar{X}^{(j)},\bar{D}_{j}+\bar{G}_{j})\to (Y_{j},D_{Y_{j}}+G_{Y_{j}})$ for all $j$ such that $K_{Y_{j}}+D_{Y_{j}}+G_{Y_{j}}$ is ample and $K_{\bar{X}^{(j)}}+\bar{D}_{j}+\bar{G}_{j}=\pi^{*}_{j}(K_{Y_{j}}+D_{Y_{j}}+G_{Y_{j}})$, and 
\item
An involution $\tau'$ of $\amalg_{j}D^{\nu}_{Y_{j}}$, which is induced by $\tau$, such that $\tau'_{*}G_{\amalg_{j}D^{\nu}_{Y_{j}}}=G_{\amalg_{j}D^{\nu}_{Y_{j}}}$. 
\end{itemize}
By the gluing theory of Koll\'ar (\cite[Corollary 5.37, Corollary 5.33, Theorem 5.38]{kollar-mmp}), we get a projective slc pair $(Y,\Delta_{Y})$ whose 
normalization is $\amalg_{j}(Y_{j},D_{Y_{j}}+G_{Y_{j}})$ and the conductor is $\sum_{j}D_{Y_{j}}$. 
Then, we have a morphism
$$(\bar{X},\bar{\Delta})\to (Y,\Delta_{Y})$$
which is the composition of $\bar{X}=\amalg_{j}\bar{X}^{(j)}\to \amalg_{j}Y_{j}$ and $\amalg_{j}Y_{j} \to Y$. 
By construction, $K_{Y}+\Delta_{Y}$ is ample and the involution $\tau$ of  $\bar{D}^{\nu}$ is an involution over $Y$. 
\end{step4}

\begin{step4}\label{step3--slc}
In this final step, we construct a birational morphism $(X,\Delta)\to (Y,\Delta_{Y})$ satisfying the conditions of Theorem \ref{thm--mmpslclogbig}. 

In Step \ref{step2--slc}, we constructed a morphism $(\bar{X},\bar{\Delta})\to (Y,\Delta_{Y})$ such that the involution $\tau$ of $\bar{D}^{\nu}$ is an involution over $Y$. 
Since $\tau$ is the involution defined by the normalization $(\bar{X},\bar{\Delta})\to (X,\Delta)$, the set theoretical equivalence relation defined by $\tau$ is finite. 
By applying \cite[Corollary 5.33, Theorem 5.38]{kollar-mmp} to $\tau$ over $Y$, we get a projective slc pair $(X',\Delta_{X'})$ with a morphism $X'\to Y$ such that the normalization of $X'$ is $\bar{X}$ and the conductor is $\bar{D}$. 
By \cite[Proposition 5.3]{kollar-mmp}, there is an isomorphism $X\to X'$  such that the birational transform of $\Delta$ on $X'$ is $\Delta_{X'}$. 
Let $\pi\colon X\to Y$ be the composition of $X\to X'$ and $X'\to Y$. 

Now we have a morphism
$$\pi\colon (X,\Delta) \to (Y,\Delta_{Y})$$
of projective slc pairs. 
By construction, $\pi$ is birational and an isomorphism over the generic points of codimension one singular locus. 
We also see that $K_{Y}+\Delta_{Y}$ is ample. 
Finally, we will prove $K_{X}+\Delta=\pi^{*}(K_{Y}+\Delta_{Y})$. 
We define $\Gamma$ on $X$ by $$K_{X}+\Gamma=\pi^{*}(K_{Y}+\Delta_{Y}).$$ 
We will show $\Gamma=\Delta$. 
Since $\pi$ is an isomorphism over the generic points of codimension one singular locus, the support of $\Gamma$ does not contain any codimension one singular locus. 
Thus, it is sufficient to show $\Gamma=\Delta$ outside the codimension one singular locus of $X$ and non-normal locus of $X$. 
By construction in Step \ref{step2--slc}, we have $K_{\bar{X}}+\bar{\Delta}=\nu^{*}\pi^{*}(K_{Y}+\Delta_{Y})$, where $\nu\colon \bar{X}\to X$ is the normalization. 
This implies that $\Gamma = \Delta$ on the normal locus of $X$, so it follows that $\Gamma=\Delta$. 
In this way, we have $K_{X}+\Delta=\pi^{*}(K_{Y}+\Delta_{Y})$. 
\end{step4}
As discussed above, $\pi\colon (X,\Delta) \to (Y,\Delta_{Y})$ is the desired birational morphism. 
We finish the proof. 
\end{proof}

\section{Varieties over an algebraically closed field of characteristic zero}\label{sec5}

In this section, we give a remark on generalizations of results of this paper to varieties over an algebraically closed field of characteristic zero.  

We go back to the definitions of the invariant Iitaka dimension (Definition \ref{defn--inv-iitaka-dim}) and the numerical dimension (Definition \ref{defn--num-dim}). 
Let $k$ be an algebraically closed field of characteristic zero. 
For any normal projective variety $X$ over $k$ and $\mathbb{R}$-Cartier divisor $D$ on $X$, we can define $\kappa_{\iota}(X,D)$ and $\kappa_{\sigma}(X,D)$ by the same way as in Definition \ref{defn--inv-iitaka-dim} and Definition \ref{defn--num-dim}, respectively. 
A problem comes up when we define the relative invariant Iitaka dimension and the relative numerical dimension. 
For a given projective morphism from a normal variety to a variety and an $\mathbb{R}$-Cartier divisor, the relative invariant Iitaka dimension and the relative numerical dimension are defined with a sufficiently general fiber of the Stein factorization of the morphism. 
However, when $k$ is countable, we cannot always find a sufficiently general closed point, so sufficiently general fibers do not always exist.   
Therefore, when the base field is not necessarily the complex number field $\mathbb{C}$, we define the relative invariant Iitaka dimension and the relative numerical dimension as follows. 

\begin{defn}\label{defn--dim-2}
Let $X\to Z$ be a projective morphism from a normal variety to a variety over an algebraically closed field ${{k}}$ of characteristic zero, and let $X_{\bar{\eta}}$ be the geometric generic fiber of the Stein factorization of $X\to Z$.  
Let $D$ be an $\mathbb{R}$-Cartier divisor on $X$. 
We define the {\em relative invariant Iitaka dimension} of $D$ over $Z$, denoted by $\hat{\kappa}_{\iota}(X/Z,D)$, by $\kappa_{\iota}(X_{\bar{\eta}},D|_{X_{\bar{\eta}}})$.  
Similarly, we define the {\em relative numerical dimension} of $D$ over $Z$, denoted by $\hat{\kappa}_{\sigma}(X/Z,D)$, by $\kappa_{\sigma}(X_{\bar{\eta}},D|_{X_{\bar{\eta}}})$. 
\end{defn}

\begin{defn}\label{defn--abund-2}
Let $\pi\colon X\to Z$ be a projective morphism from a normal variety to a variety over an algebraically closed field $k$ of characteristic zero. 
For any $\mathbb{R}$-Cartier divisor $D$ on $X$, we say that $D$ is $\pi$-{\em abundant} (or {\em abundant over} $Z$) if the equality $\hat{\kappa}_{\iota}(X/Z,D)=\hat{\kappa}_{\sigma}(X/Z,D)$ holds. 
\end{defn}

With Definition \ref{defn--abund-2}, for any lc pair $(X,\Delta)\to Z$ and any $\mathbb{R}$-Cartier divisor $D$ on $X$, we can define the property of being log abundant over $Z$. 

When ${{k}}=\mathbb{C}$ and the base variety of a morphism is quasi-projective, the following lemma shows that the new definitions of the relative invariant Iitaka dimension and the relative numerical dimension coincide with those of Definition \ref{defn--inv-iitaka-dim} and Definition \ref{defn--num-dim}, respectively. 

\begin{lem}\label{lem--equality}
Let $\pi\colon X\to Z$ be a projective morphism from a normal variety to a variety over $\mathbb{C}$, and let $D$ be an $\mathbb{R}$-Cartier divisor on $X$. 
If $Z$ is quasi-projective, then $\hat{\kappa}_{\iota}(X/Z,D)=\kappa_{\iota}(X/Z,D)$ and $\hat{\kappa}_{\sigma}(X/Z,D)=\kappa_{\sigma}(X/Z,D)$. 
\end{lem}

\begin{proof}
The equality $\hat{\kappa}_{\sigma}(X/Z,D)=\kappa_{\sigma}(X/Z,D)$ follows from Definition \ref{defn--num-dim}, the generic flatness, the semi-continuity of dimension of cohomology for flat coherent sheaves, and the flat base change theorem for cohomology.  
By definition, $\kappa_{\iota}(X/Z,D) \geq 0$ if and only if $D\sim_{\mathbb{R},Z}E$ for some $E\geq0$. 

We show that the latter condition is equivalent to $\hat{\kappa}_{\iota}(X/Z,D)\geq 0$. 
It is obvious that the existence of an effective $\mathbb{R}$-divisor $E$ such that $D\sim_{\mathbb{R},Z}E$ implies $\hat{\kappa}_{\iota}(X/Z,D)\geq 0$. 
Conversely, suppose that $\hat{\kappa}_{\iota}(X/Z,D)\geq 0$. 
Then, by the standard argument of convex geometry, we can find $\mathbb{Q}$-Cartier divisors $D_{1},\cdots, D_{l}$, positive real numbers $r_{1},\cdots, r_{l}$, and effective $\mathbb{Q}$-divisors $\bar{E}_{1},\cdots , \bar{E}_{l}$ on $X_{\bar{\eta}}$ such that $\sum_{i=1}^{l}r_{i}=1$, $\sum_{i=1}^{l}r_{i}D_{i}=D$, and 
$$D_{i}|_{X_{\bar{\eta}}}\sim_{\mathbb{Q}} \bar{E}_{i}$$
for each $1\leq i\leq l$ (see, for example, \cite[Proof of Lemma 2.10]{hashizumehu}).  
Since $D_{i}$ are $\mathbb{Q}$-Cartier, by the flat base change theorem for cohomology, the restriction of any $D_{i}$ to the generic fiber of $\pi$ is $\mathbb{Q}$-linearly equivalent to an effective $\mathbb{Q}$-divisor.
By the quasi-projectivity of $Z$, for every $1\leq i\leq l$ we can find $E_{i}\geq0$ such that $D_{i}\sim_{\mathbb{Q},Z}E_{i}$. 
Then
$$D=\sum_{i=1}^{l}r_{i}D_{i}\sim_{\mathbb{R},Z}\sum_{i=1}^{l}r_{i}E_{i}.$$
Since all $r_{i}$ are positive, the right hand side is effective. 
In this way, $\hat{\kappa}_{\iota}(X/Z,D)\geq 0$ implies the existence of an effective $\mathbb{R}$-divisor $E$ such that $D\sim_{\mathbb{R},Z}E$. 

The above discussion shows that $\kappa_{\iota}(X/Z,D) \geq 0$ is equivalent to $\hat{\kappa}_{\iota}(X/Z,D)\geq 0$. 
When $\kappa_{\iota}(X/Z,D) \geq 0$, since there is an effective $\mathbb{R}$-divisor $E$ on $X$ such that $D\sim_{\mathbb{R},Z}E$, the equality $\hat{\kappa}_{\iota}(X/Z,D)=\kappa_{\iota}(X/Z,D)$ directly follows. 
From these discussions, we have $\hat{\kappa}_{\sigma}(X/Z,D)=\kappa_{\sigma}(X/Z,D)$ and $\hat{\kappa}_{\iota}(X/Z,D)=\kappa_{\iota}(X/Z,D)$. 
We finish the proof. 
\end{proof}

\begin{rem}
Lemma \ref{lem--equality} for $\mathbb{Q}$-divisors is well-known. 
On the other hand, when $D$ has irrational coefficients, we may have $\kappa_{\iota}(X/Z,D) \geq 0$ though $H^{0}(F,\mathcal{O}_{F}(\llcorner mD|_{F}\lrcorner))=\{0\}$ for all $m \in \mathbb{Z}_{>0}$ and all fibers $F$ of $X\to Z$ (for example, put $X=\mathbb{P}^{1}_{\mathbb{C}}$, $Z={\rm Spec}\,\mathbb{C}$ and $D=\sqrt{2}(p_{1}-p_{2})$ for distinct closed points $p_{1}$ and $p_{2}$). 
Hence, the standard argument with $H^{0}(X_{\bar{\eta}},\mathcal{O}_{X_{\bar{\eta}}}(\llcorner mD|_{X_{\bar{\eta}}}\lrcorner))$ does not work well. 
\end{rem}

Thanks to Lemma \ref{lem--equality}, in the results of this paper we may assume that the property of being log abundant is defined with Definition \ref{defn--abund-2}. 
Then, all the results of this paper hold for varieties over an algebraically closed field $k$ of characteristic zero.


\end{document}